\documentclass[11pt]{amsart}

\usepackage{amsmath,amsthm,amssymb,amsfonts}
\usepackage{xcolor}

\usepackage{xargs}  
\usepackage[colorinlistoftodos,prependcaption,textsize=tiny,obeyFinal]{todonotes}
\newcommandx{\ebltodo}[2][1=]{\todo[linecolor=red,backgroundcolor=red!25,bordercolor=red,#1]{#2}}

\usepackage{enumitem}
\setenumerate{leftmargin=*}
\setitemize{leftmargin=*}
\usepackage{amscd}
\usepackage[latin1]{inputenc}
\DeclareMathAlphabet{\mathpzc}{OT1}{pzc}{m}{it}
\usepackage{bbm}

\numberwithin{equation}{section}

\swapnumbers
\newtheorem{thm}[subsection]{Theorem}

\newtheorem*{cor*}{Corollary}
\newtheorem{lemma}[subsection]{Lemma}
\newtheorem*{sublemma}{Sublemma}

\newtheorem{propos}[subsection]{Proposition}

\newtheorem*{thm*}{Theorem}
\newtheorem*{thma*}{Theorem A}
\newtheorem*{thmb*}{Theorem B}
\newtheorem*{thmc*}{Theorem C}

\newtheorem{conj}{Conjecture}

\newcounter{consta}

\newcounter{constk}
\renewcommand{\theconstk}{{\kappa_{\arabic{constk}}}}
\newcommand{\constk}{\refstepcounter{constk}\theconstk}

\newcounter{constc}

\newcounter{constE}
\renewcommand{\theconstE}{{{C}_{\arabic{constE}}}}
\newcommand{\constE}{\refstepcounter{constE}\theconstE}

\def\bbz{\mathbb{Z}}
\def\bbq{\mathbb{Q}}

\def\bbr{\mathbb{R}}

\def\bbc{\mathbb{C}}

\def\bbh{\mathbb{H}}
\def\bbn{\mathbb{N}}
\def\Gbf{\mathbb{G}}

\def\Scal{\mathcal{S}}

\def\hfrak{\mathfrak{h}}

\def\rfrak{\mathfrak{r}}

\def\gfrak{\mathfrak{g}}

\def\Gbf{\mathbf{G}}

\def\G{\Gbf}

\DeclareMathOperator\diff{d}

\def\vol{{\rm{vol}}}
\def\SL{{\rm{SL}}}
\def\PSL{{\rm{PSL}}}

\def\SO{{\rm{SO}}}
\def\Lie{{\rm Lie}}

\DeclareMathOperator\Ad{Ad}

\def\sl{{\mathfrak{sl}}}

\def\vare{\varepsilon}

\def\zg0{Z_{G_\omega}(s)}

\def\zg{Z_G(s)}

\def\be{\begin{equation}}
\def\ee{\end{equation}}

\def\dist{{\rm dist}}

\def\lf{\mathfrak l}

\def\rH{{}^\rho\H}

\def\diff{\operatorname{d}}

\def\dist{d}

\def\mixexp{\kappa_X}

\def\vY{{\mathsf v}}
\def\rH{T}
\def\boxH{\mathsf B^H}
\def\boxHs{\mathsf B^{H}}
\def\boxG{\mathsf B^{G}}

\def\rwm{\nu}
\def\convN{\rwm^{(n)}}
\def\convL{\rwm^{(\ell)}}
\def\rws{t}

\def\injr{\eta}
\def\dexp{\delta}
\def\rel{r}
\def\rot{\mathsf r}
\def\ave{\int_{0}^1}
\def\uvk{u_\rel}
\def\uvkd{\diff\!\rel}

\def\onst{m_0}
\def\osa{m_\alpha}
\def\mfht{f}

\def\margI{I}
\def\noI{\psi}

\def\inj{\operatorname{inj}}
\def\lf{\mathbb R}
\def\qlf{\mathbb C}

\def\rc{\mathbb F}
\def\qi{i}
\def\uni{e}
\def\nuni{e}
\def\coneH{\mathsf E}
\def\cone{\mathcal E}

\def\umt{\mathsf Q}

\DeclareMathOperator{\supp}{supp}

\newcommand{\white}{\mathcal W}
\newcommand{\black}{\mathcal B}

\newcommand{\rhsc}{b}
\newcommand{\rhsco}{b_1}
\newcommand{\sfh}{\mathsf h}
\newcommand{\sfs}{\mathsf s}
\newcommand{\cX}{\mathsf C_X}
\newcommand{\bcX}{\bar{\mathsf C}_X}

\makeatletter
\newcommand*\bigcdot{\mathpalette\bigcdot@{.5}}
\newcommand*\bigcdot@[2]{\mathbin{\vcenter{\hbox{\scalebox{#2}{$\m@th#1\bullet$}}}}}
\makeatother

\begin{document}
\title[Polynomial effective density]{Polynomial effective density\\ in quotients 
of $\mathbb H^3$ and $\mathbb H^2\times\mathbb H^2$}

\author{E.~Lindenstrauss}
\address{E.L.: The Einstein Institute of Mathematics, Edmond J.\ Safra Campus, 
Givat Ram, The Hebrew University of Jerusalem, Jerusalem, 91904, Israel}
\email{elon@math.huji.ac.il}
\thanks{E.L.\ acknowledges support by ERC 2020 grant HomDyn (grant no.~833423).}

\author{A.~Mohammadi}
\address{A.M.: Department of Mathematics, University of California, San Diego, CA 92093}
\email{ammohammadi@ucsd.edu}
\thanks{A.M.\ acknowledges support by the NSF grants DMS-1764246 and 2055122}

\begin{abstract}
We prove effective density theorems, with a polynomial error rate, for orbits of the upper triangular subgroup of $\SL_2(\lf)$ 
in arithmetic quotients of $\SL_2(\qlf)$ and $\SL_2(\lf)\times\SL_2(\lf)$. 

The proof is based on the use of a Margulis function, tools from incidence geometry, 
and the spectral gap of the ambient space. 
\end{abstract}

\maketitle

\tableofcontents

\section{Introduction}
The quantitative understanding of the behavior of orbits in homogeneous spaces is a fundamental problem. 
Let $G$ be a connected Lie group and $\Gamma\subset G$ a lattice (a discrete subgroup with finite covolume). Let $L\subset G$ be a closed connected subgroup. Ratner's celebrated resolution of Raghunathan's conjectures,~\cite{Ratner-Acta, Ratner-measure, Ratner-topological}, provides a complete classification for the closure of {\em individual} $L$-orbits in $G/\Gamma$ if $L$ is unipotent, or more generally is generated by unipotent subgroups (this is true even if $L$ is not assumed to be connected, see \cite{Shah-Gen-Uni}). Prior to Ratner's work, some important special cases of this problem were studied by Margulis~\cite{Margulis-Oppenheim}, and Dani and Margulis~\cite{DM-Oppenheim, DM-MathAnn}.

These remarkable results all share the lacuna that they are not quantitative, e.g.\ they do not provide any rate at which the orbit fills up its closure. 
Indeed Ratner's work relies on the pointwise ergodic theorem which is hard to effectivize. The work of Dani and Margulis uses minimal sets, which though formally ineffective can be effectivized with some effort;
a result in this spirit was obtained by Margulis and the first named author in \cite{LM-Oppenheim}, though the rates obtained there are of polylog form, and that too after significant effort.
With Margulis and Shah, we have obtained a general effective orbit closure theorem for unipotent orbits on arithmetic quotients, the first piece of this being \cite{LMMS} and the continuation is in preparation; however the rates obtained are even worse than \cite{LM-Oppenheim}.

When $G$ is a unipotent group, Green and Tao gave an effective equidistribution theorem for orbits of subgroups $L\subset G$ (that of course will also be unipotent) in~\cite{GrTao-Nil} with polynomial error rates. 
When $G$ is semisimple, however, 
not much seems to be known. A notable exception is the case where $L\subset G$ is a horospherical subgroups, that is to say if there is an element $a \in G$ so that 
\[
L=\{g \in G: a^n g a^{-n} \to 1 \text{ as $n\to\infty$}\},
\]
for instance if $L$ is the full group of strictly upper triangular matrices in $G=\SL_n(\bbr)$. 
In this case, the behaviour of individual orbits can be related to decay of matrix coefficients, and hence effective equidistribution with polynomial error rate can be established. The first works in this direction we are aware of by Sarnak \cite{Sarnak-Thesis}, Burger \cite{Burger-Horo}, and Kleinbock and Margulis \cite{KMnonquasi} based on Margulis' thesis, as well as the more recent papers by Flaminio and Forni \cite{FlFor}, Str\"{o}mbergsson \cite{Strom-Horo}, and Sarnak and Ubis \cite{Sarnak-Ubis}. Quantitative horospheric equidistribution has now been established in much greater generality e.g.\ by Kleinbock and Margulis in \cite{KM-Drich}, McAdam in \cite{McAdam} and by Asaf Katz \cite{Katz-Quantitative}. Moreover a quantitative equidistribution estimate twisted by a character was proved by Venkatesh \cite{Venkatesh-Sparse} and further developed by Tani and Vishe as well as Flaminio, Forni, and Tanis \cite{Tanis-Vishe, FFT-twisted}; this was generalized to a disjointness result with a general nil-system by Asaf Katz in \cite{Katz-Quantitative}.  
Closely related is the case of translates of periodic orbits of subgroups $L\subset G$ which are fixed by an involution by Duke, Rudnick and Sarnak, Eskin and McMullen, and Benoist and Oh in \cite{DRS-Counting, Eskin-McMullen-1993, BO-Counting}.

Beyond the horospherical case\footnote{Strictly speaking, the twisted horospherical averages considered in \cite{Venkatesh-Sparse,Tanis-Vishe, FFT-twisted, Katz-Quantitative} can also be considered as a non-horospherical flow on a suitable product space, though they are closely related to the horospherical case.} (and the related case of groups fixed by an involution) equidistribution results with polynomial rates were known only for skew products by Str\"{o}mbergsson \cite{Strom-Semi}, Str\"{o}mbergsson and Vishe \cite{Strombergsson-Vishe} and by Wooyeon Kim \cite{ASLn-Kim}, for random walks by automorphisms of the torus (cf.\ \cite{BFLM} by Bourgain Furman, Mozes and the first named author and subsequent works in this direction, e.g.\ \cite{He-Saxce} by He and de Saxce), and for the special case of periodic orbits of increasing volume by Einsiedler, Margulis, Venkatesh and by these three authors with the second named author \cite{EMV,EMMV}. There are also some quantitative equidistribution results for particular types of unipotent orbits, e.g.\ \cite{Chow-Yang} by Chow and Lei Yang.

In this paper, we prove an effective density theorem, with a {\em polynomial} error rate, for orbits of the upper triangular subgroup of $\SL_2(\lf)$ in arithmetic quotients of $\SL_2(\qlf)$ and $\SL_2(\lf)\times\SL_2(\lf)$. These are first results in the literature which provide a polynomial rate for general orbits in a homogeneous space of a semisimple group, beyond the aforementioned case of horospherical subgroups.

Let us now fix some notation in order to state the main theorems. Let     
\[
\text{$G=\SL_2(\qlf)\quad$ or $\quad G=\SL_2(\lf)\times\SL_2(\lf)$}.
\] 
Let $\Gamma\subset G$ be a lattice, and put $X=G/\Gamma$.

Let $d$ be the right invariant metric on $G$ which is defined using the killing form. 
This metric induces a metric $\dist_X$ on $X$, 
and natural volume forms on $X$ and its submanifolds. The injectivity radius of a point $x\in X$ may be defined using this metric. For every $\injr>0$, let  
\[
X_\injr=\{x\in X: \text{injectivity radius of $x$ is $\geq \injr$}\}.
\]

Throughout the paper, $H$ denotes $\SL_2(\lf)$ if $G=\SL_2(\qlf)$
or the diagonally embedded copy of $\SL_2(\lf)$ in $G$ if $G=\SL_2(\lf)\times\SL_2(\lf)$. That is
\[
\SL_2(\lf)\subset\SL_2(\qlf)\quad\text{or}\quad\{(g,g): g\in\SL_2(\lf)\}\subset \SL_2(\lf)\times\SL_2(\lf).
\] 
Let $P\subset H$ denote the group of upper triangular matrices in $H$.

An orbit $Hx\subset X$ is periodic if $H\cap{\rm Stab}(x)$ is a lattice in $H$. For the semisimple group $H$, the orbit $Hx$ is periodic iff it is closed. 

Let $|\;|$ denote the absolute value on $\qlf$, 
and let $\|\;\|$ denote the maximum norm on ${\rm Mat}_2(\qlf)$ or ${\rm Mat}_{2}(\lf)\times {\rm Mat}_{2}(\lf)$ with respect to the standard basis. For every $T> 0$ and every subgroup $L\subset G$, let 
\[
B_L(e,T)=\{g\in L: \|g-I\|\leq T\}.
\]

The following is the main theorem in this paper.

\begin{thm}\label{thm:main}
Assume that $\Gamma$ is an arithmetic lattice.
For every $0<\dexp<1/2$, every $x_0\in X$, and large enough $T$ (depending explicitly on $\dexp$ and the injectivity radius of $x_0$) 
at least one of the following holds.
\begin{enumerate}
\item For every $x\in X_{\rH^{-\dexp\ref{k:main-1}}}$, we have  
\[
\dist_X\Bigl(x,B_P\Big(e,\rH^A\Big).x_0\Bigr)\leq \ref{c:main-1}\rH^{-\dexp\ref{k:main-1}}.
\]
\item There exists $x'\in X$ such that $Hx'$ is periodic with $\vol(Hx')\leq\rH^{\dexp}$, and 
\[
\dist_X(x',x_0)\leq \ref{c:main-1}\rH^{-1}.
\] 
\end{enumerate} 
Where $A$, $\constk\label{k:main-1}$, and $\constE\label{c:main-1}$ are positive constants depending on $X$. 
\end{thm}

The proof of Theorem~\ref{thm:main} has a similar flavor to~\cite{GJS-SU2} by Gamburd, Jakobson, and Sarnak as well as to the work of Bourgain and Gamburd~\cite{BG-SU2, BG-SL2} and the aforementioned work of Bourgain, Furman, Lindenstrauss, and Mozes~\cite{BFLM}. 
Indeed in the first step, we use a Diophantine condition to produce some dimension  at a certain scale ({\em initial dimension}). In the second step, we use a Margulis function to show that by passing to a larger scale and translating $B_P(e,T^\delta).x_0$ with a random element of controlled size, we obtain a set with {\em large dimension}. Margulis functions were introduced in the context of homogeneous dynamics in \cite{EMM-Up} by Eskin, Margulis, and Mozes, and have become an indispensable tool in homogeneous dynamics and beyond. 

We then use a projection theorem to move this additional dimension to the direction of a horospherical subgroup of 
$G$. The projection theorem we use is an adaptation of the work of K\"{a}enm\"{a}ki, Orponen, and Venieri~\cite{kenmki2017marstrandtype} and is based on the works of Wolff and Schlag~\cite{Wolff, Schlag}. Finally, we use an argument due to Venkatesh~\cite{Venkatesh-Sparse} to conclude the proof. 

\subsection*{The main proposition}\label{sec:main-step}
Let $U\subset N$ denote the group of upper triangular unipotent matrices in $H\subset G$, respectively. 

More explicitly, if $G=\SL_2(\qlf)$, then  
\[
N=\left\{n(\rel,s)=\begin{pmatrix} 1 & \rel+is\\ 0 &1\end{pmatrix}: (r,s)\in\lf^2\right\}
\] 
and $U=\{n(\rel,0):\rel\in \lf\}$; we will often denote the elements in $U$ by $u_\rel$, i.e., $n(\rel,0)$ will often be denoted by $u_\rel$ for $\rel\in\bbr$. Let 
\[
V=\{n(0, s)=v_s: s\in \lf\}.
\]

If $G=\SL_2(\lf)\times \SL_2(\lf)$, then  
\[
N=\left\{n(\rel,s)=\left(\begin{pmatrix} 1 & \rel+s\\ 0 &1\end{pmatrix},\begin{pmatrix} 1 & \rel\\ 0 &1\end{pmatrix}\right): (r,s)\in\lf^2\right\}
\] 
and $U=\{n(\rel,0):\rel\in \lf\}$.
As before, $n(\rel,0)$ will be denoted by $u_\rel$ for $\rel\in\bbr$. Let $V=\{n(0,s)=v_s: s\in\lf\}$. In both cases, we have $N=UV$.

\medskip

The following proposition is a crucial step in the proof. Roughly speaking, it states that for every $x_0\in X$, we can find a subset of $V$ with dimension almost $1$ near $P.x_0$ unless $x_0$ is extremely close to a periodic $H$-orbit with small volume.  

\begin{propos}[Main Proposition]\label{prop:main-prop}\label{prop:main}\label{prop:dim-1-C}
There exists some $\eta_0>0$ depending on $X$ with the following property.

Let $0<\theta,\dexp<1/2$, $0<\eta<\eta_0$, and $x_0\in X$. 
There are $\constk\label{k:main-prop}$ and $A'$, depending on $\theta$,
and $T_1$ depending on $\dexp$, $\eta$, and the injectivity radius of $x_0$, so that for all $T>T_1$ at least one of the following holds. 
 
\begin{enumerate}
\item There exists a finite subset $I\subset [0,1]$ 
so that both of the following are satisfied. 
\begin{enumerate}
\item The set $I$ supports a probability measure $\rho$ which satisfies 
\[
\rho(J)\leq C_{\theta}|J|^{1-\theta}
\]
for every interval $J$ with $|J|\geq \rH^{-\dexp\ref{k:main-prop}}$ where $C_\theta\geq 1$ depends on $\theta$. 
\item There is a point $y_0\in X_{\eta}$ so that 
\[
d_X\Big(v_s.y_0,B_P\Big(e,\rH^{A'}\Big).x_0\Bigr)\leq  \ref{c:main-2}\rH^{-\dexp\ref{k:main-prop}}
\] 
for all $s\in I\cup\{0\}$. 
\end{enumerate}
\item There exists $x'\in X$ so that $Hx'$ is periodic with $\vol(Hx')\leq\rH^{\dexp}$ and 
\[
\dist_X(x',x_0)\leq \ref{c:main-2}\rH^{-1}.
\]
\end{enumerate}
Where $\constE\label{c:main-2}$ depends on $X$.
\end{propos}

The proof of this proposition will be completed in \S\ref{sec:proof-main-prop}; it involves three main steps, which we now outline.

\begin{enumerate}
\item Let us assume that the injectivity radius of $x_0$ is bounded below by some constant depending on $X$; 
we can always reduce to this case using certain non-divergence results which are discussed in~\S\ref{sec:non-div}. 

Since we are interested in information about how points approach each other transversal to $H$, we will work with a thickening of $P.x_0$ with $\mathsf B^H$, a {\em small} neighborhood of the identity in $H$. In the first step, we use Proposition~\ref{prop:closing-lemma} (a closing lemma) 
to show that either Proposition~\ref{prop:main-prop}(2) holds, 
or we can find some $x\in\Bigl(\mathsf B^H\cdot B_P\Bigl(e,T^{O(\delta)}\Bigr)\Bigr).x_0$, whose injectivity radius is bounded below depending on $X$, so that any two 
nearby points in 
$\Bigl(\mathsf B^H\cdot B_P\Bigl(e,T^{\delta}\Bigr)\Bigr).x$ 
have distance $>T^{-1}$ transversal to $H$.   
\item Assuming Proposition~\ref{prop:main-prop}(2) does not hold, in the second step, we use a Margulis function to show that translations of the aforementioned thickening of $B_P\Bigl(e,T^{\delta}\Bigr).x$ by certain random elements in $B_P\Bigl(e, T^{O_\theta(1)}\Bigr)$ have dimension $1-\theta$ transversal to $H$ at scale $T^{-0.1\delta}$. This step is carried out in \S\ref{sec:Marg-func-ini-dim}.

The random elements we use in this step further have the property that translations of $\Bigl(\mathsf B^H\cdot B_P\Bigl(e,T^{\delta}\Bigr)\Bigr).x$ with them stay near $P.x$ --- this property is reminiscent of Margulis' thickening technique, albeit unlike the latter we only thicken in $H$ and not in $G$.

\item In the third step, we use a projection theorem (Theorem~\ref{thm:proj-thm}) combined with some arguments in homogeneous dynamics, to project the aforementioned entropy to the direction of $N$. This is the content of \S\ref{sec:Mars-proj}.
\end{enumerate}

Let us now elaborate on how Proposition~\ref{prop:main-prop} may be used to complete the proof of Theorem~\ref{thm:main}. 

The argument is based on the quantitative decay of correlations for the ambient space $X$: 
There exists $\mixexp>0$ so that  
\be\label{eq:actual-mixing-intro}
\biggl|\int \varphi(gx)\psi(x)\diff\!{m_X}-\int\varphi\diff\!{m_X}\int\psi\diff\!{m_X}\biggr|\ll_{\varphi,\psi} \nuni^{-\mixexp d(e,g)}
\ee
for all $\varphi,\psi\in C^\infty_c(X)+\bbc\cdot 1$, where $m_X$ is the probability Haar measure on $X$ and $d$ is our fixed right $G$-invariant metric on $G$. See e.g.~\cite[\S2.4]{KMnonquasi} and references there for~\eqref{eq:actual-mixing-intro}; we note that $\mixexp$ is absolute if $\Gamma$ is a congruence subgroup, see~\cite{Burger-Sarnak, Cl-tau, Gor-Mau-Oh}.

As it is well studied,~\eqref{eq:actual-mixing-intro} implies quantitative equidistribution results for expanding pieces of the horospherical group $N$ in $X$. 
Note, however, that we are only supplied with the set 
\[
B=\{u_\rel v_s:\rel\in[0,1], s\in I\}
\]
where $I$ is as in Proposition~\ref{prop:main-prop}, i.e., 
we do not have the luxury of using an open subset of $N$.
To remedy this issue, we use an argument due to Venkatesh~\cite{Venkatesh-Sparse} 
and show that so long as $\theta$ is small enough --- this is quantified using~\eqref{eq:actual-mixing-intro} ---   expanding translations of $B$ are already equidistributed in $X$, see Proposition~\ref{prop:1-epsilon-N}.    

\subsection*{Periodic orbits}
The techniques we develop here allow us to prove an effective density theorem for periodic orbits of $H$ as well. We will show in Lemma~\ref{lem:non-div-closed} that there exists some $\eta_X>0$ 
so that for every periodic orbit $Y$, we have 
\be\label{eq: muY Xeta intro}
\mu_Y(X_{\eta_X})\geq 0.9
\ee
where $\mu_Y$ denotes the $H$-invariant probability measure on $Y$.

\begin{thm}\label{thm:main-closed}
Let $Y\subset X$ be a periodic $H$-orbit in $X$. 
Then for every $x\in X_{\vol(Y)^{-\ref{k:periodic}}}$ we have 
\[
\dist_X(x,Y)\leq \ref{c:periodic}\vol(Y)^{-\ref{k:periodic}}.
\]
Where $\constk\label{k:periodic}\geq\mixexp^4/L$ (for an absolute constant $L$) and 
$\constE\label{c:periodic}$ depends explicitly on $\mixexp$, $\vol(X)$, and the minimum of the injectivity radius of points in $X_{\eta_X}$, 
see~\eqref{eq: def C3 and k3}. 
If $\Gamma$ is congruence, $\ref{k:periodic}$ is absolute.
\end{thm}

If $\Gamma$ is an arithmetic lattice, Theorem~\ref{thm:main-closed} is a rather special case of a theorem of Einsiedler, Margulis, and Venkatesh~\cite{EMV} or  (when  the  corresponding $\bbq$-group has over $\bbr$ compact factors) the followup work by Einsiedler, Margulis, and Venkatesh and the second named author~\cite{EMMV}. Note however that Theorem~\ref{thm:main-closed} does {\em not} require $\Gamma$ to be arithmetic. 
In particular, unlike~\cite{EMV, EMMV}, our argument does not rely on property $(\tau)$.

By the arithmeticity theorems of Selberg and Margulis, irreducible lattices in $\SL_2(\bbr)\times \SL_2(\bbr)$ are arithmetic.
Regarding reducible quotients of $\SL_2(\bbr)\times \SL_2(\bbr)$, if such a quotient  $\SL_2(\bbr)\times\SL_2(\bbr)/\Gamma_1\times\Gamma_2$ contains infinitely many closed orbits of $H$, then $\Gamma_2$ is commensurable to $\Gamma_1$ (up to a conjugation) and moreover $\Gamma_1$ has infinite index in its commensurator. By a theorem of Margulis, it follows that $\Gamma_1$ is arithmetic, see~\cite[Ch.~IX]{Margulis-Book}. 
Moreover, it was recently shown,~\cite{MM, BFMS}, that if $\SL_2(\bbc)/\Gamma$ contains infinitely many closed orbits of $H$, then $\Gamma$ is arithmetic. 

Thus in all cases covered by Theorem~\ref{thm:main-closed}, either $\Gamma$ is arithmetic hence \cite{EMV, EMMV} apply (though the proof we give here is very different) or there are only finitely many closed $H$-orbits.
The key point of Theorem~\ref{thm:main-closed} is that the rate of equidistribution depends only on rather coarse properties of $X$ 
namely the rate of mixing $\mixexp$, the volume of $X$, and the injectivity radius of the compact core of $X$, suitably interpreted. This can be used in some special cases to give an effective version of the finiteness theorems of~\cite{MM, BFMS}, as we discuss in the next subsection. It is interesting to note that the proofs in~\cite{MM, BFMS} rely on equidistribution results~\cite{Mozes-Shah} which are in the spirit of Theorem~\ref{thm:main-closed}, albeit in a qualitative form. 

\subsection*{Totally geodesic planes in hybrid manifolds}
Gromov and Piatetski-Shapiro~\cite{GrPS-Non-arith} constructed examples of non-arithmetic hyperbolic manifolds by gluing together pieces of non-commensurable arithmetic manifolds. 
Let $\Gamma_1$ and $\Gamma_2$ be two torsion free lattices in ${\rm Isom}(\bbh^3)$ --- recall that ${\rm Isom}(\bbh^3)$ is an index 2 subgroup of ${\rm O}(3,1)$ and that $\SL_2(\bbc)$ is locally isomorphic to ${\rm O}(3,1)$. Let $M_i=\bbh^3/\Gamma_i$. Assume further that for $i=1,2$, there exists $3$-dimensional submanifolds  with boundary $N_i\subset M_i$ so that 
\begin{itemize}
    \item The Zariski closure of $\pi_1(N_i)\subset \Gamma_i$ contains ${\rm O}(3,1)^\circ$ where ${\rm O}(3,1)^\circ$ is the connected component of the identity in ${\rm O}(3,1)$. 
    \item Every connected component of $\partial N_i$ is a totally geodesic embedded surface in $M_i$ which separates $M_i$. 
    \item $\partial N_1$ and $\partial N_2$ are isometric. 
\end{itemize}
Let $M$ be the manifold obtained by gluing $N_1$ and $N_2$ using the isometry between $\partial N_1$ and $\partial N_2$. Then $M$ carries a complete hyperbolic metric, thus, we consider $\pi_1(M)$ as a lattice in ${\rm O}(3,1)$. 
Let $\Gamma'=\pi_1(M)\cap{\rm O}(3,1)^\circ$, and let $\Gamma$ denote the inverse image of $\Gamma'$ in $G=\SL_2(\bbc)$. If $\Gamma_1$ and $\Gamma_2$ are arithmetic and non-commensurable, then $M$ is non-arithmetic, i.e., $\Gamma$ is a non-arithmetic lattice in $G$. 
A totally geodesic plane in $M$ lifts to a periodic orbit of $H=\SL_2(\bbr)$ in $X=G/\Gamma$.

The following finiteness theorem, in qualitative form, was proved by Fisher, Lafont, Miller, and Stover~\cite[Thm.~1.4]{FLMS}, see also~\cite[\S12]{BO-GF}.

\begin{thm}\label{thm:finiteness-intro}
Let $M$ be a hyperbolic $3$-manifold obtained by gluing the pieces $N_1$ and $N_2$ from non-commensurable arithmetic manifolds along $\Sigma=\partial N_1=\partial N_2$ as described above. 
The number of totally geodesic planes in $M$ is at most 
\[
L\biggl({\rm area}(\Sigma)\vol(X)\eta_X^{-1}\mixexp^{-1}\biggr)^{{L}/{\mixexp^4}}
\]
where $L$ is absolute and $X=G/\Gamma$ is as above. 
\end{thm}

\subsection*{Acknowledgment}
We would like to thank the Hausdorff Institute for its hospitality during the winter of 2020.
A.M.\ would like to thank the Institute for Advanced Study for its hospitality during the fall of 2019 where parts of this project were carried out. The authors would like to thank Gregory Margulis and Nimish Shah for many discussions about effective density, and Joshua Zahl for helpful communications regarding projections theorems. We would also like to thank Zhiren Wang with whom we discussed related questions. We thank the anonymous referees for their helpful comments.   

\section{Notation and preliminaries}\label{sec:notation}

Throughout the paper     
\[
\text{$G=\SL_2(\qlf)\quad$ or $\quad G=\SL_2(\lf)\times\SL_2(\lf)$}.
\] 
Let $\Gamma\subset G$ be a lattice, and put $X=G/\Gamma$. 

We define the subgroups $H$, $N$, $U$, and $V$ as in the  introduction.

Also let $U^-=\{u^-_\rel:\rel\in\lf\}$ denote the group of lower triangular unipotent matrices in $H$.

For every $t\in\bbr$, let $a_t$ denote the images of 
\be\label{eq:def-a-ell}
\begin{pmatrix} \uni^{t/2} & 0 \\ 0 & \uni^{-t/2} \end{pmatrix}
\ee
in $H$. Note that $a_t n(r,s)a_{-t}=n(\uni^{t}(r,s))$ for all $t\in\bbr$ and all $(r,s)\in \lf^2$.

\subsection*{Lie algebras and norms}
Let $|\;|$ denote the usual absolute value on $\qlf$ (and on $\lf$).  
Let $\|\;\|$ denotes the maximum norm on ${\rm Mat}_2(\qlf)$ and ${\rm Mat}_2(\lf)\times {\rm Mat}_2(\lf)$,
with respect to the standard basis.

Let $\gfrak=\Lie(G)$, that is, $\gfrak=\mathfrak {sl}_2(\qlf)$ or $\gfrak=\mathfrak{sl}_2(\lf)\oplus \mathfrak{sl}_2(\lf)$.
We write $\mathfrak g=\mathfrak h\oplus \mathfrak r$ where $\hfrak=\Lie(H)\simeq\mathfrak{sl}_2(\bbr)$,
$\rfrak=\qi\mathfrak{sl}_2(\bbr)$ if $\gfrak=\mathfrak{sl}_2(\qlf)$ and 
$\rfrak=\mathfrak{sl}_2(\bbr)\oplus\{0\}$
if $\gfrak=\mathfrak{sl}_2(\bbr)\oplus \mathfrak{sl}_2(\bbr)$.

Throughout the paper, we will use the uniform notation
\[
w=\begin{pmatrix}
w_{11} & w_{12}\\
w_{21} & w_{22}
\end{pmatrix}
\]
for elements $w\in\rfrak$, where $w_{ij}\in i\bbr$ if $G=\SL_2(\bbc)$ and $w_{ij}\in\bbr$ if $G=\SL_2(\bbr)\times\SL_2(\bbr)$.

Note that $\rfrak$ is a {\em Lie algebra} in the case $G=\SL_2(\bbr)\times\SL_2(\bbr)$, but not when $G=\SL_2(\bbc)$.

We fix a norm on $\hfrak$ by taking the maximum norm where the coordinates are given by $\Lie(U)$, $\Lie(U^-)$, and $\Lie(A)$; similarly fix a norm on $\rfrak$. 
By taking maximum of these two norms we get a norm on $\gfrak$. These norms will also be denoted by $\|\;\|$.

Let $\constE\label{E:2ball}\label{E:dist-sheet}\geq 1$ be so that 
\be\label{eq:2ball}
\text{$\|hw\|\leq \ref{E:2ball}\|w\|$ for all $\|h-I\|\leq 2$ and all $w\in \gfrak$.}  
\ee

For all ${\beta}>0$, we define
\be\label{eq:def-BoxH}
\boxH_{\beta}:=\{u_s^-:|s|\leq {\beta}\}\cdot\{a_t: |t|\leq \beta\}\cdot\{u_\rel:|\rel|\leq {\beta}\}
\ee 
for all $0<{\beta}<1$.  Note that for all $h_i\in(\boxH_{\beta})^{\pm1}$, $i=1,\ldots,5$, we have  
\be\label{eq:B-beta-almost-group}
 h_1\cdots h_5\in \boxH_{100\beta}.
\ee

We also define $\boxG_{\beta}:=\boxH_\beta\cdot\exp(B_\rfrak(0,\beta))$ where $B_\rfrak(0,{\beta})$ denotes the ball of radius $\beta$ in $\rfrak$ with respect to $\|\;\|$.

We deviate slightly from the notation in the introduction, and 
define the injectivity radius of $x\in X$ using $\boxG_{\beta}$ instead of the metric $d$ on $G$. 
Put
\be\label{eq:def-inj}
\inj(x)=\min\Big\{0.01, \sup\Big\{\beta: \text{ $g\mapsto gx$ is injective on $\boxG_{10\beta}$}\Big\}\Big\}.
\ee
Taking a further minimum if necessary, we always assume that the injectivity radius of $x$ defined using the metric $d$ dominates $\inj(x)$.  

For every $\injr>0$, let  
\[
X_\injr=\Bigl\{x\in X: \inj(x)\geq \injr\Bigr\}.
\]

\subsection*{Constants and the $\star$-notation}
In our analysis, the dependence of the exponents on $\Gamma$ are via the application of results in \S\ref{sec:Exp-Mixing}, see~\eqref{eq:actual-mixing}, and~\S\ref{sec:closing-lemma}.

We will use the notation $A\asymp B$ when the ratio between the two lies in $[C^{-1}, C]$
for some constant $C\ge 1$ which depends at most on $G$ and $\Gamma$ in general. 
We write $A\ll B^\star$ (resp.\ $A\ll B$) to mean that $A\le C B^\kappa$ (resp.\ $A\leq CB$) 
for some constant $C>0$ depending on $G$ and $\Gamma$, and $\kappa>0$ which follows the above convention about exponents.

\begin{lemma}\label{lem:BCH}
There exist absolute constants $\beta_0$ and $\constE\label{E:BCH}\geq1$ so that the following holds.  
Let $0<\beta\leq \beta_0$, and let $w_1,w_2\in B_\rfrak(0,\beta)$. There are $h\in H$ and $w\in\rfrak$ which satisfy 
\[
\text{$0.5\|w_1-w_2\|\leq \|w\|\leq 2\|w_1-w_2\|\quad$ and $\quad\|h-I\|\leq \ref{E:BCH}\beta\|w\|$}
\] 
so that $\exp(w_1)\exp(-w_2)=h\exp(w)$.

\end{lemma}

\begin{proof}
Using the Baker--Campbell--Hausdorff formula, we have 
\[
\exp(w_1)\exp(-w_2)=\exp(w_1-w_2+\bar w)
\]
where $\bar w\in\gfrak$ and $\|\bar w\|\ll \beta \|w_1-w_2\|$.

Using the open mapping theorem and Baker--Campbell--Hausdorff formula again, for all small enough $\beta$, 
there is $(w_\hfrak,w_\rfrak)=B_\hfrak(0,C\beta)\times B_\rfrak(0,C\beta)$ and $w'\in\gfrak$ with $\|w'\|\ll\|w_\hfrak\|\|w_{\rfrak}\|$, so that 
\be\label{eq:def-whfrak-wrfrak}
\exp(w_1-w_2+\bar w)=\exp(w_\hfrak)\exp(w_\rfrak)=\exp(w_\hfrak+w_{\rfrak}+w')
\ee
where $C$ and the implied constant are absolute. 
 
We show that $h=\exp(w_\hfrak)$ and $w=w_\rfrak$ satisfy the claims in the lemma.
In view of~\eqref{eq:def-whfrak-wrfrak}, we need to verify the bounds on $\|h-I\|$ and $\|w_\rfrak\|$. 

First note that if $\beta$ is small enough, ~\eqref{eq:def-whfrak-wrfrak} implies that 
\be\label{eq:def-whfrak-wrfrak-2}
w_1-w_2+\bar w=w_\hfrak+w_{\rfrak}+w'.
\ee
Recall that we are using the max norm with respect to $\rfrak$ and $\hfrak$ which are two orthogonal subspaces.
Note also that $w_1, w_2, w_\rfrak\in\rfrak$ and $w_\hfrak\in\hfrak$.
Thus,~\eqref{eq:def-whfrak-wrfrak-2} implies that $\|w_\hfrak\|\ll \|\bar w\|+\|w'\|$.
Recall now that $\|\bar w\|\ll \beta \|w_1-w_2\|$ and $\|w'\|\ll \|w_\hfrak\|\|w_\rfrak\|\ll \beta\|w_\hfrak\|$.
Thus assuming $\beta$ is small enough, we conclude that $\|w_\hfrak\|\ll \beta\|w_1-w_2\|$ as we wanted to show.

To see the estimate on $\|w_\rfrak\|$, we again use~\eqref{eq:def-whfrak-wrfrak-2}. 
Indeed $(w_1-w_2)-w_\rfrak=w_\hfrak+w'-\bar w$; moreover, $\|\bar w\|\ll \beta \|w_1-w_2\|$, $\|w_\hfrak\|\ll \beta\|w_1-w_2\|$,
and $\|w'\|\ll \|w_\hfrak\|\|w_\rfrak\|\ll \beta\|w_\hfrak\|\ll \beta^2\|w_1-w_2\|$. Again assuming $\beta$ is small enough, we conclude that  
\[
0.5\|w_1-w_2\|\leq \|w_\rfrak\|\leq 2\|w_1-w_2\|,
\]
which finishes the proof. 
\end{proof}

\begin{lemma}\label{lem:dist-sheet} 
There exists $\beta_0$ so that the following holds for all $0<\beta\leq \beta_0$.
Let $x\in X_{10\beta}$ and $w\in B_\rfrak(0,\beta)$. If there are $h,h'\in \boxH_{2\beta}$
so that $\exp(w')hx=h'\exp(w)x$, then 
\[
\text{$h'=h\quad$ and $\quad w'=\Ad(h)w$}.
\]
Moreover, we have $\|w'\|\leq 2\|w\|$.
\end{lemma}

\begin{proof}
Recall that $\rfrak$ is invariant under the adjoint action of $H$. 
We rewrite the equation $\exp(w')hx=h'\exp(w)x$ as follows
\be\label{eq:dist-sheet-1}
\exp(w')hx=\exp(\Ad(h')w)h'x.
\ee
Since $h'\in \boxH_{2\beta}$, we have $\Ad(h')w'=w'+\hat w$ where $\|\hat w\|\ll \beta\|w'\|$.
Therefore, assuming $\beta$ is small enough, we have 
$0.5\|w\|\leq \|\Ad(h')w'\|\leq 2\|w\|$. 
This estimate,~\eqref{eq:dist-sheet-1}, and the fact that $x\in X_{10\beta}$
imply that 
\[
\exp(w')h=\exp(\Ad(h')w)h'.
\]

Moreover, the map $(\bar w,\bar h)\mapsto\exp(\bar w)\bar h$
from $B_\rfrak(0,2\beta)\times\boxH_{2\beta}$ to $G$ is injective, for all small enough $\beta$. 
Therefore, $h=h'$ and $w'=\Ad(h')w$.

The final claim follows as $\|w'\|=\|\Ad(h')w\|\leq 2\|w\|$.
\end{proof}

\subsection*{The set $\coneH_{\eta, t,\beta}$}
For all $\eta,\beta>0$ and $t\geq0$, set 
\be\label{eq:def-Ct}
\coneH_{\eta, t,\beta}:=\boxHs_{\beta}\cdot a_t\cdot \big\{u_r: r\in[0,\eta]\big\} \subset H.
\ee
Then $m_H(\coneH_{\eta,t,\beta})\asymp \eta{\beta}^2\nuni^{t}$ where $m_H$ denotes our fixed Haar measure on $H$.

Throughout the paper, the notation $\coneH_{\eta, t,\beta}$ will be used only for $\eta,t, \beta>0$ 
which satisfy $\nuni^{-0.01t}<\beta<\eta^2$ even if this is not explicitly mentioned.

For all $\eta,\beta, m>0$, put 
\be\label{eq:def-B-ell-beta}
\umt^H_{\eta,\beta,m}=\Bigl\{u^-_s: |s|\leq \beta \nuni^{- m}\Bigr\}\cdot\{a_t: |t|\leq \beta\}\cdot\Bigl\{u_r: |r|\leq \eta\Bigr\}.
\ee
Roughly speaking, $\umt_{\eta, \beta,m}^H$ is a {\em small thickening} of the $(\beta,\eta)$-neighborhood of the identity in $AU$. 
We write $\umt^H_{\beta,m}$ for $\umt^H_{\beta,\beta,m}$

The following lemma will also be used in the sequel.

\begin{lemma}\label{lem:commutation-rel}
\begin{enumerate}
\item Let ${m}\geq 1$, and let $0<\eta,\beta<0.1$. Then 
\[
\bigg(\Bigl(\umt_{0.01\eta,0.01\beta,m}^H\Bigr)^{\pm1}\bigg)^3\subset \umt_{\eta,\beta, m}^H.
\]
\item For all $0\leq \beta\leq \eta\leq 1$, $t,m>0$, and all $|r|\leq 2$, we have
\be\label{eq:well-rd-tau-1}
\Bigl(\umt_{\beta^2, m}^H\Bigr)^{\pm1}\cdot a_m u_r \coneH_{\eta, t,\beta'}\subset a_m u_r\coneH_{\eta,t, \beta},
\ee
where $\beta'= \beta-100\beta^2$.
\end{enumerate}
\end{lemma}

\begin{proof}
Recall that for all $a,b,c,d$ with $ad-bc=1$ and $a\neq 0$, we have 
\[
\begin{pmatrix} a& b \\ c & d \end{pmatrix}=\begin{pmatrix} 1& 0 \\ c/a & 1 \end{pmatrix}\begin{pmatrix} a & 0 \\ 0 & 1/a \end{pmatrix}\begin{pmatrix} 1 & b/a \\ 0 & 1 \end{pmatrix}.
\]
The claim in part~(1) follows from this identity.

To see part~(2), recall that 
\[
(u_{s}^-a u_{\rel'})\cdot (a_{{m}}\uvk) =a_{{m}}\uvk \uvk^{-1}u_{\uni^{{m}}s}^-a u_{\uni^{-{m}}\rel'}\uvk
\]
for all $u_{s}^-a u_{\rel'}\in\umt_{\beta^2, m}^H$.

Note that $\nuni^{{m}}|s|\leq \beta^2$ and $\nuni^{-{m}}|\rel'|\leq \beta^2$. 
Let now 
\[
(u^-_{c}a_du_b)\cdot a_t \cdot u_{r''}\in\coneH_{\eta, t,\beta-100\beta^2}
\]
where $|c|, |d|, |b|\leq \beta-100\beta^2$, $|r''|\leq \eta$. 

Then 
\begin{align*}
(u_{s}^-a u_{\rel'})(a_{{m}}\uvk)(u^-_{c}a_du_ba_tu_{r''})&=a_mu_r(\uvk^{-1}u_{\uni^{{m}}s}^-a  u_{\uni^{-{m}}\rel'}\uvk)(u^-_{c}a_du_b)a_t u_{r''}.
\end{align*}

Since $|r|\leq 2$, we have $u_r\cdot\boxH_{\beta^2}\cdot u_{-r}\subset \boxH_{10\beta^2}$. Moreover, $\boxH_{10\beta^2}\cdot\boxH_{\beta}\subset \boxH_{\beta+100\beta^2}$. The claim follows.
\end{proof}

\subsection*{A linear algebra lemma}
Note that both $\hfrak$ and $\rfrak$ 
are invariant under the adjoint representation of $H$ on $\mathfrak g$; moreover, 
both of these representations are isomorphic to the adjoint representation of $H$ on $\Lie(H)$. 

We will use the following lemma in the sequel

\begin{lemma}[\cite{EMM-Upp}, Lemma 5.1, and~\cite{EM-RW}]\label{lem:EMM-(2,1)}
Let $1/3<\alpha<1$, $0\neq w\in\mathfrak g$, and $t>0$. Then 
\[
\ave\|a_{t}{\uvk} w\|^{-\alpha}\uvkd\leq \frac{\ref{E:LA}\nuni^{-\hat\alpha t}}{2-2^\alpha}{\|w\|^{-\alpha}};
\]
where $\constE\label{E:LA}$ is an absolute constant and $\hat\alpha=\tfrac{1-\alpha}{4}$.
\end{lemma}

We will apply the above lemma with $t=\ell m_\alpha$, $\ell\in\bbn$, where $m_{\alpha}$ is defined by 
$\frac{\ref{E:LA}}{2-2^\alpha}\nuni^{-\hat\alpha {m_\alpha}}=\nuni^{-1}$. The choice of $m_\alpha$ and Lemma~\ref{lem:EMM-(2,1)}
imply
\be\label{eq:EMM-use}
\ave{\|a_{{m_\alpha}}\uvk w\|^{-\alpha}}\uvkd\leq \nuni^{-1}\|w\|^{-\alpha}.
\ee

\section{Nondivergence results}\label{sec:non-div}
In this section, we record some facts which will be used to deal with non-uniform lattices; 
the results in this section are known to the experts.
Our goal here is to tailor these results to our applications in the paper. 

Throughout this section, $\Gamma$ is assumed to be non-uniform unless otherwise is explicated. 
{\em We do not assume $\Gamma$ is arithmetic in this section.} 

To deal with cases where $\Gamma$ may not be arithmetic, we appeal to some facts from hyperbolic geometry, see Case 1 below. 
If $\Gamma$ is a non-uniform irreducible lattice in $\SL_2(\bbr)\times\SL_2(\bbr)$, i.e.\ Case 2 below, 
$\Gamma$ is arithmetic by a theorem of Selberg --- this is a special case of Margulis' arithmeticity theorem. 

\medskip

\begin{propos}\label{prop:Non-div-main}\label{prop:one-return}\label{lem:one-return}
There exist $\constE\label{E:non-div-main}\geq 1$ with the following property. 
Let $0<\vare, \eta<1$ and $x\in X$.
Let $I\subset[-10,10]$ be an interval with $|I|\geq\eta$. Then
\[
\Bigl|\Bigl\{r\in I:\inj(a_t\uvk x)< \vare^2\Bigr\}\Bigr|<\ref{E:non-div-main}\vare |I|
\]
so long as $t\geq |\log(\eta^2\inj(x))|+\ref{E:non-div-main}$.
\end{propos}

Proposition~\ref{prop:one-return} in particular implies that 
for all $t\geq \log\bigl(\eta^2\inj (x)\bigr)+O(1)$ {\em most} points in $\{a_t\uvk x: r\in I\}$ return to a fixed compact subset of $X$.

For the proof of the proposition,
it is more convenient to investigate two separate cases as follows. These are: 

\medskip

{\em Case 1:} $G=\SL_2(\bbc)$ or $G=\SL_2(\bbr)\times \SL_2(\bbr)$ and $\Gamma$ is reducible.

{\em Case 2:} $G=\SL_2(\bbr)\times \SL_2(\bbr)$ and $\Gamma$ is irreducible. 

\medskip

The proofs ultimately rely on non-divergence results of Margulis, Dani, and Kleinbock. To prepare the stage for such results to be applicable, in Case 1 we use the thick-thin decomposition from hyperbolic geometry. This will be completed in this section. 
In Case 2 thanks to Selberg's theorem $\Gamma$ is an arithmetic lattice. The proof in this case uses explicit reduction theory of such lattices and and the aforementioned works of Margulis et al; this proof is given in Appendix~\ref{sec:app-non-div}.

Let us thus assume $G=\SL_2(\bbc)$ or $G=\SL_2(\bbr)\times \SL_2(\bbr)$ and $\Gamma$ is reducible.
Let $\rc$ denote $\bbr$ or $\bbc$, and let $\Delta\subset \SL_2(\rc)$ be a lattice. 
Using the thick-thin decomposition of $\SL_2(\rc)/\Delta$,  
there exists a compact subset $\mathfrak S\subset\SL_2(\rc)/\Delta$ and a finite collection of disjoint
cusps $\{\mathfrak C_j: 1\leq j\leq\ell\}$ so that 
\[
\SL_2(\rc)/\Delta=\mathfrak S \bigsqcup (\sqcup_{j=1}^{\ell} \mathfrak C_j).
\]

Each cusp $\mathfrak C_j$ corresponds to the $\Delta$-orbit of a parabolic fixed point of $\Delta$ in $\partial\bbh^d$, $d=2$ or $3$ depending on $\rc$;  
alternatively, $\mathfrak C_j$ corresponds to a tube of closed $U$-orbits 
\[
a_t\mathsf Ng_j\Delta\subset \SL_2(\rc) \qquad t<0
\]  
where $\mathsf N$ denotes the group of upper triangular unipotent matrices in $\SL_2(\rc)$.

We will also consider a linearized version of the thick-thin decomposition. 
It is more convenient to identify
$\SL_2(\rc)/\{\pm I\}$ with $\SO(\mathsf Q)^\circ$ where $\mathsf Q(v_1,v_2,v_3)=2v_1v_3+v_2^2$ if $d=2$, and $\mathsf Q(v_1,v_2,v_3,v_4)=2v_1v_4+v_2^2+v_3^2$ if $d=3$. 
We choose this identification so that $\mathsf N$ fixes $\mathsf e_1$ where $\{\mathsf e_j\}$ is the standard basis for $\mathbb R^{d+1}$.

If $d=2$, that is $\rc=\bbr$, we let $L=\SO(\mathsf Q)^\circ$ and write $W=\mathbb R^3$. If $d=3$, that is: $\rc=\bbc$, we let $L$
be the isometry group of the restriction of $\mathsf Q$ to  the subspace $W$ spanned by $\{\mathsf e_1, \mathsf e_3, \mathsf e_4\}$ --- in the latter case $L\simeq\PSL_2(\bbr)$ and $h\mathsf e_2=\mathsf e_2$ for all $h\in L$. 
Note that in both cases the adjoint action of $H$ on $\sl_2(\bbr)$ factors through the action of $L$ on $W$. 

Set $v_j:= g_j^{-1}\mathsf e_1$ for $1\leq j\leq \ell$ where $\mathsf e_1$ is the first coordinate vector in $\bbr^{d+1}$ and $g_j\in\SL_2(\rc)$. 
Note that $\Delta v_j\subset \bbr^{d+1}$ is a closed (and hence discrete) subset of $\bbr^{d+1}$, see e.g.~\cite[Lemma 6.2]{MO-MargFun}.

Given a point $g\Delta\in\SL_2(\rc)/\Delta$ we define 
\[
\omega_\Delta(g\Delta)= \max\biggl\{2, \max\Bigl\{{\|g\delta v_j\|}^{-1}: \delta \in \Delta, 1\leq j\leq \ell\Bigr\}\biggr\}.
\] 

For the following see e.g.~\cite[\S6]{MO-MargFun}.

\begin{lemma}\label{lem:one-return-1}
Let $\Delta\subset\SL_2(\rc)$ be a lattice.
There exists some $C=C(\Delta)>2$ so that the following holds. 
Assume that $\omega_\Delta(g\Delta)\geq C$ for some $g\Delta\in \SL_2(\rc)/\Delta$.
Then there exists some $1\leq j_0\leq \ell$ and some $\delta_0\in\Delta$ so that $\|g\delta_0v_{j_0}\|^{-1} =\omega_\Delta(g\Delta)$ and 
\[
\|g\delta v_{j}\|> 1/C,\quad\text{ for all } \;(\delta,j)\neq (\delta_0, j_0).
\]
\end{lemma}

We will also use the following elementary lemma.

\begin{lemma}\label{lem:one-return-1-1}
Let $\eta>0$, and let $I$ be an interval of length at least $\eta$. 
There exists some $\constE\label{E:C-alpha}$ so that the following holds. 
Let $\varrho>0$, and let $v\in \SO(\mathsf Q)^\circ .\mathsf e_1$. Then 
\[
\Big|\Big\{r\in I:\|a_t\uvk v\|\leq \nuni^{t}\eta\|v\|\varrho^2\Big\}\Big|\leq \ref{E:C-alpha}\varrho|I|.
\]
\end{lemma}

\begin{proof}
Note that we may assume $\varrho$ is small compared to absolute constants.

Let us consider the case $d=3$, the other case, i.e., $d=2$, is contained in this case. 
Recall that $W$ denotes the $\bbr$-span of $\{\mathsf e_1,\mathsf e_3,\mathsf e_4\}$; 
write $v=c_v\mathsf e_2+w_v$ where $w_v\in W$ and $c_v\in\bbr$.
Since $\mathsf Q(v)=0$, we have $\|w_v\|\geq c\|v\|$ for some absolute constant $0<c<1$.
Moreover, for every $h\in L=H$
\be\label{eq:he2=e2}
hv=c_v\mathsf e_2+hw_v.
\ee 

Identifying $W$ with the adjoint representation of $H$, 
for every $w\in W$  and every $0<\delta<1$, let
\[
I(w,\delta)=\Big\{r\in I: |(\Ad(\uvk) w)_{12}|\leq 0.01\delta\eta^2\|w\|\Big\}
\]
where $w_{ij}$ is the $(i,j)$-th entry of $w\in\sl_2(\bbr)$. 

A direct computation gives
\be\label{eq:r-theta-12}
\Bigl(\Ad(\uvk) w\Bigr)_{12}=-w_{21}r^2-2w_{11}r+w_{12}.
\ee
Therefore, $\sup_{I}|(\Ad(\uvk) w)_{12}|\geq 0.01\eta^2\|w\|$ --- recall that $|I|\geq\eta$.
We conclude that $|I(w,\delta)|\leq C\delta^{1/2}|I|$ for some $C>0$, see e.g.~\cite[\S 3]{KM-Nondiv}.

Let $\delta=100c^{-1}\varrho^{2}$, where we assume $\varrho$ is small enough so that $\delta<1$.  
Let $v$ be as in the statement, and define $w_v$ as above.
Then $\|w_v\|\geq c\|v\|$ and $|I(w_v,\delta)|\leq 10Cc^{-1/2}\varrho|I|$. 

Let $r\in I\setminus I(w_v,\delta)$, then 
\[
\|(\Ad(\uvk)w_v)_{12}\|\geq c^{-1}\eta^2\|w_v\| \varrho^{2}.
\]
Since $a_t$ expands the $(1,2)$-entry by a factor of $e^{t}$, we conclude  
\begin{align*}
\|a_t\uvk v\|&\geq \|a_t\uvk w_v\|&&\text{by~\eqref{eq:he2=e2}}\\
&\geq \nuni^t|(\Ad(\uvk)w_v)_{12}|\geq c^{-1}\nuni^t\eta^2\|w_v\|\varrho^{2}\\
&\geq \nuni^t\eta^2\|v\|\varrho^{2}.
\end{align*}
The claim thus holds with $\ref{E:C-alpha}=10Cc^{-1/2}$.
\end{proof}

\begin{proof}[Proof of Proposition~\ref{prop:one-return}: Case 1]
Let us first consider $G=\SL_2(\bbr)\times \SL_2(\bbr)$. 
Since $\Gamma$ is reducible, there exists a finite index subgroup $\Gamma'\subset\Gamma$ so that 
$\Gamma'=\Gamma_1\times\Gamma_2$. The constant $\ref{E:non-div-main}$ in Proposition~\ref{lem:one-return}
is allowed to depend on the index of $\Gamma'$ in $\Gamma$, thus, abusing the notation, we replace $\Gamma$ by $\Gamma'$ 
in the remaining parts of the argument. In particular,  
\[
X=X_1\times X_2=\SL_2(\bbr)/\Gamma_1 \times \SL_2(\bbr)/\Gamma_2.
\]

Let us write $\omega_{i}$ for $\omega_{\Gamma_i}$, for $i=1,2$. Define  
\be\label{eq:def-omega}
\omega(x):=\max\{\omega_{1}(x_1), \omega_{2}(x_2)\}
\ee
for all $x=(x_1,x_2)\in X$. 

We denote the corresponding vectors for $\Gamma_1$ by $v_{1j}$, $1\leq j\leq \ell_1$, and for $\Gamma_2$
by $v_{2k}$, $1\leq k\leq \ell_2$.  

Note that $\omega(x)\asymp{\rm inj}(x)^{-1}$, see e.g.~\cite[Prop.~6.7]{MO-MargFun}.
Therefore, it suffices to prove the proposition with ${\rm inj}(x)$ replaced by $\omega(x)$.

Let $(g_1,g_2)\in G$, $(\gamma_1,\gamma_2)\in\Gamma$, $1\leq j\leq \ell_1$, and $1\leq k\leq \ell_2$. 
By Lemma~\ref{lem:one-return-1-1} applied with $g_1\gamma_1v_{1j}$ and $g_2\gamma_2v_{2k}$, we conclude  
\[
\Bigl|\Bigl\{r\in I:\|a_t\uvk(g_1\gamma_1v_{1j}, g_2\gamma_2v_{2k})\|\leq \nuni^{t}\eta^2\|(g_1v_{1j},g_2v_{2k})\|\varrho^2\Bigr\}\Bigr|\leq 2\ref{E:C-alpha}\varrho|I|
\]
for every $0<\varrho<1$. 

Let $\varrho_0=0.1\ref{E:C-alpha}^{-1}$, and choose $(g_1,g_2)\in G$ so that $x=(g_1\Gamma,g_2\Gamma)$.
Then the above implies that for all $(\gamma_1,\gamma_2)\in\Gamma$, all $1\leq j\leq \ell_1$, and all $1\leq k\leq \ell_2$, 
there exists some $r\in I$ so that 
\begin{align}
\notag\|a_t\uvk(g_1\gamma_1v_{1j}, g_2\gamma_2v_{2k})\|&\geq \nuni^{t}\eta^2\|(g\gamma_1v_{1j},g_2\gamma_2v_{2k})\|\vare^2\\
\label{eq:vect-grow-KM}&\geq \nuni^{t}\eta^2\omega(x)^{-1}\varrho_0^2.
\end{align}
In view of~\eqref{eq:vect-grow-KM}, and by choosing $\ref{E:non-div-main}$ 
large enough to account for the implicit constant in $\omega(x)\asymp \inj(x)^{-1}$, we have   
\[
\sup\{\|a_t\uvk(g_1\gamma_1v_{1j}, g_2\gamma_2v_{2k})\|: r\in I\}\geq \varrho_0^2
\]
so long as $t\geq |\log(\eta^2\inj(x))|+\ref{E:non-div-main}$. 

Therefore, we may apply~\cite[Thm.~4.1]{KM-Nondiv} and the proposition follows in this case.
The argument in the case $G=\SL_2(\bbc)$ is similar --- in light of 
Lemma~\ref{lem:one-return-1}, the use of~\cite[Thm.~4.1]{KM-Nondiv} simplifies significantly.
\end{proof}

As we mentioned the proof in Case 2 is given in Appendix~\ref{sec:app-non-div}.

\begin{propos}\label{prop:non-div}
There exists $0<\eta_X<1$, depending on $X$, so that the following holds. 
Let $0<\eta<1$ and let $x\in X$. Let $I\subset[-10,10]$ be an interval with length at least $\eta$. Then 
\[
|\{r\in I: a_t u_rx\in X_{\eta_X}\}|\geq 0.99|I|
\]
for all $t\geq |\log(\eta^2\inj(x))|+\ref{E:non-div-main}$. 
\end{propos}

\begin{proof}
Apply Proposition~\ref{lem:one-return} with $\vare=0.01\ref{E:non-div-main}^{-1}$. The claim thus holds with $\eta_X=\vare^2$. 
\end{proof}

\subsection{The subsets $X_{\rm cpt}$ and $\mathfrak S_{\rm cpt}$}\label{sec:SiegelSet} 
Decreasing $\eta_X$ if necessary we always assume that $X\setminus X_{\eta_X}$ is a disjoint union (possibly empty) of finitely many cusps. 

If $X$ is compact, let $X_{\rm cpt}=X$; otherwise, let $X_{\rm cpt}=\{gx: x\in X_{\eta_X}, \|g-I\|\leq 2\}$   
where $X_{\eta_X}$ is given by Proposition~\ref{prop:non-div}. %In particular, if $\lf$ is non-Archimedean, then $X_{\rm cpt}=X$.

We also fix once and for all a compact subset with piecewise smooth boundary $\mathfrak S_{\rm cpt}\subset G$ which projects onto $X_{\rm cpt}$.

\medskip

We end this section with the following

\begin{lemma}\label{lem:non-div-closed}
Let $Y$ be a periodic $H$-orbit. 
Then $\mu_Y(X_{\eta_X})\geq 0.9$ where $\mu_Y$ denotes the $H$-invariant probability measure on $Y$.
\end{lemma}

\begin{proof}
Let $\varphi=\mathbbm 1_{X_{\eta}}$, and let $y\in Y$. 
Then by~\cite[\S2.2.2]{KMnonquasi} we have 
\[
\lim_{t\to\infty}\int_0^1 \varphi(a_tu_ry)\diff\!r=\int \varphi\diff\!\mu_Y.
\]

The lemma thus follows from Proposition~\ref{prop:non-div}. 
\end{proof}

 \section{From large dimension to effective density}\label{sec:Exp-Mixing}
In this section we use the exponential decay of correlations for the ambient space $X$ to prove 
Proposition~\ref{prop:1-epsilon-N}, which says that 
expanding translations of subsets of $N$ which are foliated by local $U$ orbits and have dimension close but not necessarily equal to 2 are equidistributed in $X$. 

This proposition will be used in the proofs of Theorems~\ref{thm:main} and~\ref{thm:main-closed}, but it is also 
of independent interest. The proof is similar to an argument in~\cite[\S3]{Venkatesh-Sparse}.

Recall our notation from~\S\ref{sec:notation}: $n(r,s)=u_rv_s$ where $v_s=n(0,s)$ and $u_r=n(r,0)\in U$. 
Recall also that $a_tn(r,s)a_{-t}=n(e^t(r,s))$ for all $t\in\bbr$ and all $(r,s)\in \bbr^2$.

We need the following estimate on the decay of correlations in $X$. 
There exists $\mixexp$ depending on $X$ so that  
\be\label{eq:actual-mixing}
\biggl|\int \varphi(gx)\psi(x)\diff\!{m_X}-\int\varphi\diff\!{m_X}\int\psi\diff\!{m_X}\biggr|\ll \nuni^{-\mixexp d(e,g)}\Scal(\varphi)\Scal(\psi)
\ee
for all $\varphi,\psi\in C^\infty_c(X)+\bbc\cdot 1$ where the implied constant is absolute and $d$ is our fixed right $G$-invariant on $G$, see e.g.~\cite[\S2.4]{KMnonquasi} and references there. 
We note that $\mixexp$ is absolute if $\Gamma$ is a congruence subgroup, see~\cite{Burger-Sarnak, Cl-tau, Gor-Mau-Oh}.

Here $\Scal(\cdot)$ is a certain Sobolev norm on $C_c^\infty(X)+\bbc\cdot1$ 
which is assumed to dominate $\|\cdot\|_\infty$ and the 
Lipschitz norm $\|\cdot\|_{\rm Lip}$. Moreover, $\Scal(g.f)\ll\|g\|^\star\Scal(f)$ where the implied constants are absolute.

Let us put 
\be\label{eq: cX Mixing section}
\bcX=\eta_X^{-1}\vol(G/\Gamma) 
\ee
where $\eta_X$ is as in Proposition~\ref{prop:non-div} and $\vol(G/\Gamma)$ is computed using the Riemannian metric $d$.

\medskip

We also need the following statement.

\begin{propos}[\cite{KMnonquasi}, Prop.~2.4.8]\label{prop:equid-trans-horo}
There exists $\constk\label{k:thick-mixing-prop}\gg \mixexp$ (where the implied constant is absolute) and 
an absolute constant $\constk\label{k:cusp-mixing-prop}$ so that the following holds.  
Let $0<\eta<1$, $t>0$, and $x\in X_{\eta}$. Then for every $f\in C_c^\infty(X)+\bbc\cdot1$, 
\[
\biggl|\int_{B_N(0,1)}f(a_tn.x)\diff\!n-\int f\diff\!m_X\biggr|\leq \ref{C:thickening KM} \eta^{-1/\ref{k:cusp-mixing-prop}}\Scal(f)e^{-\ref{k:thick-mixing-prop} t} 
\]
where $B_N(0,1)=\Bigl\{u_rv_s: 0\leq r,s\leq 1\Bigr\}$, the measure on $N$ is normalized so that $B_N(0,1)$ has measure $1$, and 
$\constE\label{C:thickening KM}\leq L\bcX^L$ for an absolute constant $L$ and $\bcX$ as in~\eqref{eq: cX Mixing section}. 
\end{propos} 

\begin{proof}
This statement is well known to the experts, see e.g.~\cite{KMnonquasi, KM-Drich, McAdam, Katz-Quantitative}; 
we reproduce the argument for the convenience of the reader.

Throughout the argument, the implied exponents are absolute and implied multiplicative constants are $\leq L\bcX^L$ for an absolute $L$. 
Let $0\leq\varphi^+\leq 1$ be a smooth function supported on 
$B_N(0,1)$ so that $\int_{B_N(0,1)}(1-\varphi^+)\diff\!n\leq \nuni^{-\kappa t}$ and $\Scal(\varphi^+)\ll \nuni^{\star \kappa t}$ for some $\kappa$ which will be optimized later. Then 
\be\label{eq:thickening-mixing-1}
\biggl|\int_{B_N(0,1)}f(a_tn.x)\diff\!n-\int_{N}f(a_tn.x)\varphi^+(n)\diff\!n\biggr|\ll \|f\|_\infty\nuni^{-\kappa t}.
\ee

Recall that $B_N(0,1)X_\eta\subset X_{0.1\eta}$; using a smooth partition of unity argument, we can write $\varphi^+=\sum_{j=1}^M \varphi^+_j$ so that $M\ll \eta^{-\star}$, $\Scal(\varphi^+_j)\ll \eta^{-\star}\nuni^{\star\kappa t}$, and the map $g\mapsto gy$ is injective on $\supp(\varphi^+_j)$ for all $y\in B_N(0,1).X_\eta$ and all $j$.

In consequence, we may fix one $\varphi^+_j$ for the rest of the argument. Arguing as in~\cite[Prop.~2.4.8]{KMnonquasi}, see also~\cite[Thm.~2.3]{KM-Drich}, 
there exists a compactly supported smooth function 
$\varphi$ (an $\nuni^{-\kappa t}$-{\em thickening} of $\varphi_j^+$ along the weak-stable directions in $G$) so that $\Scal(\varphi)\ll_X \eta^{-\star}\nuni^{\star\kappa t}$ and
\be\label{eq:thickening-mixing-2}
\biggl|\int_N f(a_tn.x)\varphi_j^+(n)\diff\!n-\int_X f(a_t y)\varphi(y)\diff\!{m_X}(y)\biggr|\ll
\|f\|_{\rm Lip}\nuni^{-\kappa t},
\ee
where $\|f\|_{\rm Lip}$ is the Lipschitz constant of $f$. 

Finally in view of~\eqref{eq:actual-mixing}, we have  
\be\label{eq:thickening-mixing-3}
\begin{aligned}
\biggl|\int f(a_t y)\varphi(y)\diff\!{m_X}(y)-\int f\diff\!{m_X}\int\varphi\diff\!{m_X}\biggr|&\ll \Scal(f)\Scal(\varphi)\nuni^{-\mixexp t}\\
&\ll \eta^{-\star}\nuni^{\star\kappa t}\Scal(f)\nuni^{-\mixexp t}.
\end{aligned}
\ee
The claim follows from~\eqref{eq:thickening-mixing-1},~\eqref{eq:thickening-mixing-2}, 
and~\eqref{eq:thickening-mixing-3} by optimizing $\kappa$.
\end{proof}

The following is a generalization of Proposition~\ref{prop:equid-trans-horo} where one replaces 
the average over $B_N(0,1)$ with an average over certain subsets of dimension close to $2$, but not necessarily equal to $2$.

\begin{propos}\label{prop:1-epsilon-N}
There exist $\constk\label{k:mixing}$ and $\vare_0$ (both $\gg\mixexp^2$ with an absolute implied constant) 
so that the following holds. Let $0\leq\vare\leq \vare_0$ and $0<b\leq 0.1$. Let $\rho$ be a probability measure on 
$[0,1]$ which satisfies 
\be\label{eq:C-rho-reg-N}
\rho(J)\leq C b^{1-\vare}
\ee
for every interval $J$ of length $b$ and a constant $C\geq 1$.

Let $0<\eta<1$, $x\in X_{\eta}$, then  
\[
\biggl|\int_0^1\int_0^1 f(a_tu_rv_s.x)\diff\!r\diff\!\rho(s)-\int f\diff\!{m_X}\biggr|\leq \ref{C:1-epsilon-N} C\eta^{-\frac{1}{2\ref{k:cusp-mixing-prop}}}\Scal(f)e^{-\ref{k:mixing} t}. 
\]
for all $|\log b|/4\leq t\leq |\log b|/2$ and all $f\in C_c^\infty(X)+\bbc\cdot1$, where  $\constE\label{C:1-epsilon-N}\leq L\bcX^L$ for an absolute constant $L$ and $\bcX$ as in~\eqref{eq: cX Mixing section}.
\end{propos}
 
\begin{proof}
We will prove this for the case $G=\SL_2(\bbr)\times\SL_2(\bbr)$; 
the proof in the case $G=\SL_2(\bbc)$ is similar. 

Throughout the argument, the implicit multiplicative constants are $\leq L\bcX^L$ for some absolute $L$. 

Without loss of generality, we may assume $\int_Xf\diff\!{m_X}=0$. 

Let $M\in\bbn$ be so that $1/M\leq b\leq 1/(M-1)$.
For every $1\leq j\leq M$, let $I_j=\Big[\frac{j-1}{M}, \frac{j}{M}\Big)$; also put $s_j=\frac{2j-1}{2M}$ and $c_j=\rho(I_j)$ for all $j$.
Since $I_j$'s are disjoint, we have $\sum_j c_j=1$.

For all such $j$, let
\[
\mathsf B_j=\Bigl\{u_rv_s: r\in[0,1], s\in (s_j-\tfrac{b}{4},s_j+\tfrac{b}{4})\Bigr\}.
\]
In view of the choice of $M$, we have $\mathsf B_j\cap\mathsf B_{j'}=\emptyset$ for all $j\neq j'$.
Let $\varphi=\textstyle\sum_j2b^{-1}c_j\mathbbm 1_{\mathsf B_j}$. Then $\int_N \varphi(n(r,s))\diff\!r\diff\!s=1$.

We make the following observation. Using~\eqref{eq:C-rho-reg-N}, we have $c_j\leq Cb^{1-\vare}$ for all $j$. This and the fact that $\mathsf B_j$'s are disjoint imply that  
\be\label{eq:phi-bound}
\varphi(n(z))\leq \max\{2b^{-1}c_j: 1\leq j\leq M\}\leq 2Cb^{-\vare}
\ee
for all $n(z)\in N$; here and in what follows, $z=(r,s)$ and $\diff\!z=\diff\!r\diff\!s$. 

Using the fact that $I_j$'s are disjoint, we have 
\[
\int_0^1\int_0^1 f(a_tu_rv_s.x)\diff\!r\diff\!\rho(s)=\sum_j\int_{I_j}\int f(a_tu_rv_s.x)\diff\!r\diff\!\rho(s);
\]
thus, we conclude that
\begin{align}
\label{eq:iint-rho-sj}&\bigg|\int_0^1\int_0^1f(a_tu_rv_s.x)\diff\!r\diff\!\rho(s)-\sum_{j}c_j\int f(a_tu_rv_{s_j}.x)\diff\!r\bigg|\\
\notag&\leq \sum_{j}\int_{I_j}\int \Big|f(a_tu_rv_s.x)-f(a_tu_rv_{s_j}.x)\Big|\diff\!r\diff\!\rho(s)\ll  \Scal(f)b^{1/2}
\end{align}
where we used the facts that $|s-s_j|\leq b$ and $t\leq |\log b|/2$ in the last inequality.

In view of~\eqref{eq:iint-rho-sj}, thus, we need to bound $\sum_jc_j\int f(a_tu_rv_{s_j}x)\diff\!r$.
Similar to~\eqref{eq:iint-rho-sj}, we can now make the following computation. 
 \begin{align}
 \label{eq:mix-1st}\biggl|\sum_{j} \int_0^1c_jf(a_tn(r,s_j).x)\diff\!r-\int_{N}\varphi(n(z))f(a_tn(z).x)&\diff\!z\biggr|\leq\\
\notag \sum_{j} \int_0^12b^{-1}c_j\int_{s_j-\tfrac{b}{4}}^{s_j+\tfrac{b}{4}} \Bigl|f(a_tn(r,s_j).x)-f(a_tn(r,s).x)\diff\!s\Bigr|\diff\!r&\ll\Scal(f)b^{1/2}
 \end{align}
where again we used the facts that $|s-s_j|\leq b$ and $t\leq |\log b|/2$.
 
 Thus, it suffices to investigate 
 \[
 A_1=\int \varphi(n(z))f(a_tn(z).x)\diff\!z.
 \] 
 To that end, let $\ell\geq 2$ be a parameter which will be optimized later. Set $\tau=e^{\frac{1-\ell}{\ell}t}= e^{-t+\frac{t}{\ell}}$, and define  
 \[
 A_2:=\frac{1}{\tau}\int_{0}^\tau\int \varphi(n(z))f(a_t u_rn(z).x)\diff\!z\diff\!r;
 \]
roughly speaking, we introduce an extra averaging in the direction of $U$. 
 
For every $0\leq r\leq \tau$, we have 
$|(\mathsf B_j+r)\Delta \mathsf B_j|\ll |\mathsf B_j|\tau$. Hence, 
\begin{align*}
\biggl|\int \varphi(z)&f(a_tu_rn(z).x)\diff\!z-\int \varphi(z)f(a_tn(z).x)\diff\!z\biggr|\\
&\leq\sum_j2b^{-1}c_j\int_{(\mathsf B_j+r)\Delta \mathsf B_j}|f(a_tn(z)x)|\diff\!z\\
&\leq \sum_j2b^{-1}c_j|\mathsf B_j|\tau\|f\|_\infty\\
&\leq \|f\|_\infty\tau\ll \Scal(f)\tau;
\end{align*}
we used $|\mathsf B_j|=b/2$ for every $j$ and $\sum c_j=1$, in the penultimate inequality. 
Averaging the above over $[0,\tau]$, we conclude that 
\be\label{eq:mix-2nd}
|A_1-A_2|\ll \Scal(f)\tau\leq \Scal(f)e^{-t/2}\ll \Scal(f)b^{1/8};
\ee 
recall that $\tau=e^{\frac{1-\ell}{\ell}t}$, $\ell\geq 2$, and $t\geq |\log b|/4$. 
 
In consequence, we have reduced to the study of $A_2$ to which we now turn. By the Cauchy-Schwarz inequality, we have
\[
|A_2|^2\leq \int \biggl(\frac{1}{\tau}\int_{0}^\tau f(a_tu_rn(z).x)\diff\!r\biggr)^2\varphi(n(z))\diff\!z.
\]
Now using $\biggl(\frac{1}{\tau}\int_{0}^\tau f(a_tn(r+z).x)\diff\!r\biggr)^2\geq0$,~\eqref{eq:phi-bound}, and the above estimate, 
we conclude  
\begin{align}
 \notag|A_2|^2&\leq{2Cb^{-\vare}}\int_{B(0,1)}\biggl(\frac{1}{\tau}\int_{0}^\tau f(a_tn(z)u_r.x)\diff\!r\biggr)^2\diff\!z\\
\label{eq:mix-A2-est}&=\frac{1}{\tau^2}\int_0^\tau\int_0^\tau\int_{B(0,1)}{2Cb^{-\vare}}\hat f_{r_1,r_2}(a_tn(z).x)\diff\!z\diff\!r_1\diff\!r_2
 \end{align}
 where $B(0,1)=B_N(0,1)=\{u_rv_s: 0\leq r,s\leq 1\}$ has measure 1 with respect to $\diff\!z$, and for all 
 $r_1,r_2\in [0,\tau]$ we put 
 \[
 \hat f_{r_1,r_2}(y)=f(a_tu(r_1)a_{-t}.y)f(a_tu(r_2)a_{-t}.y).
 \]
 
 Note that $\Scal(\hat f_{r_1,r_2})\ll \Scal(f)^2 (e^t\tau)^\star\ll \Scal(f)^2 e^{\star t/\ell}$. 
 We now choose $\ell\ll 1/\ref{k:thick-mixing-prop}$ large enough so that 
 \be\label{eq:Sob-fhat}
 \Scal(\hat f_{r_1,r_2})\ll \Scal(f)^2 e^{\ref{k:thick-mixing-prop} t/2}.
 \ee
 
 By Proposition~\ref{prop:equid-trans-horo}, we have  
 \begin{multline*}
 \biggl|{b^{-\vare}}\int_{B(0,1)}\hat f_{r_1,r_2}(a_tn(z)x)\diff\!z\biggr|=b^{-\vare}\int_X\hat f_{r_1,r_2}\diff\!{m_X}\\
 + b^{-\vare} \eta^{-1/\ref{k:cusp-mixing-prop}}O(\Scal(\hat f_{r_1,r_2})e^{-\ref{k:thick-mixing-prop} t}).
 \end{multline*}
 Recall from~\eqref{eq:Sob-fhat} that $\Scal(\hat f_{r_1,r_2})e^{-\ref{k:thick-mixing-prop} t}\leq \Scal(f)^2e^{-\ref{k:thick-mixing-prop} t/2}$. 
 Moreover, since $t\geq |\log b|/4$ if we assume $\vare\leq \ref{k:thick-mixing-prop}/16$, then
 $e^{-\ref{k:thick-mixing-prop} t/2}b^{-\vare}\leq b^{\ref{k:thick-mixing-prop}/16}$. Altogether, we conclude that 
 \begin{multline}\label{eq:mixing-ineq-main}
 \biggl|{b^{-\vare}}\int_{B(0,1)}\hat f_{r_1,r_2}(a_tn(z)x)\diff\!z\biggr|=b^{-\vare}\int_X\hat f_{r_1,r_2}\diff\!{m_X} \\+ \Scal(f)^2 \eta^{-1/\ref{k:cusp-mixing-prop}}b^{\ref{k:thick-mixing-prop}/16}.
 \end{multline} 
 
 We now use estimates on the decay of matrix coefficients,~\eqref{eq:actual-mixing}, together with the fact that
 $d(e,u_t)\geq |t|$, and obtain the following bound.
 \be\label{eq:int-fhat}
 \biggl|\int_X\hat f_{r_1,r_2}(x)\diff\!{m_X}\biggr|\ll \Scal(f)^2e^{-\frac{\mixexp}{2\ell}t}\quad\text{if $|r_1-r_2|>e^{-t+\frac{t}{2\ell}}$.}
 \ee

 Divide now the integral $\int_0^{\tau}\int_0^\tau$ in~\eqref{eq:mix-A2-est}
 into terms: one with $|r_1-r_2|>e^{-t+\frac{t}{2\ell}}=\tau e^{-\frac{t}{2\ell}}$ and the other its complement.
 We thus get from~\eqref{eq:mix-A2-est},~\eqref{eq:mixing-ineq-main}, and~\eqref{eq:int-fhat} that
 \[
 |A_2|^2\ll C\eta^{-\frac{1}{\ref{k:cusp-mixing-prop}}}\Scal(f)^2\biggl(b^{-\vare} \biggl(e^{\frac{-\mixexp}{2\ell} t}+ e^{\frac{-1}{2\ell}t}\biggr)+b^{\ref{k:thick-mixing-prop}/16}\biggr).
 \]
 
Recall that $\ell\ll 1/\ref{k:thick-mixing-prop}$ and $\ref{k:thick-mixing-prop}\gg \mixexp$. Thus if $\vare\leq \ref{k:thick-mixing-prop}^2/L$ 
for a large enough $L$, the above, 
together with~\eqref{eq:iint-rho-sj}, \eqref{eq:mix-1st}, and~\eqref{eq:mix-2nd}, finishes the proof.
\end{proof}

\section{A Marstrand type projection theorem}\label{sec:Mars-proj}

In this section, we combine a certain projection theorem with some arguments in homogeneous dynamics 
to prove Proposition~\ref{prop:proj-general}. The outcome of this proposition will serve as an input when we apply Proposition~\ref{prop:1-epsilon-N}. 

\begin{propos}\label{prop:proj-general} 
Let $0<\eta<0.01\eta_X$, and let $0<100\vare< \alpha<1$. 
Suppose there exist $x_1\in X_{\eta}$ and $F\subset B_\rfrak(0,\eta^2)$, containing $0$, so that
\begin{align}
\notag&\mathcal F:=\{\exp(w)x_1: w\in F\}\subset X_\injr\qquad\text{and}\\
\label{eq:energy-F-proj-thm}&\sum_{w'\in F\setminus \{w\}}\|w-w'\|^{-\alpha}\leq D\cdot (\#F)^{1+\vare}\qquad\text{for all $w\in F$},
\end{align}
for some $D\geq 1$.

Assume further that $\#F$ is large enough, depending explicitly on $\eta$ and $\vare$. 
Then exists a finite subset $I\subset [0,1]$, some $\rhsco>0$ with
\be\label{eq:size of b1}
(\#F)^{-\frac{3-\alpha+5\vare}{3-\alpha+20\vare}}\leq \rhsco\leq (\#F)^{-\vare},
\ee 
and some $x_2\in X_\eta \cap \Bigl(a_{|\log(\rhsco)|}\cdot\{u_r: |r|\leq 2\}\Bigr).\mathcal F$ so that both of the following statements hold true. 
\begin{enumerate} 
\item The set $I$ supports a probability measure $\rho$ which satisfies  
\[
\rho(J)\leq C'_\vare\cdot |J|^{\alpha-30\vare}
\]
for all intervals $J$ with $|J|\geq (\#F)^{\frac{-15\vare}{3-\alpha+20\vare}}$, where $C'_\vare\ll\vare^{-\star}$ (with absolute implied constants). 
\item There is an absolute constant $C$, so that for all $s\in I$, we have 
\[
v_sx_2\in \Big(\boxG_{C\rhsco} \cdot a_{|\log(\rhsco)|}\cdot \{u_r: |r|\leq 2\}\Big).\mathcal F.
\]
\end{enumerate}
\end{propos}

The proof of Proposition~\ref{prop:proj-general} is based on the following projection theorem.
This theorem may be thought of as a finitary version of the work of K\"{a}enm\"{a}ki, Orponen, and Venieri,~\cite{kenmki2017marstrandtype}. 
Its proof, which is given in Appendix~\ref{sec:proof-proj}, is based on the works of Wolff and Schlag,~\cite{Wolff,Schlag} which in turn relies on a cell decomposition theorem of
Clarkson, Edelsbrunner, Guibas, Sharir, and Welzl~\cite{Cell-Decomposition}.

\begin{thm}\label{thm:proj-thm}
Let $0<\alpha, b_0, b_1<1$ ($\alpha$ should be thought of fixed, and $b_0<b_1$ as small). 
Let $E\subset B_\rfrak(0,\rhsco)$ be so that
\[
\tfrac{\#(E\,\cap\, B_\rfrak(w,\rhsc))}{\#E}\leq D'\cdot (\rhsc/\rhsco)^{\alpha}
\]
for all $w\in\rfrak$ and all $\rhsc\geq b_0$, and some $D'\geq 1$. 
Let $0<\kappa<0.1$, and let $J\subset\bbr$ be an interval.
There exists $J'\subset J$ with $|J'|\geq 0.9|J|$ satisfying the following. Let $r\in J'$, then
there exists a subset $E_r\subset E$ with 
\[
\#E_r\geq 0.9\cdot (\#E)
\]
such that for all $w\in E_r$ and all $\rhsc\geq b_0$, we have 
\[
\tfrac{\#\{w'\in E\,:\, |\xi_r(w')-\xi_r(w)|\leq \rhsc\}}{\#E}\leq C_\kappa\cdot(\rhsc/\rhsco)^{\alpha-7\kappa}
\] 
where $C_\kappa$ is a constant which depends polynomially on $\kappa$, $|J|$, and $D'$, and 
\be\label{eq:w12-xi-w-r-use}
\xi_r(w)=(\Ad(u_r)w)_{12}=-w_{21}r^2-2w_{11}r+w_{12}.
\ee
with $w_{ij}$ denoting the $(i,j)$-th entry of $w\in\rfrak$. 
\end{thm}

The proof of Proposition~\ref{prop:proj-general} will also use the following version of \cite[Lemma 5.2]{BFLM}, see also~\cite{Bour-Proj}. 
We reproduce the argument in Appendix~\ref{sec:max-ineq-proj}.

\begin{lemma}\label{lem:max-ineq-proj-thm}
Let $F\subset B_\rfrak(0,1)$ be a subset which satisfies~\eqref{eq:energy-F-proj-thm}. 
Then there exist $w_0\in F$, $\rhsco>0$, with
\[
(\#F)^{-\frac{3-\alpha+5\vare}{3-\alpha+20\vare}}\leq \rhsco\leq (\#F)^{-\vare},
\] 
and a subset $F'\subset B_\rfrak(w_0,\rhsco)\cap F$ so that the following holds. 
Let $w\in \rfrak$, and let $\rhsc\geq(\#F)^{-1}$. Then  
\[
\tfrac{\#(F'\cap B(w,\rhsc))}{\#F'}\leq C'\cdot (\rhsc/\rhsco)^{\alpha-20\vare}
\]
where $C'\ll_D\vare^{-\star}$ with absolute implied constants. 
\end{lemma}

We now begin the proof of the proposition.

\begin{proof}[Proof of Proposition~\ref{prop:proj-general}]
The general strategy is straightforward. 
First we apply Lemma~\ref{lem:max-ineq-proj-thm} to replace the set $F$ with a local version of it,
i.e., we replace $F$ with $F'\subset B_\rfrak(w_0,\rhsco)\cap F$. Then using Theorem~\ref{thm:proj-thm}, we  
project the {\em discretized dimension} in $\rfrak$ to the direction of $\Lie(V)=\rfrak\cap\Lie(N)$. 
Finally, we use the action of $A$ to expand this subset of $V$ to size $1$. 

The details however are a bit more involved, in particular, we need to carefully control the size of various elements; 
we also need to use Proposition~\ref{lem:one-return} (when $X$ is not compact) to ensure returns to $X_\eta$.  

Throughout the proof, we will assume $\#F$ is large enough so that 
\be\label{eq:num-elemnts-F-proj}
(\#F)^{-\vare}\leq (2\ref{E:BCH}\ref{E:non-div-main})^{-1}\eta^3,
\ee 
see Lemma~\ref{lem:BCH} and Proposition~\ref{lem:one-return}.

\subsection*{Localizing the entropy}
Apply Lemma~\ref{lem:max-ineq-proj-thm} with $F$ as in the proposition. 
Let $w_0\in F$, $\rhsco>0$, and $F'\subset B_\rfrak(w_0,\rhsco)\cap F$ be given by that lemma;
in particular, we have 
\be\label{eq:size-r-prj}
(\#F)^{-\frac{3-\alpha+5\vare}{3-\alpha+20\vare}}\leq \rhsco\leq (\#F)^{-\vare}.
\ee

Replacing $w_0$ with a different point in $F$ and increasing $C'$  if necessary, we will assume that 
$F'\subset B_\rfrak(w_0, b_1/(6\ref{E:BCH}))\cap F$. In view of Lemma~\ref{lem:BCH}, for all $w'\in F'$, there exist  $h\in H$ and $w\in \rfrak$ so that
\be \label{eq:h-exp(w)}
\begin{aligned}
    h\exp(w)&=\exp(w')\exp(-w_0)\\
    \|h-I\|\leq \rhsco^2/3\quad &\text{and}\quad \|w\|\leq 2\|w_0-w'\|\leq \rhsco/(3\ref{E:BCH}).
\end{aligned}
\ee

Set 
\be\label{eq:def-E}
E=\Bigl\{w\in\rfrak: \exists h \in H, w'\in F' \text{ so that $h,w,w_0,w'$ satisfy \eqref{eq:h-exp(w)}} \Bigr\}.
\ee 

\begin{lemma}\label{lem:CampHaus-again}
Let the notation be as above. Then 
\be\label{eq:E-regular}
\tfrac{\#(E\,\cap\, B(w,\rhsc))}{\#E}\leq \hat C\cdot (\rhsc/\rhsco)^{\alpha-20\vare}
\ee
for all $w\in\rfrak$ and $\rhsc\geq(\#F)^{-1}$ where 
$\hat C\leq 2C'$.  
\end{lemma}

This lemma is proved after the completion of the proof of the proposition.

Let $x'_2:=\exp(w_0)x_1$, and let $w'\in F'$. 
Then if $h$ and $w$ are as in \eqref{eq:h-exp(w)},
\be\label{eq:h-expw-Fcal}
h\exp(w)x'_2=\exp(w')\exp(-w_0)\exp(w_0)x_1=\exp(w')x_1\in \mathcal F.
\ee

We also need the following elementary lemma whose proof will be given after the completion of the proof of the proposition. 

\begin{lemma}\label{lem:1pt-1-vare-trans}
There exists $r_0\in[0,1]$ and a subset  
\[
\bar E\subset \Ad(u_{r_0})E\cap \Bigl\{w\in B_\rfrak(0, \injr): |w_{12}|\geq 10^{-3}\|w\|\Bigr\}
\] 
so that $\#\bar E\geq\#E/4$.
\end{lemma}

Thanks to Lemma~\ref{lem:1pt-1-vare-trans}, we may replace $x'_2$ by $u_{r_0} x'_2$ for some $r_0\in[0,1]$ 
and $E$ by a subset $\bar E$ with $\#\bar E\geq\# E/4$ (which we continue to denote by $E$), to ensure that
\be\label{eq:E-general-pos}
E\subset \Bigl\{w\in B_\rfrak(0, \injr): |w_{12}|\geq 10^{-3}\|w\|\Bigr\}
\ee
where $w_{12}$ denotes the $(1,2)$-th entry of $w\in\rfrak$, see~\eqref{eq:w12-xi-w-r-use}. 
Note that~\eqref{eq:E-regular}
holds for the new $E$ with $4\hat C$, we suppress the factor $4$.

\subsection*{Estimates on the size of elements}
Let $t=|\log(\rhsco)|$. By~\eqref{eq:h-expw-Fcal}, for all $r\in [0,1]$, we have 
\be\label{eq:atau-h-expw-Fcal}
a_t u_r h\exp(w).x'_2\in a_t\cdot\{u_r: r\in [0,1]\}.\mathcal F,
\ee
where $w\in E$, i.e, $h\exp(w)=\exp(w')\exp(-w_0)$.

We now investigate properties of the element $a_t u_r h\exp(w)u_{-r}a_{-t}$.
In view of~\eqref{eq:h-exp(w)} and the definition of $t$, for all $r\in[0,1]$, we have  
\begin{subequations}
\begin{align}
\label{eq:choice-t}&\|\Ad(a_t u_r)w\|\leq 1, \quad\text{and}\\
\label{eq:h-exp(w)-final}&\|a_t u_r hu_{-r}a_{-t}-I\|\leq \rhsco;
\end{align}
\end{subequations}
note, moreover, that $a_t u_r h u_{-r}a_{-t}\in H$.

In view of~\eqref{eq:E-general-pos}, for all $|r|\leq 10^{-4}$ we have  
\[
|(\Ad(u_r)w)_{12}|\geq 10^{-4}\|w\|.
\]
Therefore, for all $|r|\leq 10^{-4}$, we have 
\[
\Ad(a_t u_r)w=\begin{pmatrix}v_{11} & v_{12}\\ v_{21} & v_{22}\end{pmatrix}
\]
where $|v_{11}|, |v_{22}|\leq 10^{4}\nuni^{-t}|v_{12}|$ and $|v_{21}|\leq 10^{4}\nuni^{-2t}|v_{12}|$. 
Hence for $|r|\leq 10^{-4}$, we have 
\[
a_t u_r h\exp(w).x'_2=(a_tu_r hu_{-r}a_{-t}) \cdot g \cdot \exp\Bigl(e^t(\Ad(a_t u_r)w)_{12}E_{12}\Bigr) .a_tu_r x'_2;
\]
for some $g\in G$ which in view of the estimate in~\eqref{eq:choice-t} satisfies 
\be\label{eq:est-size-g-proj-thm}
\|g-I\|\ll \rhsco
\ee 
with an absolute implied constant.

Using~\eqref{eq:atau-h-expw-Fcal} and~\eqref{eq:h-exp(w)-final}, 
we conclude that 
\begin{multline}\label{eq:g-w12}
\exp\Bigl(e^t(\Ad(a_t u_r)w)_{12}E_{12}\Bigr).a_tu_r x'_2\in\\ \Big(\boxG_{C\rhsco} \cdot \boxH_{C\rhsco}\cdot a_t\cdot\{u_r: r\in [0,1]\}\Big).\mathcal F,
\end{multline}
where $C$ is an absolute constant.

\subsection*{Applying Theorem~\ref{thm:proj-thm}}
We now choose a particular $|r|\leq 10^{-4}$ in order to the define the set $I$ in Proposition~\ref{prop:proj-general}.
This choice is based on Proposition~\ref{lem:one-return} and Theorem~\ref{thm:proj-thm}. 

Recall that $t=|\log(\rhsco)|$ and 
\be\label{eq:delta0-est-again}
\rhsco\leq (\#F)^{-\vare}\leq (2\ref{E:BCH}\ref{E:non-div-main})^{-1}\eta^3.
\ee
Apply Proposition~\ref{lem:one-return} with $t$, $x'_2=\exp(w_0)x_1\in X_\eta$, and the interval $J=[-10^{-4},10^{-4}]$. Then if we set 
\be\label{eq:recurrence-proj}
J''=\{r: |r|\leq 10^{-4}, a_tu_r .x'_2\in X_{\eta}\}
\ee
by the proposition $|J''|> 0.9\cdot 2\cdot 10^{-4}$.

We also apply Theorem~\ref{thm:proj-thm} with $E$, $J=[-10^{-4},10^{-4}]$, $\alpha-20\vare$, and $\kappa=\vare$. 
Let $J'$ be given by that Theorem. Fix some $r \in J' \cap J''$ 
 for the remainder of the argument. 

Put $x_2:=a_tu_r.x'_2$.
By definition of $J''$ in~\eqref{eq:recurrence-proj}, $x_2\in X_{\eta}$, and by~\eqref{eq:g-w12} \be\label{eq:atau-rtheta-Boxz1}
\exp\Bigl(\nuni^t(\Ad(u_r w)_{12})\Bigr).x_2\in \Big(\boxG_{C\rhsco}\cdot\boxH_{C\rhsco}\cdot a_t\cdot\{u_r: r\in [0,1]\}\Big).\mathcal F.
\ee

In the notation of Theorem~\ref{thm:proj-thm}, put 
\[
I:=\{\nuni^t\xi_r(w):w\in E_r\};
\]
recall that $\xi_r(w)=(\Ad(a_\tau\rot_\theta)w)_{12}$.
We will show that the proposition holds with $x_2$, $I$, and $\rhsco$. 
First note that the claimed bound \eqref{eq:size of b1} on $\rhsco$ in the statement of the proposition holds in view of~\eqref{eq:size-r-prj}.
The assertion in part~(2) of the proposition also holds by~\eqref{eq:atau-rtheta-Boxz1}. 

Thus it only remains to establish~(1) of the proposition. Let $\rho$ be the pushforward of the normalized counting measure on $E_r$ under the map
$w\mapsto\nuni^t\xi_r(w)$. That is, 
\[
\rho(K)=\tfrac{\#\{w\in E_r\;:\; \nuni^t\xi_r(w)\in K\}}{\#E_r}
\] 
for any interval $K\subset \bbr$.

Recall again that $\nuni^{-t}=\rhsco$. 
Let $w\in E_r$, and put $s=\nuni^t\xi_r(w)$. 
By Theorem~\ref{thm:proj-thm}, and in view of the fact that $\#E_r\geq 0.9 \cdot (\#E)$, for every $\rhsc\geq \nuni^{t}\cdot(\#F)^{-1}$, we have that
\be\label{eq:proj-thm-use}
\begin{aligned}
\rho\Big(\{s'\in I: |s-s'|\leq \rhsc\}\Big)&=\frac{\#\Big\{w'\in E_r: |\xi_{r}(w')-\xi_r(w)|\leq \nuni^{-t}\rhsc\Big\}}{\#E_r}\\
&\leq \bar C_\vare\cdot (\nuni^{-t}\rhsc/\rhsco)^{\alpha-27\vare}=\bar C_\vare \rhsc^{\alpha-27\vare}
\end{aligned}
\ee
where $\bar C_\vare\ll \vare^{-\star}$. 

Using the estimate in~\eqref{eq:size-r-prj}, we have 
\[
\nuni^t\cdot(\#F)^{-1}\leq (\#F)^{\frac{-15\vare}{3-\alpha+20\vare}};
\] 
this estimate and~\eqref{eq:proj-thm-use} finish the proof of part~(1).
\end{proof}

\begin{proof}[Proof of Lemma~\ref{lem:CampHaus-again}]
Let $\bar \injr\leq 0.01$, and let $w_0\in B_\rfrak(0,\bar\injr)$.
Define the map $f:B_\rfrak(0,\bar\injr)\to B_\rfrak(0,2\bar\injr)$ by $f(w')=w$ where 
\[
h\exp(w)=\exp(w')\exp(-w_0)\quad\text{ with $h\in \mathsf B_{2\ref{E:BCH}\bar\injr^2}^H$ and $w\in B_{\rfrak}(0,2\bar\injr)$.}
\]

By the Baker-Campel-Hausdorff formula, see Lemma~\ref{lem:BCH}, $f$ is a diffeomorphism. Moreover, we have 
\[
\Bigl\|{\rm D}_{w'}\Bigl(f^{\pm1}\Bigr)-I\Bigr\|\leq 0.1
\]
for all $w'\in B_\rfrak(0,\bar\injr)$, in particular, ${\rm D}_{w'}(f^{\pm1})$ is invertible for all $w'\in B_\rfrak(0,\bar\injr)$. 

We conclude that $\#f(E)=\#E$, and 
\[
\# \Bigl(B_\rfrak(\bar w,\rhsc)\cap f(E)\Bigr)\leq \#\Bigl(B_\rfrak(f^{-1}(\bar w),2\rhsc)\cap E\Bigr)
\]
for all $\rhsc\leq \bar\eta$. The claim follows. 
\end{proof}

\begin{proof}[Proof of Lemma~\ref{lem:1pt-1-vare-trans}]
This is a consequence of the fact that the adjoint action of $H$ on $\rfrak$ is irreducible;
the argument below is based on explicit computations. 

Recall that $\|w\|=\max\{|w_{12}|, |w_{21}|, |w|_{21}\}$; moreover, recall that 
\be\label{eq:Ad-r-theta-12}
\Bigl(\Ad(\uvk)w\Bigr)_{12}=-w_{21}r^2-2w_{11}r+w_{12}.
\ee

Now if
\[
\#\{w\in E: |w_{12}|\geq 0.001\|w\|\}\geq \#E/4,
\]
then the claim holds with $r_0=0$. 

Therefore, we assume $\#\hat{E}\geq \frac{3\cdot(\#E)}{4}$ where $\hat{E}=\{w\in E: |w_{12}|\leq 0.001\|w\|\}$.
If
\[
\#\{w\in\hat{E}: |w_{11}|\geq 0.1\|w\|\}\geq \#E/4,
\]
then the claim holds with $r_0=0.1$ and the set on the left side of the above.

Therefore, we may assume 
\[
\#\{w\in\hat{E}: |w_{11}|\leq 0.1\|w\|\}\geq\#E/2.
\]
For every $w$ in the set on the left side of the above, $\|w\|=|w_{21}|$. 
The claim now holds with $r_0=0.9$ and the set on the left side of the above.
\end{proof}

\section{A closing lemma}\label{sec:closing-lemma}
For the proof of Theorem~\ref{thm:main}, one needs to guarantee that a certain initial separation is satisfied. 
This is the task in this section.        
This initial separation estimate is then bootstrapped in~\S\ref{sec:Marg-func-ini-dim} to give a better (finitary) dimension estimate that is used to conclude the theorem.
Throughout this section, {\em $\Gamma$ is assumed to be arithmetic}.
Indeed, this section is the only place 
where arithmeticity of $\Gamma$  is used in this paper, more specifically Lemma~\ref{lem:non-elementary}. Superficially arithmeticity is also used Lemma~\ref{lem:almost-inv}, but there the usage of arithmeticity is rather mild --- by local rigidity a lattice $\Gamma$ in $\SL(2,\bbc)$ or an irreducible lattice in $\SL(2,\bbr)\times\SL(2,\bbr)$ can be conjugated to have algebraic entries in some number field, which is good enough for our (relatively coarse) purposes. 

\medskip

Recall from~\eqref{eq:def-Ct} the definition 
\[
\coneH_{\eta, t,\beta}=\boxHs_{\beta}\cdot a_t\cdot \big\{u_r: r\in[0,\eta]\big\} \subset H;
\]
recall also that we always assume $\nuni^{-0.01t}<\beta<1$, and in this section we will be mainly interested in the case $\eta=1$; to simplify the notation, we will write $\coneH_t$ for~$\coneH_{1, t,\beta}$.

Let $x\in X$ and $t>0$. For every $z\in\coneH_t.x$, put
\be\label{eq:def-I-e-y}
\margI_t(z):=\Bigl\{w\in \rfrak: 0<\|w\|<\inj(z),\, \exp(w) z\in \coneH_t.x\Bigr\}.
\ee
Note that this is a finite subset of $\rfrak$. In~\eqref{eq:def-I-h-y-C}, we will define $I_\cone(h,z)$ 
for all $h\in H$ and more general sets $\cone$.

Let $0<\alpha<1$. Define the function $f_{t,\alpha}:\coneH_t.x\to [2,\infty)$ (which we will later use as a Margulis function in the bootstrap phase of the proof) as follows
\[
f_{t,\alpha}(z)=\begin{cases} \sum_{w\in I_t(z)}\|w\|^{-\alpha} & \text{if $\margI_t(z)\neq\emptyset $}\\
\inj(z)^{-\alpha}&\text{otherwise}
\end{cases}.
\]

The following is the main result of this section. 

\begin{propos}\label{prop:closing-lemma}
There exists $D_0$ (which depends explicitly on $\Gamma$) satisfying the following. 
Let $D\geq D_0+1$, and let $x_0\in X$. 
Then for all large enough $t$ (depending explicitly on $\inj(x_0)$ and $X$) 
at least one of the following holds.
\begin{enumerate}
\item  There is some $x\in X_{\rm cpt}\cap \{a_{8t}u_r.x_0: r\in[0,1]\}$ such that 
\begin{enumerate}
\item $\sfh\mapsto \sfh x$ is injective over $\coneH_{\rws}$.
\item For all $z\in\coneH_\rws.x$, we have 
\[
f_{t,\alpha}(z)\leq  \nuni^{D\rws}
\]
\end{enumerate}
for all $0<\alpha<1$. 

\item There is $x'\in X$ such that $H.x'$ is periodic with
\[
\vol(H.x')\leq \nuni^{D_0\rws}\quad\text{and}\quad\dist_X(x',x_0)\leq \nuni^{(-D+D_0)\rws}.
\] 
\end{enumerate} 
\end{propos}

The proof we give here is similar to that of Margulis and the first named author in~\cite[Lemma 5.2]{LM-Oppenheim}. A certain Diophantine condition (namely, {\em inheritable boundedness condition}) is used 
in the formulation of loc.\ cit.\ to guarantee in particular that our initial point is not close to a periodic $U$ orbit. We do not need such a condition here since 
we consider essentially translations of local $U$ orbits by expanding elements in $A$, and not long orbits of $U$ (this is reminiscent of a result of Nimish Shah \cite[Thm.~1.1]{Shah-Expanding}). 
As in \cite{LM-Oppenheim} the argument is elementary; a result of similar spirit to our Proposition~\ref{prop:closing-lemma} is proved by Einsiedler, Margulis, and Venkatesh in~\cite[Prop.~13.1]{EMV} using property-$\tau$, i.e.\ a uniform spectral gap.

Let us begin with some preliminary statements. 
In Proposition~\ref{prop:closing-lemma}, we are allowed to choose $t$ large depending on $\Gamma$.
Therefore, by passing to a finite index subgroup, we will assume that both of the following hold: {\em $\Gamma$ is torsion free} and if $\Gamma\subset\SL_2(\bbr)\times\SL_2(\bbr)$ is reducible, then $\Gamma=\Gamma_1\times\Gamma_2$

It is more convenient to consider $G$ as the set of $\lf$-points of an algebraic group defined over $\lf$ --- this way $H$ can be realized of as an algebraic subgroup of $G$. 
To that end, we let ${\bf G}=\SL_2\times \SL_2$ if $G=\SL_2(\lf)\times\SL_2(\lf)$. 
If $G=\SL_2(\qlf)$, we let ${\bf G}={\rm Res}_{\qlf/\lf}(\SL_2)$.
In either case, ${\bf G}$ is defined over $\lf$ and $G={\bf G}(\lf)$.

Recall that $\Gamma$ is assumed to be arithmetic. 
Therefore, there exists a semisimple $\bbq$-group $\tilde\G\subset\SL_M$, for some $M$, and an epimorphism 
$\rho:\tilde\G(\bbr)\to \G(\bbr)=G$ of $\bbr$-groups with compact kernel so that 
\be\label{eq: dimension of Gamma}
\text{$\Gamma$ is commensurable with $\rho(\tilde\G(\bbz))$}
\ee
where $\tilde\G(\bbz)=\tilde\G(\bbr)\cap\SL_M(\bbz)$.
Note that $\tilde\G$ can be chosen to be $\bbq$-almost simple unless $\Gamma\subset\SL_2(\bbr)\times\SL_2(\bbr)$ is a reducible lattice, in which case $\tilde\G$ can be chosen to have two $\bbq$-almost simple factors.

Let $\tilde\gfrak=\Lie(\tilde\G(\bbr))$, this Lie algebra has a natural $\bbq$-structure. 
Moreover, $\tilde\gfrak_\bbz:=\tilde\gfrak\cap\sl_M(\bbz)$ is a $\tilde\G(\bbz)$-stable lattice in $\tilde\gfrak$.

We continue to write $\Lie(G)=\gfrak$ and $\Lie(H)=\hfrak$; 
these are considered as $6$-dimensional (resp.\ $3$-dimensional) $\lf$-vector spaces. 

Let $v_H$ be a unit vector on the line $\wedge^3\hfrak$. Note that 
\[
N_G(H)=\{g\in G: gv_H=v_H\}
\] 
which contains $H$ as a subgroup of index two.

Recall also that we fixed a compact subset $\mathfrak S_{\rm cpt}\subset G$ which projects onto $X_{\rm cpt}$, see \S\ref{sec:SiegelSet} for the notation.

\begin{lemma}\label{lem:non-elementary}
There exist $\constE\label{E:non-el-1}$ and $\constk\label{k:non-el-2}$ depending on $M$ and $\mathfrak S_{\rm cpt}$, 
so that the following holds.  
Let $\gamma_1,\gamma_2\in\Gamma$ be two non-commuting elements. 
If $g\in \mathfrak S_{\rm cpt}$ is so that $\gamma_i g^{-1}v_H=g^{-1}v_H$ for $i=1,2$, then $Hg\Gamma$ is a closed orbit with 
\[
\vol(Hg\Gamma)\leq\ref{E:non-el-1}  \Bigl(\max\{\|\gamma_1^{\pm1}\|,\|\gamma_2^{\pm1}\|\}\Bigr)^{\ref{k:non-el-2}}.
\] 
\end{lemma} 

\begin{proof}
In view of our assumption in the lemma, we have 
\[
\langle \gamma_1,\gamma_2\rangle\subset {\rm Stab}_G(g^{-1}v_H)=N_G(g^{-1}Hg).
\]

Let $\Lambda_1:=\langle g\gamma_1g^{-1},g\gamma_2g^{-1}\rangle$.
We claim that $\Lambda:=\Lambda_1 \cap H$ is Zariski dense in $H$. Indeed since $\langle \gamma_1,\gamma_2\rangle$ is a torsion free, non-commutative, discrete subgroup of $N_G(g^{-1}Hg)$, we have  $\Lambda$ is discrete and torsion free. This and the fact that $H\simeq \SL_2(\bbr)$ imply that if $\Lambda$ is non-commutative, then it is Zariski dense in $H$. 
Assume thus that $\Lambda$ is commutative, which implies that $\Lambda\simeq \bbz$ and that $\Lambda\subsetneq\Lambda_1$ 
(recall that $\Lambda_1$ is non-commutative).    
Since $N_G(H)=HC$ where $C$ is the center of $G$ if $G=\SL_2(\bbr)\times\SL_2(\bbr)$ and $C=\langle{\rm diag}(i,-i)\rangle$ if $G=\SL_2(\bbc)$, we have $N_G(H)/H\simeq \bbz/2\bbz$; thus  $\Lambda_1/\Lambda\simeq \bbz/2\bbz$. 
This implies that $\Lambda_1$ is isomorphic to $\bbz$ or $\bbz\times\bbz/2\bbz$ or $\bbz/2\bbz\ltimes\bbz$. 
Either possibility leads to a contradiction to $\Lambda_1$ being non-commutative and torsion free.

Let $\bf L$ be the Zariski closure  of $\langle \gamma_1,\gamma_2\rangle$. In view of the above discussion, 
\be\label{eq:L-gHg-1}
g^{-1}Hg\subset {\bf L}(\bbr)\subset N_G(g^{-1}Hg).
\ee 
Since $N_G(H)/H\simeq\bbz/2\bbz$, replacing $\gamma_i$ by $\gamma_i^2$ if necessary we assume that 
${\bf L}(\bbr)=g^{-1}Hg$. 

Let $\tilde\gamma_i\in\tilde\G(\bbz)$ be so that $\rho(\tilde\gamma_i)=\gamma_i$. 
Then the Zariski closure $\tilde{\bf L}$ of $\langle \tilde\gamma_1,\tilde\gamma_2\rangle$ is semisimple and 
$\rho(\tilde{\bf L}(\bbr))={\bf L}(\bbr)$. Therefore, in view of a theorem of Borel and Harish-Chandra~\cite[Thm.~7.8]{Borel-Harish-Chandra}, we have 
$\tilde{\bf L}(\bbr)\cap \tilde\G(\bbz)$ is a lattice in $\tilde{\bf L}(\bbr)$. 

This implies that ${\bf L}(\bbr)\Gamma$ is a periodic orbit, which in view of~\eqref{eq:L-gHg-1}
implies that $Hg\Gamma$ is a periodic orbit. 

\medskip

We now turn to the proof of the second claim. 
Let $\tilde{\mathfrak l}=\Lie(\tilde{\bf L}(\bbr))\subset \tilde\gfrak$. 
Then $\tilde{\mathfrak l}$ is a rational subspace of $\tilde\gfrak$; we will show that 
the height of this subspace is $\ll \Theta^\star$ where $\Theta:=\max\{\|\gamma_1^{\pm1}\|,\|\gamma_2^{\pm1}\|\}$. 
That is to say: $\tilde{\mathfrak l}$ has a basis consisting of vectors in $\tilde\gfrak_\bbz\cap\tilde{\mathfrak l}$ 
with norm $\ll \Theta^\star$, e.g., by Minkowski's second theorem.

Indeed by Chevalley's theorem and the fact that $\tilde{\bf L}(\bbr)$ is semisimple (hence it has no character),   
there exists a finite dimensional $\bbq$-representation of $\tilde\G$ on a space $\Phi$ with the following property.
Let $\Phi^0$ denote the vectors in $\Phi_\bbr$ which are fixed by $\tilde{\bf L}(\bbr)$, then 
\[
\tilde{\bf L}(\bbr)=\{g\in\tilde\G(\bbr): g.q=q,\text{ for all $q\in\Phi^0$}\};
\]
in terms of the Lie algebras, this is $\tilde{\mathfrak l}=\{w\in\tilde\gfrak: w.\Phi^0=0\}$.

Since $\langle \tilde\gamma_1,\tilde\gamma_2\rangle$ is Zariski dense in $\tilde{\bf L}$, we conclude that  
$\Phi^0$ is a rational subspace with height $\ll (\max\{\|\tilde\gamma_1^{\pm1}\|,\|\tilde\gamma_2^{\pm1}\|\})^\star\ll\Theta^\star$; we used the fact
that $\rho(\tilde\gamma_i)=\gamma_i$ to bound $\|\tilde\gamma_i^{\pm1}\|$ from above by 
$\|\gamma_i^{\pm1}\|^\star$ for $i=1,2$. 

Using this and the fact that $\tilde{\mathfrak l}=\{w\in\tilde\gfrak: w.\Phi^0=0\}$, we conclude 
that height of $\tilde{\mathfrak l}$ is $\ll \Theta^\star$ as we claimed. 
This height bound implies that 
\[
\vol\Big(\tilde{\bf L}(\bbr)\tilde\G(\bbz)\Big)\ll \Theta^\star.
\]
see e.g.~\cite[\S 17]{EMV}, or~\cite[App.\ B]{EMMV} (see also~\cite[\S2]{ELMV-1}, which treats the case of tori; the proof there works for the semisimple case as well).

We deduce that $\vol({\bf L}(\bbr)\Gamma)\ll \Theta^\star$; 
recall that the kernel of $\rho$ is compact and $\mathbf {L}(\bbr)=\rho(\tilde{\mathbf{L}}(\bbr))$. 
The claimed bound on $\vol(Hg\Gamma)$ now follows in view of~\eqref{eq:L-gHg-1} and the fact that $g\in\mathfrak S_{\rm cpt}$. 
\end{proof}

We also need the following lemma.

\begin{lemma}\label{lem:almost-inv}
There exist $\constk\label{k:Eq-proj}$, $\constk\label{k:Eq-proj-2}$, and $\constE\label{E:Eq-proj-mul}$ 
so that the following holds. Let $\gamma_1,\gamma_2\in\Gamma$ be two non-commuting elements,
and let 
\[
\delta\leq  \ref{E:Eq-proj-mul}^{-1}\Big(\max\{\|\gamma_1^{\pm1}\|, \|\gamma_2^{\pm1}\|\}\Big)^{-\ref{k:Eq-proj}}.
\]
Suppose there exists some $g\in \mathfrak S_{\rm cpt}$ so that $\gamma_i g^{-1}v_H=\epsilon_ig^{-1}v_H$ for $i=1,2$ 
where $\|\epsilon_i-I\|\leq \delta$. 
Then, there is some $g'\in G$ such that 
\[
\|g'-g^{-1}\|\leq \ref{E:Eq-proj-mul} \delta  \Big(\max\{\|\gamma_1^{\pm1}\|, \|\gamma_2^{\pm1}\|\}\Big)^{\ref{k:Eq-proj-2}}
\] 
and $\gamma_ig'v_H=g'v_H$ for $i=1,2$.

\end{lemma}

\begin{proof}
This is essentially proved in~\cite[\S13.3, \S13.4]{EMV}, we recall parts of the argument for the convenience of the reader.

With a slight change in the notation from the proof of the previous lemma, 
let $\tilde{\bf L}$ be the $\bbr$-group defined by $\tilde{\bf L}(\bbr)=\rho^{-1}(g^{-1}Hg)\subset\tilde\G(\bbr)$, 
and let $d=\dim(\tilde{\bf L}(\bbr))$. Fix a unit vector $v_0$ on the line $\wedge^d(\Lie(\tilde{\bf L}(\bbr)))$.

Let also $\tilde\gamma_i\in\tilde\G(\bbz)$ be so that $\rho(\tilde\gamma_i)=\gamma_i$, for $i=1,2$. 
Then~\cite[Lemma 13.1]{EMV} holds true for linear transformation 
\[
A=(\tilde\gamma_1-I)\oplus(\tilde\gamma_2-I)
\] 
from $\wedge^d\tilde\gfrak$ to $\wedge^d\tilde\gfrak\oplus \wedge^d\tilde\gfrak$. 
Therefore, there exists a vector $w\in \wedge^d\tilde\gfrak$, with 
\be\label{eq:almker-ker}
\|w-v_0\|\leq C\Theta^\kappa\delta
\ee
so that $Aw=0$, where $\Theta:=\max\{\|{\gamma}_1^{\pm1}\|, \|{\gamma}_2^{\pm1}\|\}$, $C$ depends on $\tilde\G$ and $\kappa$ depends on $\dim\tilde\G$. We again used $\rho(\tilde\gamma_i)=\gamma_i$ to bound $\|\tilde\gamma_i^{\pm1}\|$ by a power of 
$\|\gamma_i^{\pm1}\|$. 

This implies that $\tilde\gamma_iw=w$ for $i=1,2$.
By~\cite[Lemma 13.2]{EMV}, there exist $\bar C$ and $\bar \kappa\geq1$ so that if 
\[
\|w-v_0\|\leq \bar C^{-1}\Theta^{-\bar\kappa},
\]
then there exists $\tilde g\in\tilde\G(\bbr)$ satisfying that $\|\tilde g-I\|\leq C'\|w-v_0\|$ and 
\[
\text{$\tilde\gamma_i\tilde gv_0=\tilde gv_0$ for $i=1,2$,}
\] 
see~\cite{EquiProj-Luna} for sharper results concerning equivariant projections.   

Let now $\delta$ satisfy
\[
0<\delta\leq (C\bar C)^{-1}\Theta^{-\kappa'-\kappa}.
\]
Then~\eqref{eq:almker-ker} implies that there exists some $\tilde g\in \tilde\G(\bbr)$ with 
$\|\tilde g-I\|\leq C'C\Theta^\kappa\delta$
so that $\tilde\gamma_i\tilde gv_0=\tilde gv_0$ for $i=1,2$.
This estimate implies that 
\[
\|\rho(\tilde g)g^{-1}-g^{-1}\|\leq C''\Theta^\kappa\delta
\]
for some $C''$ depending on $\tilde\G$.

Let $g'=\rho(\tilde g)g^{-1}$. Then $\gamma_ig'v_H=g'v_H$ and the claim holds for $g'v_H$.
\end{proof}

We need the following lemma, see Lemma~\ref{lem:noI-tri-bd} in the sequel for a more general statement.

\begin{lemma}\label{lem: explain 4t}
Let $x\in X_{\rm cpt}$. Then for every $z\in\coneH_t.x$, we have \[\#\margI_\rws(z)\ll \nuni^{4\rws}.\] 
\end{lemma}

\noindent
For the convenience of the reader, we recall from \eqref{eq:def-I-e-y} that
\[
\margI_t(z):=\Bigl\{w\in \rfrak: 0<\|w\|<\inj(z),\, \exp(w) z\in \coneH_t.x\Bigr\}.
\]
\begin{proof}
Recall from~\eqref{eq:def-inj} that 
\[
\inj(z)=\min\Big\{0.01, \sup\Big\{\vare: \text{ $g\mapsto gz$ is injective on $\boxG_{10\vare}$}\Big\}\Big\}
\]
where for every $0<\vare\leq0.1$, we put $\boxG_{\vare}:=\boxH_\vare\cdot\exp(B_\rfrak(0,\vare))$.

Note that since $x\in X_{\rm cpt}$, we have 
\be\label{eq: inj sfh x}
\inj(\sfh x)> 10ce^{-t}\quad\text{for all $\sfh\in\coneH_t$}
\ee
where $c$ depends only on $X$. 

Let $z\in\coneH_t.x$ and $w\in\margI_{t}(z)$ (hence $\exp(w)z\in \coneH_t.x$). Therefore,
\[
\boxH_{ce^{-t}}\exp(w)z\subset \coneH_{t+}.x
\]
where we define $\coneH_{t+}=\boxH_{\beta+2ce^{-t}}\cdot\coneH_t$. 

In view of~\eqref{eq: inj sfh x} and the definition of $\inj(z)$, the map $(\sfh, w)\mapsto \sfh\exp(w)z$ is injective over 
$\boxH_{ce^{-t}}\times \exp(B_\rfrak(0,\inj(z)))$. Hence we have
\[
\boxH_{ce^{-t}}\exp(w)z\cap \boxH_{ce^{-t}}\exp(w')z=\emptyset \qquad\text{for all distinct $w,w'\in \margI_{t}(z)$}.
\]
Since $m_H(\coneH_{t+})\ll e^{t}$ and $m_H(\boxH_{ce^{-t}})\gg e^{-3t}$, the claim follows. 
\end{proof}

\begin{proof}[Proof of Proposition~\ref{prop:closing-lemma}]
By Proposition~\ref{prop:non-div} if $d\geq |\log (10^{-6}\inj(y))|+\ref{E:non-div-main}$, then  
\be\label{eq:cpct-return}
|\{r\in J: a_{d} u_ry\in X_{\rm cpt}\}|\geq 0.99|J|
\ee
for all $J\subset [0,1]$ with $|J|\geq 10^{-3}$.

Let $\rws\geq |\log(10^{-6}\inj(x_0))|+\ref{E:non-div-main}$ for the rest of the argument. 
Let $r_0\in[0,1/2]$ be so that $x_1=a_tu_{r_0}x_0$ satisfies both of the following: $x_1\in X_{\rm cpt}$ and $a_{7t}x_1\in X_{\rm cpt}$. Write $x_1=g_1\Gamma$ where $g_1\in \mathfrak S_{\rm cpt}$. 

We introduce the shorthand notation $h_r:=a_{7t}u_r$, for any $r\in[0,1]$. Note that for all $r\in[0,1]$, we have $h_rx_1\in \{a_{8t}u_{r'}x_0: r'\in [0,1]\}$. Assume now the claim in part~(1) fails for all $r\in[0,1]$ so that $h_rx_1\in X_{\rm cpt}$. 
That is: for all $r\in[0,1]$ so that $h_rx_1\in X_{\rm cpt}$ 
\begin{itemize}
\item either there exists $z\in \coneH_\rws.h_rx_1$ so that $f_{\rws,\alpha}(z)>\nuni^{D\rws}$,
\item or the map $\sfh\mapsto \sfh h_rx_1$ is not injective on $\coneH_{\rws}$.
\end{itemize}

In what follows all the implied multiplicative constants depend only on~$X$. 

\subsection*{Finding lattice elements $\gamma_r$}
Let us first investigate the former situation. That is: fix $r\in[0,1]$ so that $h_rx_1\in X_{\rm cpt}$ and suppose that for some $z=\sfh_1h_rx_1\in \coneH_\rws.h_rx_1$,
it holds that $f_{\rws,\alpha}(z)>\nuni^{D\rws}$. 
Since $h_rx_1\in X_{\rm cpt}$, we have  
\be\label{eq:Ct-cusp}
\inj(\sfh h_rx_1)\gg \nuni^{-\rws}, \quad\text{for all $\sfh\in\coneH_\rws$}.
\ee
Using the definition of $f_{\rws,\alpha}$, thus, we conclude that 
if $\margI_\rws(z)=\emptyset$, then $f_{t,\alpha}(z)\ll \nuni^\rws$. Hence, assuming $t$ is large enough, $\margI_\rws(z)\neq\emptyset$; recall also from Lemma~\ref{lem: explain 4t} that $\#\margI_\rws(z)\ll \nuni^{4\rws}$.

Altogether, if $D\geq5$ and $t$ is large enough, there exists some $w\in I_\rws(z)$ with 
\[
0<\|w\|\leq \nuni^{(-D+5)\rws}.
\]

The above implies that for some $w\in \rfrak$ with $\|w\|\leq \nuni^{(-D+5)\rws}$ 
and $\sfh_1\neq \sfh_2\in\coneH_\rws$, we have $\exp(w)\sfh_1h_rx_1=\sfh_2h_rx_1$.
Thus 
\be\label{eq:wh-sh}
\exp(w_r)h_r^{-1}\sfs_rh_rx_1=x_1
\ee 
where $\sfs_r=\sfh_2^{-1}\sfh_1$, $w_r=\Ad(h_r^{-1}\sfh_2^{-1})w$. 
In particular, $\|w_r\|\ll \nuni^{(-D+13)\rws}$.
Assuming $t$ is large enough compared to the implied multiplicative constant,
\be\label{eq:wh-est}
0<\|w_r\|\leq \nuni^{(-D+14)\rws}.
\ee
Recall that $x_1=g_1\Gamma$ where $g_1\in\mathfrak S_{\rm cpt}$, thus,~\eqref{eq:wh-sh} implies 
\be\label{eq:gamma-h}
\exp(w_r)h_r^{-1}\sfs_rh_r=g_1\gamma_rg_1^{-1}
\ee
where $1\neq\sfs_r\in H$ with $\|\sfs_r\|\ll \nuni^{\rws}$ and $e\neq\gamma_r\in \Gamma$.

Similarly, if $\sfh\mapsto \sfh h_rx_1$ is not injective, we conclude that
\[
h_r^{-1}\sfs_r h_r=g_1\gamma_rg_1^{-1}\neq e.
\]
In this case we actually have $e\neq \gamma_r\in g_1^{-1}Hg_1$ --- we will not use this extra information in what follows.  

\subsection*{Some properties of the elements $\gamma_r$}
Note that, in either case, we have 
\be\label{eq:size-gammah}
\|\gamma_r^{\pm1}\|\leq \nuni^{9\rws}
\ee
again we assumed $\rws$ is large compared to $\|g_1\|$ hence the estimate $\ll \nuni^{8\rws}$ 
is replaced by $\leq \nuni^{9t}$.

Let $\xi>0$ be so that $\|g\gamma g^{-1}-I\|\geq 20\xi$ for all $\gamma\in\Gamma\setminus \{1\}$ and $g\in\mathfrak{S}_{\rm cpt}$.   
Write $\sfs_r=\begin{pmatrix} a_1 & a_2 \\ a_3 & a_4 \end{pmatrix}\in H$ where $|a_i|\leq 10\nuni^{t}$.
Then by~\eqref{eq:gamma-h}, we have 
\[
\|h_r^{-1}\sfs_r h_r-I\|=\biggl\|u_{-r}\begin{pmatrix} a_1 & e^{-7t}a_2 \\ e^{7t}a_3 & a_4 \end{pmatrix}u_r-I\biggr\|\geq 10\xi
\]
which implies that 
\be\label{eq: closing not unipotent}
\max\{e^{7t}|a_3|, |a_1-1|, |a_4-1|\}\geq \xi \gg 1.
\ee
Note also that if $e^{7t}|a_3|<\xi$, then $|a_2a_3|\leq10\xi e^{-6t}$, thus $|a_1a_4-1|\ll e^{-6t}$. We conclude from~\eqref{eq: closing not unipotent} that $|a_1-a_4|\gg 1$. Altogether, 
\be\label{eq: a1-a4}
\max\{e^{7t}|a_3|, |a_1-a_4|\}\gg 1.
\ee
 
 \medskip

Let $I_{\rm cpt} = \{r\in[0,1]: h_rx_1\in X_{\rm cpt}\}$ and $J_{\rm cpt}=\{r\in[1/2,1]: h_rx_1\in X_{\rm cpt}\}$.

\subsection*{Claim:} There are $\gg e^{3t}$ distinct elements in $\{\gamma_r: r\in J_{\rm cpt}\}$.

 By~\eqref{eq:cpct-return} applied with $y=x_1$, $d=7t$, and $J=[1/2,1]$ we have $|J_{\rm cpt}|\geq 1/4$ (assuming $t$ is large enough).  
Fix $r\in J_{\rm cpt}$ as above, and consider the set of $r' \in J_{\rm cpt}$ so that and $\gamma_r=\gamma_{r'}$.
Then for each such $r'$, 
\begin{align*}
h_r^{-1}\sfs_r h_r&=\exp(-w_r)g_1\gamma_rg_1^{-1}=\exp(-w_r)\exp(w_{r'})h_{r'}^{-1}\sfs_{r'} h_{r'}\\
&=\exp(w_{rr'})h_{r'}^{-1}\sfs_{r'} h_{r'}
\end{align*}
where $w_{rr'}\in \gfrak$ and $\|w_{rr'}\|\ll \nuni^{(-D+14)\rws}$. 

Set $\tau=e^{7t}(r'-r)$. Assuming $D\geq30$, we conclude that
\begin{equation}\label{eq: u_tau equation}
  u_{\tau}\sfs_r u_{-\tau}=h_{r'}h_r^{-1}\,\sfs_r\, h_rh_{r'}^{-1}=\exp(\hat w_{rr'})\sfs_{r'}
\end{equation}
where $\|\hat w_{rr'}\|=\|\Ad(h_{r'})w_{rr'}\|\ll \nuni^{(-D+21)}$.

Finally, we compute 
\[
u_{\tau}\sfs_r u_{-\tau}=\begin{pmatrix} a_1+a_3\tau& a_2+(a_4-a_1)\tau-a_3\tau^2  \\  a_3& a_4-a_3\tau\end{pmatrix}.
\]

In view of~\eqref{eq: a1-a4}, for every $r\in J_{\rm cpt}$ the set of $r'\in J_{\rm cpt}$ so that 
\begin{equation}\label{eq: define J_r}
  |a_2e^{-7t}+(a_4-a_1)(r'-r)-a_3e^{7t}(r'-r)^2|\leq 10^{4} e^{-6t}  
\end{equation}
has measure $\ll e^{-3t}$ since at least one of the coefficients of this quadratic polynomial is of size $\gg1$. Let $J_r$ be the set of $r'\in J_{\rm cpt}$ for which \eqref{eq: define J_r} holds.

If $r'\in J_{\rm cpt} \setminus J_{r}$,
then $|a_2+(a_4-a_1)\tau-a_3\tau^2|>10^{4}e^{t}$ (recall that $\tau=e^{7t}(r'-r)$), thus for all $r'\in J_{\rm cpt} \setminus J_{r}$, we have 
\[
\|u_{\tau}\sfs_r u_{-\tau}\|> 10^4e^{t}> \|\exp(\hat w_{rr'})\sfs_{r'}\|,
\]
in contradiction to \eqref{eq: u_tau equation}.

In other words, for each $\gamma \in \Gamma$ the set of $r \in J_{\rm cpt}$ for which $\gamma_r = \gamma$ has measure $\ll e^{-3t}$ and so the set $\{\gamma_r : r \in J_{\rm cpt}\}$ has at least $\gg e^{3t}$ distinct elements, establishing the claim.

\subsection*{Zariski closure of the group generated by $\{\gamma_r : r \in I_{\rm cpt}\}$} {$\,$}
\\

\noindent
We now consider two possibilities for the elements $\{\gamma_r  : r \in I_{\rm cpt}\}$. 

\subsection*{Case 1} The family $\{\gamma_r: r \in I_{\rm cpt}\}$ is commutative. \\

Let ${\bf L}$ denote the Zariski closure of $\langle \gamma_{r}: r\in I_{\rm cpt}\rangle$. Since $\langle \gamma_{r}\rangle$ is commutative, so is~${\bf L}$. Let $C_{\bf G}$ denote the center of $\bf G$.
We claim that ${\bf L}={\bf L}'{\bf C'}$ where ${\bf C}'\subset C_{\bf G}$ and ${\bf L}'$ is either a unipotent group or a torus. Indeed since ${\bf L}$ is commutative, we have ${\bf L}={\bf T}{\bf V}$ where $\bf T$ is a (possibly finite) algebraic subgroup of a torus, $\bf V$ is a unipotent group and ${\bf T}$ and $\bf V$ commute. Therefore, if both $\bf T$ and ${\bf V}$ are non-central, then  $G=\SL_2(\bbr)\times\SL_2(\bbr)$ and $\Gamma=\Gamma_1\times\Gamma_2$ is reducible.
Moreover, ${\bf T}\subset {\bf T}' C_{\bf G}$ where ${\bf T}'$ is an algebraic subgroup of a torus, and ${\bf T}'$ and ${\bf V}$ belong to different $\SL_2(\bbr)$ factors in $G$. Let us assume $\bf V$ belongs to the second factor. Recall from~\eqref{eq:wh-sh} that 
\be\label{eq:wh-sh'}
\exp(w_r)h_r^{-1}\sfs_rh_r=g_1\gamma_rg_1^{-1}
\ee
where $\|w_r\|\leq e^{(-D+14)t}$ with $D\geq 30$ and 
$h_r^{-1}\sfs_rh_r\in H=\{(h,h): h\in \SL_2(\bbr)\}$. 
Now if $\gamma_r=(\gamma_r^1,\gamma_r^2)$, then~\eqref{eq:wh-sh'} together with the bound $\|h_r^{-1}\sfs_rh_r\|\ll e^{8t}$ implies that  $|{\rm tr}(\gamma_r^1)-{\rm tr}(\gamma_r^2)|\ll e^{(-D+22)t}$; moreover, since $\gamma_r^2\in{\bf V} C_{\bf G}$, we have $|{\rm tr}(\gamma_r^2)|=2$. This and the fact that the length of closed geodesics in (finite volume) hyperbolic surfaces is bounded away from zero imply that $|{\rm tr}(\gamma_r^1)|=2$ if $t$ is large enough. This contradicts the fact that $\bf T$ is a non-central subgroup of a torus. Hence, the claim holds.

We now show that ${\bf L}'$ is indeed a unipotent group.
In view of the above discussion, $\#\{\gamma_r: r\in J_{\rm cpt}\}\geq e^{3t}$.
Note also that that for every torus $T\subset G$, we have  
\[
\#(B_T(e,R)\cap \Gamma)\ll (\log R)^2,
\]
where the implied constant is absolute. These, in view of the bound $\|\gamma_{r}\|\leq e^{9t}$, see~\eqref{eq:size-gammah}, 
imply that ${\bf L}'$ is unipotent.

Since ${\bf L}'$ is a unipotent subgroup of $\bf G$, we have that
\[\#\{\gamma_r: \|\gamma_r\|\leq e^{4t/3}\}\ll e^{8t/3}.\] 
Furthermore, there are $\gg e^{3t}$ distinct elements $\gamma_r$ with $r\in J_{\rm cpt}$.
Thus  
\[
\#\{\gamma_r : \|\gamma_r\|>100e^{4t/3} \text{ and } r\in J_{\rm cpt}\}\gg e^{3t}.
\]

For every $r\in I_{\rm cpt}$, write 
\[
\sfs_{r}=\begin{pmatrix} a_{1,r} & a_{2,r} \\ a_{3,r} & a_{4,r} \end{pmatrix}\in H
\] 
where $|a_{j,r}|\leq 10\nuni^{t}$. 

We will obtain an improvement of~\eqref{eq: closing not unipotent}.  
Let $\xi\leq \Upsilon\leq e^{4t/3}$ and assume that $\|g_1\gamma_rg_1^{-1}-I\|\geq 20 \Upsilon$ --- by definition of $\xi$, this holds with $\Upsilon=\xi$ for all $r\in I_{\rm cpt}$ and as we have just seen this also holds for with $\Upsilon=e^{4t/3}$ for many choices of $r\in J_{\rm cpt}$.
We claim 
\begin{equation}\label{eq: lower bound a(3,r)}
    |a_{3,r}|\geq \Upsilon e^{-7t}.
\end{equation}
Indeed by~\eqref{eq:gamma-h}, we have 
\[
\|h_r^{-1}\sfs_r h_r-I\|=\biggl\|u_{-r}\begin{pmatrix} a_{1,r} & e^{-7t}a_{2,r} \\ e^{7t}a_{3,r} & a_{4,r} \end{pmatrix}u_r-I\biggr\|\geq 10 \Upsilon.
\]
This implies that $\max\{e^{7t}|a_{3,r}|, |a_{1,r}-1|, |a_{3,r}-1|\}\geq \Upsilon$. Assume contrary to our claim that $|a_{3,r}|<\Upsilon e^{-7t}$. 
Then 
\begin{equation}\label{eq:max1,4}
    \max\{|a_{1,r}-1|, |a_{4,r}-1|\}\geq \Upsilon;
\end{equation}
furthermore, we get $|a_{2,r}a_{3,r}|\ll \Upsilon e^{-6t}$. 
Thus, 
\begin{equation}\label{eq:max1*4}
    |a_{1,r}a_{4,r}-1|\ll \Upsilon e^{-6t}\ll e^{-14t/3}.
\end{equation}
Moreover, since $h_r^{-1}\sfs_r h_r$ is very nearly $g_1\gamma_{r}g_1^{-1}$, and the latter is either a unipotent element or its minus, we conclude that 
\begin{equation}\label{eq:trace min}
    \min(|a_{1,r}+a_{4,r}-2|, |a_{1,r}+a_{4,r}+2|)\ll e^{(-D+22)t}.
\end{equation}
Equations~\eqref{eq:max1*4} and~\eqref{eq:trace min} contradict~\eqref{eq:max1,4} if $t$ is large enough, hence necessarily $|a_{3,r}|\geq \Upsilon e^{-7t}$.

Using this, we now show that Case 1 cannot occur. 
Since ${\bf L}'$ is unipotent, there exists some $g$  so that 
${\bf L}'(\bbr)\subset gNg^{-1}$; moreover $g$ can be chosen to be in the maximal compact subgroup of $G$ --- for our purposes, we only need to know that the size of $g$ can be bounded by an absolute constant. 

It follows that
\be\label{eq: us sr and N}
u_{-r}\begin{pmatrix} a_{1,r} & e^{-7t}a_{2,r} \\ e^{7t}a_{3,r} & a_{4,r} \end{pmatrix}u_r\in \exp(-w_r) (gNg^{-1})\cdot C_{\bf G}
\ee
for all $r\in I_{\rm cpt}$.  
We show that this leads to a contradiction when $G=\SL_2(\bbc)$, the proof in the other case is similar by considering first and second coordinates. 

Let us write $g=\begin{pmatrix} a & b \\ c & d \end{pmatrix}$, 
then for all $z\in\bbc$ we have  
\[
g\begin{pmatrix} 1 & z \\ 0 & 1 \end{pmatrix}g^{-1}= \begin{pmatrix}1-acz& a^2z\\ -c^2z & 1+acz\end{pmatrix}.
\]
Recall from the beginning of the proof that $h_0x_1\in X_{\rm cpt}$, i.e., $0\in I_{\rm cpt}$.
It follows that 
for some $z_0\in\bbc$, 
\[
\begin{pmatrix} a_{1,0} & e^{-7t}a_{2,0} \\ e^{7t}a_{3,0} & a_{4,0}\end{pmatrix}=\pm\exp(-w_r)\begin{pmatrix}1-acz_0& a^2z_0\\ -c^2z_0 & 1+acz_0\end{pmatrix}.
\] 
By \eqref{eq: lower bound a(3,r)} applied with $\Upsilon=\xi$, $|{a_{3,0}}|\geq \xi e^{-7t}$. Since $|a|, |b|, |c|, |d|\ll1$, comparing the bottom left entries of the matrices we get $|z_0|\gg 1$. 
Now, since $|a_{2,0}|\leq 10e^{t}$, comparing the top right entries we conclude that $|a|\ll e^{-3t}$. Since $\det(g)=1$, it follows that  $|c|$ is also $\gg 1$.  

Let now $r\in J_{\rm cpt}$ be so that $\|\gamma_r\|\geq 100e^{4t/3}$. We write $a'_{2,r}=e^{-7t}a_{2,r}$ and $a'_{3,r}=e^{7t}a_{3,r}$. By~\eqref{eq: lower bound a(3,r)}, applied this time with $\Upsilon=e^{4t/3}$, we have that $|a'_{3,r}|\geq e^{4t/3}$; note also that $|a'_{2,r}|\ll e^{-6t}$. In view of~\eqref{eq: us sr and N}, there exists $z_r\in\bbc$ so that 
\begin{align*}
u_{-r}\begin{pmatrix} a_{1,r} & a'_{2,r} \\ a'_{3,r} & a_{4,r} \end{pmatrix}u_r&=\begin{pmatrix} a_{1,r}-ra'_{3,r}  & a'_{2,r}+(a_{4,r}-a_{1,r})r-a'_{3,r}r^{2} \\ a'_{3,r} & a_{4,r}+ra'_{3,r} \end{pmatrix}\\
&=\pm\exp(-w_r)\begin{pmatrix}1-acz_r& a^2z_r\\ -c^2z_r & 1+acz_r\end{pmatrix}.
\end{align*}
Since $|a'_{3,r}|\geq e^{4t/3}$, $|a_{1,r}|$ and $|a_{4,r}|$ are $\ll e^{t}$, and $|a'_{2,r}|\ll e^{-6t}$,
and since $r \in [\frac12,1]$, we have that
\[
|a'_{3,r}|/10\leq |a'_{2,r}+(a_{4,r}-a_{1,r})r-a'_{3,r}r^{2}|\leq 2|a'_{3,r}|;
\]
hence, since $w_r$ is small, $a^2z_r$ and $c^2z_r$ should be comparable in size. On the other hand, using $r=0$ we already established $|a| \ll e^{-3t}$ and $|c|\gg1$, thus $|a^2z_r|\ll e^{-3t}|c^2z_r|$, in contradiction.

Altogether, we conclude that Case~1 cannot occur.

\subsection*{Case 2} There are $r,r'\in I_{\rm cpt}$ so that $\gamma_r$ and $\gamma_{r'}$ do not commute. \\ 

Let $v_H$ be as in Lemma~\ref{lem:almost-inv}. Then since $\exp(w_r)h_r^{-1}\sfs_rh_r=g_1\gamma_rg_1^{-1}$
\[
\gamma_r .g_1^{-1}v_H=\exp(\Ad(g_1^{-1})w_r).g_1^{-1}v_H.
\]
Moreover, since $\|w_r\|\leq \nuni^{(-D+14)\rws}$,  
\[
\|\Ad(g_1^{-1})w_r\|\ll \nuni^{(-D+14)\rws};
\]
similar statements also hold for $r'$.

Therefore, if $D$ is large enough, we may apply Lemma~\ref{lem:almost-inv} to conclude that there exists some $g_2\in G$ with 
\[
\|g_1-g_2\|\leq \ref{E:Eq-proj-mul}\nuni^{(-D+14+9\ref{k:Eq-proj-2})\rws},
\] 
so that $\gamma_r .g_2^{-1}v_H=g_2^{-1}v_H$ and 
$\gamma_{r'} .g_2^{-1}v_H=g_2^{-1}v_H$.

In view of Lemma~\ref{lem:non-elementary}, thus, we have $Hg_2\Gamma$ is periodic and 
\[
\vol(Hg_2\Gamma)\leq \ref{E:non-el-1}\Bigl(\max\{\|\gamma_r^{\pm1}\|,\|\gamma_{r'}^{\pm1}\|\}\Bigr)^{\ref{k:non-el-2}}\leq \ref{E:non-el-1}\nuni^{9\ref{k:non-el-2}\rws},
\]
where we used~$\|\gamma_r^{\pm1}\|,\|\gamma_{r'}^{\pm1}\|\leq e^{9t}$.

Then for $t$ large enough,  $\vol(Hg_2\Gamma)\leq\nuni^{D'_0\rws}$ and $d_X(g_1\Gamma,g_2\Gamma)\ll e^{(-D+D_0')t}$ for $D'_0=9\max\{\ref{k:non-el-2}, \ref{k:Eq-proj-2}\}+14$. 

Since $g_1\Gamma=x_1=a_tu_{r_0}x_0$, part~(2) in the proposition holds with $x'=(a_tu_{r_0})^{-1}g_2\Gamma$ and 
$D_0=\max\{D'_0+2, 30\}$ if $t$ is large enough (recall that we already assumed in several places that $D \geq 30$). 
\end{proof}

\section{Margulis functions and random walks}\label{sec:Marg-func-ini-dim}
As was mentioned earlier, the proof of Proposition~\ref{prop:main-prop} relies on two main ingredients: 
evolutions of Margulis functions under a certain random walk, 
and the (finitary) projection theorem, specifically Proposition~\ref{prop:proj-general}, proved in~\S\ref{sec:Mars-proj}. In this section we develop the necessary Margulis function techniques and show how to combine them with the results of \S\ref{sec:Mars-proj} to prove Theorem~\ref{thm:main} in \S\ref{sec:proof-main}.

The following is the main proposition encapsulating what is obtained using Margulis function techniques (and then input into Proposition~\ref{prop:proj-general}).

\begin{propos}\label{prop:dim-1-C-rfrak}\label{prop:dim-1-e}
Let $0<\eta<0.01\eta_X$, $D\geq D_0+1$, and $x_0\in X$, where $D_0$ is as in Proposition~\ref{prop:closing-lemma}, and $\eta_X$ as in~Proposition~\ref{prop:non-div}.
Then there exists $t_0$, depending on $\eta$, $\inj(x_0)$, and $X$, so that if $t\geq t_0$, then at least one of the following holds:

\begin{enumerate}
\item Let $0<\vare<0.1$ and $0<\alpha<1$. 
Then there exist $x_1\in X_{\injr}$, some $\tau$ with
$9t\leq \tau\leq 9t+2m_0 Dt$ \ (for $m_0$ depending on $\alpha$ --- see \eqref{eq:EMM-use'}), and a subset $F\subset B_\rfrak(0,1)$ containing $0$ with 
\[
\nuni^{t/2}\leq \#F\leq \nuni^{5t},
\] 
so that both of the following properties are satisfied:
\begin{itemize}
    \item $\Bigl \{\exp(w)x_1: w\in F\Bigr\}\subset \Bigl(\boxHs_{\nuni^{-t/R}} \cdot a_{\tau}\cdot \{u_rx_0: |r|\leq 4\}\Bigr)\cap X_\injr$, where $R>0$ depends on $D$, $\vare$, and $\alpha$,
    \item $\sum_{w'\neq w}\|w-w'\|^{-\alpha}\leq C\cdot (\#F)^{1+\vare}$ for all $w\in F$ (where the summation is over $w' \in F$ and $C$ is an absolute constant).
\end{itemize}

\item There is $x'\in X$ such that $Hx'$ is periodic with
\[
\vol(Hx')\leq \nuni^{D_0\rws}\quad\text{and}\quad\dist_X(x',x_0)\leq \nuni^{(-D+D_0)\rws}.
\] 
\end{enumerate}
\end{propos}

\medskip
\noindent 
Explicitly, $m_0$ is equal to $m_\alpha$ of~\eqref{eq:EMM-use}, chosen so that for all $w\in\gfrak$, we have 
\begin{equation}\label{eq:EMM-use'}
\ave{\|a_{{m_0}}\uvk w\|^{-\alpha}}\uvkd\leq \nuni^{-1}\|w\|^{-\alpha}.
\end{equation}

\subsection{The definition of a Margulis function}
Throughout this section, $\cone\subset X$ denotes a Borel set which is a disjoint finite union of local 
$H$ orbits. More precisely, there is a finite set $F$ and for every $w\in F$, there exist 
$x_w\in X$ and a bounded Borel set $\coneH_w\subset H$ satisfying the following 
\begin{itemize}
\item the map $\sfh\mapsto \sfh.x_w$ is injective over $\coneH_w$ for all $w\in F$, and 
\item $\coneH_w.x_w\cap \coneH_{w'}.x_{w'}=\emptyset$ for all $w\neq w'$,
\end{itemize}
so that $\cone=\textstyle\bigcup_{w\in F} \coneH_w.x_w$.

For every $w\in F$, let $\mu_{\coneH_w}$ denote the pushforward of the Haar measure $m_H|_{\coneH_w}$ under the map $h \mapsto h.x_w$.
Put
\be\label{eq:def-mu-cone}
\mu_\cone=\frac{1}{\sum_w m_H(\coneH_w)}\sum_w\mu_{\coneH_w}.
\ee

For every $(h,z)\in H\times \cone$, define 
\be\label{eq:def-I-h-y-C}
\margI_{\cone}(h,z):=\Bigl\{w\in \rfrak: 0<\|w\|<\inj(hz),\, \exp(w) h z\in h\cone\Bigr\}.
\ee
Since $\coneH_w$ is bounded for every $w$ and $F$ is finite, $\margI_{\cone}(h,z)$ is a finite set for all $(h,z)\in H\times\cone$. 

Fix some $0<\alpha<1$. Define the Margulis function $\mfht_{\cone}=\mfht_{\cone,\alpha}: H\times \cone\to [1,\infty)$ as follows:
\be\label{eq:def-mfht-C}
\mfht_{\cone}(h,z)=\begin{cases} \sum_{w\in I_{\cone}(h,z)}\|w\|^{-\alpha} & \text{if $\margI_{\cone}(h,z)\neq\emptyset $}\\
\inj(hz)^{-\alpha}&\text{otherwise}
\end{cases}.
\ee

Let $\rwm=\rwm(\alpha)$ be the probability measure on $H$ defined by 
\be\label{eq:def-rwm-C}
\rwm(\varphi)=\ave\varphi(a_{m_0}\uvk)\uvkd\qquad\text{for all $\varphi\in C_c(H)$,}
\ee
where $m_0$ is as in~\eqref{eq:EMM-use'}.

Define $\noI_\cone$ on $H\times \cone$ by  
\be\label{eq:def-noI-alpha-C}
\noI_{\cone}(h,z):=\Bigl(\max\bigl\{\#\margI_{\cone}(h,z),1\bigr\}\Bigr)\cdot\inj(hz)^{-\alpha}.
\ee 
We will use the following lemma to increase the {\em transversal} dimension inductively. 

\begin{lemma}\label{lem:Margulis-inequality} 
There exists some $\constE\label{E:margb-lemma-C}=\ref{E:margb-lemma-C}(\rwm)$ 
so that for all $\ell\in\bbn$ and all $z\in \cone$, we have 
\[
\int\mfht_\cone(h,z)\diff\!\convL(h)\leq \nuni^{-\ell} \mfht_\cone(e, z)+\ref{E:margb-lemma-C}\sum_{j=1}^\ell \nuni^{j-\ell}\int\noI_{\cone}(h,z)\diff\!\rwm^{(j)}(h),
\]
where $\rwm^{(j)}$ denotes the $j$-fold convolution of $\rwm$ for every $j\in\bbn$.
\end{lemma}

\begin{proof}
Throughout the argument, the set $\cone$ is fixed; 
thus, we drop it from the indices in the notation. 
Note that $\supp (\rwm)\subset\{h\in H: \|h\|\leq \nuni^{2m_0+1}\}$. Let $C\geq 1$ be so that 
\[
\|\Ad(h)w\|\leq C\|w\|
\] 
for all $h$ with $\|h\|\leq \nuni^{2m_0+1}$ and all $w\in\mathfrak g$. Increasing $C$ if necessary, we also assume that
$\inj(z)/C\leq \inj(hz)\leq C\inj(z)$ for all such $h$ and all $z\in X$.

Let $h=a_{{\onst}}\uvk$ for some $r\in[0,1]$. 
Let $z\in \cone$, and let $h'\in H$.
First, let us assume that there exists some $w\in \margI(hh',z)$ with $\|w\|<\inj(hh'z)/C^2$. 
In view of the choice of $C$, this in particular implies that both $\margI(hh',z)$ and $\margI(h',z)$ are non-empty. 
Hence, we have  
\begin{align}
\notag \mfht(hh',z)&=\sum_{w\in \margI(hh',z)}\|w\|^{-\alpha}\\
\notag&=\sum_{\|w\|< \inj(hh'z)/C^2}\|w\|^{-\alpha}+\sum_{\|w\|\geq \inj(hh'z)/C^2}\|w\|^{-\alpha}\\
\notag&\leq \sum_{w\in I(h',z)}\|\Ad(h)w\|^{-\alpha}+C^{2\alpha}\cdot\Bigl(\# \margI(hh',z)\Bigr)\cdot\inj(hh'z)^{-\alpha}\\
\label{eq:MF-1}&=\sum_{w\in I(h',z)}\|\Ad(h)w\|^{-\alpha}+C^{2\alpha}\noI(hh',z).
\end{align}

Note also that if $\|w\|\geq \inj(hh'z)/C^2$ for all $w\in \margI(hh',z)$ 
(which in view of the choice of $C$ includes the case $\margI(h',z)=\emptyset$) 
or if $\margI(hh',z)=\emptyset$, then
\begin{align}
 \label{eq:MF-2}
\mfht(hh',z)&\leq C^{2\alpha}\cdot\Bigl(\max\{\# \margI(hh',z),1\}\Bigr)\cdot\inj(hh'z)^{-\alpha}\\
\notag&=C^{2\alpha}\noI(hh',z).   
\end{align}

We now average~\eqref{eq:MF-1} and~\eqref{eq:MF-2} over $[0,1]$ and conclude hat 
\begin{multline*}
\ave \mfht(a_{{\onst}}\uvk h',z)\uvkd\leq 
\sum_{w\in \margI(h',z)}\ave \|a_{{\onst}}\uvk w\|^{-\alpha}\uvkd\quad+\\ C^{2\alpha}\ave\noI(a_{m_0}u_rh',z)\uvkd,
\end{multline*}
where we replace the summation on the right by $0$ if $\margI(h',z)=\emptyset$. 
Thus by~\eqref{eq:EMM-use'} we may conclude that  
\[
\int\mfht(hh',z)\diff\!\rwm(h)\leq \nuni^{-1}\cdot\mfht(h',z)+C^{2\alpha}\int\noI(hh',z)\diff\!\rwm(h)
\]
for all $h'\in H$. Iterating this estimate, we have 
\[
\int\mfht(h,z)\diff\!\convL(h)\leq \nuni^{-1}\int\mfht(h',z)\diff\!\rwm^{(\ell-1)}(h')+C^{2\alpha}\int\noI(h,z)\diff\!\convL(h).
\]  
The claim in the lemma thus follows from the above by induction if we let~$\ref{E:margb-lemma-C}=C^{2}$ and sum the geometric series.
\end{proof}

\subsection{Incremental dimension increase}\label{sec:def-E-Ecal}
Let $0<\injr\leq0.01\eta_X$ and $0<\beta\leq \eta^2$. Define
\[
\coneH=\boxHs_\beta\cdot\Big\{u_r: |r|\leq 0.1\eta\Big\}.
\]
Let $F\subset B_\rfrak(0,\beta)$ be a finite set, and let $y_0\in X_{2\eta}$. Then for all $w\in F$ \ $\exp(w)y_0\in X_\eta$, and 
$h\mapsto h\exp(w)y_0$ is injective on $\coneH$. Put  
\be\label{eq:def-cone}
\cone=\coneH.\{\exp(w)y_0: w\in F\}.
\ee

Let us begin with the following two elementary lemmas.

\begin{lemma}\label{lem:noI-tri-bd}
There exists $\constE\label{E:noI}>0$ so that the following holds. For every $m\in\bbn$, every $|\rel|\leq 2$, 
and every $z\in\cone$, we have 
\[
\#\margI_{\cone}(a_mu_\rel,z)\leq \ref{E:noI}\beta^{-6}\nuni^{4m}\cdot(\#F)
\]
Moreover, we have 
\[
\noI_{\cone}(a_mu_\rel,z)\leq \ref{E:noI}\beta^{-7}\nuni^{5m}\cdot(\#F).
\]
\end{lemma}

\begin{proof}
Let $z\in\cone$, and let $w\in\margI_{\cone}(a_mu_r,z)$. Then $\exp(w)a_mu_rz\in a_mu_r\cone$. 
Therefore, using Lemma~\ref{lem:commutation-rel}(2), we have  
\[
\umt_{\beta^2, m}^H.\exp(w)a_mu_rz\subset a_m u_r\cone_+
\]
where $\cone_+=\boxHs_{\beta+100\beta^2}\Big\{u_r\exp(w)y_0: |r|\leq 0.1\eta, w\in F\Big\}$ and
\[
\umt^H_{\beta^2,m}=\Bigl\{u^-_s: |s|\leq \beta^2 \nuni^{- m}\Bigr\}\cdot\{a_t: |t|\leq \beta^2\}\cdot\Bigl\{u_r: |r|\leq \beta^2\Bigr\}.
\]

Note that the map $(\sfh,w')\mapsto \sfh\exp(w')a_mu_rz$ is injective over 
\[
\umt^H_{\inj(a_mu_rz)}\times \exp(B_\rfrak(0,\inj(a_mu_rz))),
\]
and let $\mu_{\cone_+}$ is the probability measure on $\cone_+$ defined as in~\eqref{eq:def-mu-cone}. Then 
\[
a_{m}u_r.\mu_{\cone_+}\Bigl(\umt_{\beta^2, m}^H\exp(w).a_mu_rz\Bigr)\gg(\min\{\beta^2,\inj(a_mu_rz)\})^{3}\nuni^{-m}(\#F)^{-1}
\]
where the implied constant is absolute.   

Recall now that $\cone\subset X_\eta$. 
Thus, $\inj(a_mu_rz)\gg \nuni^{-m}\eta$. Recall also that $\beta\leq \eta^2$, this implies the first claim.

We now show the second claim. The above estimate and the definition of $\noI_{\cone}(h,z)$ thus imply that 
\[
\noI_{\cone}(a_mu_\rel,z)\ll \Bigl(\beta^{-6}\nuni^{4m}\cdot (\#F)\Bigr) \cdot \inj(a_mu_\rel z)^{-1};
\]
we also used $0<\alpha<1$ in the above upper bound. The second claim in the lemma follows.
\end{proof}

\begin{lemma}\label{lem:MargFun-enetrgy-rfrak}
Let the notation be as above. In particular, $y_0\in X_{2\eta}$ and 
\[
\cone=\coneH.\{\exp(w)y_0: w\in F\}
\]
where $F\subset B_\rfrak(0,\beta)$. Let $w_0\in F$, then  
\[
\textstyle\sum_{w\neq w_0}\|w-w_0\|^{-\alpha}\leq 2\mfht_{\cone}(e,z)
\]
where $z=\exp(w_0)y_0$ and the summation is over $w\in F$.
\end{lemma}

\begin{proof}
By the definition of $\mfht_{\cone}$, we have 
\[
\mfht_{\cone}(e,z)=\textstyle\sum_{v\in \margI_{\cone}(e,z)}\|v\|^{-\alpha}.
\]

Let $w_0\neq w\in F$. We will find a unique vector $v_w\in \margI_{\cone}(e,z)$ whose length is comparable to $\|w-w_0\|$.
Let us begin with the following computation. 
\begin{align}
\notag\exp(w)y&=\exp(w)\exp(-w_0)\exp(w_0)y_0\\
\notag&=h_{w}\exp(v_{w})\exp(w_0)y_0\\
\notag&=h_{w}\exp(v_w)z,
\end{align}
where $h_w\in H$, $v_w\in\rfrak$, $\|h_{w}-I\|\leq \ref{E:BCH}\beta\|v_w\|$, and 
\be\label{eq:vw-bound-energy}
0.5\|w-w_0\|\leq \|v_w\|\leq2\|w-w_0\|,
\ee
see Lemma~\ref{lem:BCH}. 

In particular, we have $\|h_{w}-I\|\ll\beta^2$; assuming $\beta\leq \injr^2$ is small enough, we conclude that 
$h_w^{\pm1}\in \boxH_{\beta}$. Hence, 
\[
\exp(v_w)z=h_w^{-1}\exp(w)y_0\in\cone.
\]
Moreover, using~\eqref{eq:vw-bound-energy}, we have $\|v_w\|\leq 2\beta\leq \inj(z)$. 
We thus conclude that $v_w\in\margI_{\cone}(e,z)$.  

Since $\exp(w) y_0\neq \exp(w')y_0$ for $w\neq w'\in F\subset B_\rfrak(0,\beta)$, the map $w\mapsto v_w$ is well-defined and one-to-one.
Altogether, we deduce that
\[
\textstyle\sum_{w\neq w_0}\|w-w_0\|^{-\alpha}\leq 2\textstyle\sum_{v\in \margI_{\cone}(e,z)}\|v\|^{-\alpha}=2\mfht_{\cone}(e,z),
\]
as was claimed.
\end{proof}

\begin{lemma}\label{lem:noI-inh-C}
There exist $0<\constk\label{k:epsilon-t}=\ref{k:epsilon-t}(\rwm)\leq \tfrac{1}{4m_0}$ and $n_0$ depending on $X$ so that the following holds.  
Let $\cone$ be defined as in \eqref{eq:def-cone}. 
Assume further that 
\be\label{eq:initial-bd}
f_\cone(e,z)\leq \nuni^{Mn}\qquad \text{for all $z\in \cone$}
\ee
for some $M>0$ and an integer $n\geq n_0$.

Then for all $0<\vare<0.1$ and all $\beta\geq \nuni^{-0.01\vare n}$ at least one of the following holds.

\begin{enumerate}
\item $\nuni^{Mn}< \nuni^{\vare n/2}\cdot(\#F)$, or
\item For all integers $0<\ell\leq \ref{k:epsilon-t}\vare n$ and all $z\in \cone$, we have 
\[
\int \mfht_\cone(h,z)\diff\!\convL(h)\leq 2\nuni^{Mn-\ell}.
\]
\end{enumerate}
\end{lemma}

\begin{proof}
By Lemma~\ref{lem:Margulis-inequality}, applied with $\mfht_\cone$, we have 
\[
\int \mfht_\cone(h,z)\diff\!\convL(h)\leq \nuni^{-\ell}\mfht_\cone(e,z)+\ref{E:margb-lemma-C}\sum_{j=1}^\ell \nuni^{j-\ell}\int\noI_\cone(h,z)\diff\!\rwm^{(j)}(h).
\]

Assuming $n$ is large enough, Lemma~\ref{lem:noI-tri-bd} implies that there exists a constant $C$ depending only on $\nu$ so that if
$j\leq \vare n/C$, then 
\[
\noI_\cone(h,z)\leq (2\ref{E:margb-lemma-C})^{-1}\nuni^{\vare n/4}\cdot (\#F),
\]
for all $h\in\supp(\rwm^{(j)})$ --- we used $\beta\geq \nuni^{-0.01 \vare n}$ and assumed $n$ is large enough to account for the factor $\ref{E:noI}\beta^{-7}$ in Lemma~\ref{lem:noI-tri-bd}.

Let $\ref{k:epsilon-t}=(2C)^{-1}$, and let $\ell\leq \ref{k:epsilon-t}\vare n$. Then 
\[
\int \mfht_\cone(h,z)\diff\!\convL(h)\leq \nuni^{-\ell}\mfht_\cone(e,z)+\nuni^{\vare n/4}\cdot(\#F)\leq \nuni^{Mn-\ell}+\nuni^{\vare n/4}\cdot(\#F).
\]
Therefore, either part~(1) holds or $\nuni^{Mn-\ell}\geq \nuni^{(0.5-\ref{k:epsilon-t})\vare n}\cdot(\#F)\geq \nuni^{\vare n/4}\cdot(\#F)$. In the latter case, the above implies that
\[
\int \mfht_\cone(h,z)\diff\!\convL(h)\leq 2\nuni^{Mn-\ell}
\] 
as we claimed in part~(2).
\end{proof}

From this point until the Lemma~\ref{lem:mfht-base-case}, we fix some $0<\vare<0.1$, and let $\beta=\nuni^{-\kappa n/2}$ where 
$0<\kappa\leq 0.02\ref{k:epsilon-t}\vare$ will be explicated later.

The following lemma will convert the estimate we obtained on average in Lemma~\ref{lem:noI-inh-C} into pointwise information at most points. This is done in a fairly straightforward way essentially by using the Chebyshev inequality. Recall from Proposition~\ref{lem:one-return} that for any interval $I\subset \bbr$ of length at least $\eta$ and $t\geq |\log(\eta^2\inj(x))|+\ref{E:non-div-main}$
\[
\Bigl|\Bigl\{r\in I:\inj(a_t\uvk x)< \vare^2\Bigr\}\Bigr|<\ref{E:non-div-main}\vare |I|.
\]

\begin{lemma}\label{lem:EG-Cheby-C}
Let the notation be as in Lemma~\ref{lem:noI-inh-C}. 
Let $0<\vare< 0.1$, and assume that 
\[
\ell=\lfloor\ref{k:epsilon-t}\vare n\rfloor\geq 3|\log\injr|+\ref{E:non-div-main}+6.
\]
Further assume that Lemma~\ref{lem:noI-inh-C}(2) holds for these choices.  

There exists a subset $L_\cone\subset\supp(\convL)$ with 
$\convL(L_\cone)\geq 1-2\nuni^{-\ell/8}$ so that both of the following hold. 
\begin{enumerate}
\item For all $h_0\in L_\cone$ we have 
\[
\int \mfht_\cone(h_0,z)\diff\!\mu_\cone(z)\leq \nuni^{Mn-\frac{7\ell}{8}}.
\]
\item For all $h_0\in L_\cone$, there exists $\cone(h_0)\subset \cone$ 
with $\mu_\cone(\cone(h_0))\geq 1-O(\eta^{1/2})$, so that for all $z\in\cone(h_0)$ we have 
\begin{subequations}
\begin{align}
\label{eq:thickening-cone-h0} \boxH_{100\beta^2}.z &\subset\cone\\
\label{eq:I-cone-isolated-1}h_0z&\in X_{2\injr}\\
\label{eq:I-cone-isolated-2} f(h_0,z)&\leq\nuni^{Mn-\frac{3\ell}{4}}.
\end{align}
\end{subequations}
\end{enumerate}
\end{lemma}

\begin{proof}
Let us begin by finding $L_\cone$ which satisfies part~(1). 
Apply Lemma~\ref{lem:noI-inh-C} with $\ell=\lfloor\ref{k:epsilon-t}\vare n\rfloor$. Since Lemma~\ref{lem:noI-inh-C}(2) holds, we have 
\[
\iint \mfht_\cone(h,z)\diff\!\mu_\cone(z)\diff\!\rwm^{(\ell)}(h)\leq 2\nuni^{Mn-\ell}.
\]
Using this estimate and Chebyshev's inequality, we have  
\be\label{eq:EG-est-use-C}
\convL\Bigl\{h\in\supp (\convL): \textstyle\int \mfht(h,z)\diff\!\mu_\cone(z)> \nuni^{Mn-\frac{7\ell}{8}}\Bigr\}<2\nuni^{-\ell/8}.
\ee

Let $L_\cone$ be the complement in $\supp (\convL)$ 
of the set on the left side of~\eqref{eq:EG-est-use-C}, and let $h_0\in L_\cone$. Then 
\be\label{eq:10L-Cheby}
\int \mfht(h_0,z)\diff\!\mu_\cone(z)\leq \nuni^{Mn-\frac{7\ell}{8}}.
\ee
The claim in part~(1) thus holds with $L_\cone$.

Let us now turn to the proof of~(2). Let $h\in\supp(\convL)$. Then $h=a_{\ell m_0}u_{\hat r}$ where 
$\hat r=\sum_{j=0}^{\ell-1} \nuni^{-j\onst}r_{j+1}$ for some $r_1,\ldots, r_\ell\in[0,1]$.

For every $z=u^-_sau_{r'}u_r\exp(w).y_0\in\cone$, we have  
\[
hz=(a_{\ell m_0}u_{\hat r})u^-_sau_{r'}u_r\exp(w).y_0=h'a_{\ell m_0}u_{r'_s+\hat r+r}\exp(w).y_0
\]
where $h'\in \boxH_{\beta}$ and $|r'_s|\ll\beta$ for an absolute implied constant. Therefore,
if $a_{\ell m_0}u_{r'_s+\hat r+r}\exp(w)y_0\in X_{4\eta}$, then
$hz\in X_{2\eta}$.

Apply Proposition~\ref{prop:Non-div-main} 
with $\exp(w)y_0\in\cone\subset X_{\injr}$ and the interval $I=[r'_s+\hat r-0.1\eta,r'_s+\hat r+0.1\injr]$. 
Since $\ell\geq 3|\log\injr|+\ref{E:non-div-main}+6$, we conclude 
\[
|\{r\in[-0.1\eta, 0.1\injr]: a_{\ell m_0}u_{r'_s+\hat r+r}\exp(w)y_0\not\in X_{4\injr}\}|\leq 0.4\ref{E:non-div-main}\injr\sqrt{\eta}.
\]
This estimate, the above observation, and the definition of $\mu_\cone$ imply that  
\be\label{eq:X-injr-C}
\mu_\cone\{z\in\cone: hz\not\in X_{2\injr}\}\leq 2\ref{E:non-div-main}\sqrt{\injr},
\ee
for every $h\in\supp(\convL)$.

Put 
\[
\cone_-=\boxHs_{\beta-200\beta^2}\{u_r\exp(w)y_0: |r|\leq0.1\eta, w\in F\};
\]
then $\mu_\cone(\cone_-)\geq 1-O(\beta)$.

Let now $h_0\in L_\cone$. 
Recall also that $0<\beta<\eta^2$. Then~\eqref{eq:X-injr-C}, implies that there is a subset $\cone'(h_0)\subset \cone_-$ with 
\[
\mu_\cone(\cone'(h_0))\geq 1-O(\eta^{1/2}),
\]
so that for all $z\in \cone'(h_0)$ we have $h_0z\in X_{2\injr}$. Hence all points in $\cone'(h_0)$ satisfy~\eqref{eq:thickening-cone-h0} and~\eqref{eq:I-cone-isolated-1}.

We will find a subset $\cone(h_0)\subset \cone'(h_0)$ which satisfies~\eqref{eq:I-cone-isolated-2}. Let
\[
\cone''=\Big\{z\in\cone'(h_0): f(h_0,z)>\nuni^{Mn-\frac{3\ell}{4}}\Big\}.
\]
Then
\begin{align*}
\mu_\cone(\cone'')\nuni^{Mn-\frac{3\ell}{4}}&\leq\int_{\cone''}\mfht(h_0,z)\diff\!\mu_\cone(z)\\
&\leq\int_\cone\mfht(h_0,z)\diff\!\mu_\cone(z)\leq \nuni^{Mn-\frac{7\ell}{8}}&&\text{by~\eqref{eq:10L-Cheby}}.
\end{align*}
We conclude from the above that 
$\mu_\cone(\cone'')\ll \nuni^{-\ell/8}$. Recall that 
$\beta=\nuni^{-\kappa n/2}$ where $0<\kappa\leq 0.02\ref{k:epsilon-t}\vare$, thus we conclude that $\mu_\cone(\cone'')\ll\eta$.

Put $\cone(h_0):=\cone'(h_0)\setminus\cone''$. Then $\mu_\cone(\cone(h_0))\geq 1-O(\eta^{1/2})$ 
and~\eqref{eq:I-cone-isolated-2} holds for every $z\in\cone(h_0)$. The proof is complete. 
\end{proof}

In the remaining parts of this section,
we will write $\umt^H$ for  
\be\label{eq:def-BH}
\umt_{\beta^2,\ell\onst}^H=\Bigl\{u^-_s: |s|\leq \beta^2 \nuni^{- \ell m_0}\Bigr\}\cdot\{a_t: |t|\leq \beta^2\}\cdot\Bigl\{u_r: |r|\leq \beta^2\Bigr\}
\ee
where $\ell=\lfloor\ref{k:epsilon-t}\vare n\rfloor$, see~\eqref{eq:def-B-ell-beta}.

Let us also define a subset in $G$ by thickening $\umt^H$ in the transversal direction as follows. 
Put
\be\label{eq:def-O-ell-C}
\umt^G:=\umt^H\cdot \exp(B_\rfrak(0,2\beta^2)).
\ee

\begin{lemma}\label{lem:E-good-h0}
There exists a covering $\Big\{\umt^G.y_{j}: j\in \mathcal J, y_j\in X_\eta\Big\}$ of $X_{2\injr}$
where $\#\mathcal J\ll\beta^{-12}\nuni^{\ell m_0}$ and the implied constant depends on $X$. 

Moreover, if for every $h_0\in L_\cone$ we let 
\be\label{eq:Jh-0}
\mathcal J(h_0)=\Big\{j\in\mathcal J: h_0.\mu_\cone\Big(h_0\cone(h_0)\cap\umt^G.y_{j}\Big)\geq \beta^{13}\nuni^{-\ell m_0}\Big\}
\ee
and define $\hat\cone(h_0)\subset\cone(h_0)$ by 
\[
h_0\hat\cone(h_0)=h_0\cone(h_0)\bigcap\Big(\textstyle\bigcup_{j\in\mathcal J(h_0)}\umt^G.y_{j}\Big),
\]
then $\mu_\cone(\hat\cone(h_0))\geq 1-O(\sqrt\eta)$ where the implied constant depends on $X$.
In particular, $\mathcal J(h_0)\neq\emptyset$. 
\end{lemma}

\begin{proof}
For simplicity in the notation, let us write $\boxG$ for 
\[
\boxG_{\beta^2}=\boxH_{\beta^2}\cdot\exp(B_\rfrak(0,\beta^2)).
\] 

We begin by constructing a covering of $\boxG$. 
First recall that 
\be\label{eq:Vitali Covering 1}
m_H(\umt^H_{0.01\beta^2, \ell m_0})\asymp \nuni^{-\ell m_0}m_H(\exp(B_\hfrak(0,\beta^2))),
\ee 
where the implied constant is absolute, see~\eqref{eq:def-B-ell-beta}.
Moreover, by Lemma~\ref{lem:commutation-rel} we have 
\be\label{eq:Vitali Covering 2}
\umt^H_{0.01\beta^2, \ell m_0}\cdot (\umt^H_{0.01\beta^2, \ell m_0})^{\pm1}\subset \umt_{\beta^2, \ell m_0}^H.
\ee
Fix a maximal subset 
$\mathcal H\subset \boxH_{\beta^2}$
so that  
\[
\umt^H_{0.01\beta^2, \ell m_0}h\cap \umt^H_{0.01\beta^2, \ell m_0}h'=\emptyset,
\]
for all $h\neq h'\in \mathcal H$.
In view of~\eqref{eq:Vitali Covering 1}, we have $\#\mathcal H\ll\nuni^{\ell m_0}$ where the implied constant is absolute. Then using~\eqref{eq:Vitali Covering 2}, we conclude that 
$\{\umt^Hh_j: h_j\in\mathcal H\}$ covers $\boxH_{\beta^2}$
and $\#\mathcal H\asymp \nuni^{\ell m_0}$.

Taking the product with $\exp(B_\rfrak(0,\beta^2))$, we thus obtain a covering 
\[
\{\umt^Hh_j\exp(B_\rfrak(0,\beta^2)):h_j\in\mathcal H\}
\]
of the set $\boxG$.  

Recall that $\beta\leq \injr^2$, and that by Lemma~\ref{lem:BCH}, we have 
$(\boxG_{\delta})^{-1}\cdot\boxG_{\delta}\subset \boxG_{c\delta}$ for all $\delta>0$, where $c$ is an absolute constant. Hence, arguing as above, there exists a covering  
\[
\{\boxG.\hat y_k:k\in\mathcal K, \hat y_k\in X_{2\injr}\},
\]
of $X_{2\injr}$ which satisfies $\#\mathcal K\asymp \beta^{-12}$ for an implied constant depending on $X$.

Combining these two coverings, we obtain a covering 
\[
\{\umt^Hh_j\exp(B_\rfrak(0,\beta^2)).\hat y_k: h_j\in \mathcal H, k\in\mathcal K\}.
\]
of $X_{2\eta}$. Note further that 
\[
\umt^Hh_j\exp(B_\rfrak(0,\beta^2))=\umt^H\exp\Big(\Ad(h_j)B_\rfrak(0,\beta^2)\Big)h_j\subset \umt^G h_j;
\]
where we used the fact that $\Ad(h_j)B_\rfrak(0,\beta^2)\subset B_\rfrak(0,2\beta^2)$ in the final inclusion above --- this holds since $\|h_j-I\|\leq2\beta^2$ and $\beta$ is small. 

Finally note that since $\hat y_k\in X_{2\eta}$ and $\|h_j-I\|\leq2\beta^2$, we have $h_j\hat y_k\in X_\eta$, for every $j,k$. Altogether, we obtain a covering 
\[
\{\umt^G.y_j:j\in \mathcal J, y_j\in X_\eta\}=\{\umt^G.h_j\hat y_k:h_j\in\mathcal H, k\in\mathcal K\}
\]
of $X_{2\eta}$ where $\#\mathcal J\ll \beta^{-12}\nuni^{\ell m_0}$. This finishes the proof of the first claim.

To see the other claims, 
let $h_0\in L_\cone$, and define $\mathcal J(h_0)$ as in the statement. 
Then for every $j\not\in\mathcal J(h_0)$, we have 
\[
h_0.\mu_\cone\Big(h_0\cone(h_0)\cap\umt^G.y_{j}\Big)< \beta^{13}\nuni^{-\ell m_0}.
\]
This estimate and the bound on $\#\mathcal J$ yield
\[
h_0.\mu_\cone\Big(h_0\cone(h_0)\cap (\cup_{j\not\in\mathcal J(h_0)}\umt^G.y_{j})\Big)\ll \beta 
\]
where the implied constant depends on $X$. 
The desired bound on the measure of $h_0\hat\cone(h_0)$ thus follows since
$h_0.\mu_\cone\Big(h_0\hat\cone(h_0))\geq 1-O(\sqrt\eta)$.

The fact that $\mathcal J(h_0)\neq\emptyset$ is a consequence of the fact that $\hat\cone(h_0)\neq\emptyset$, 
which is immediate from the above bound.  
\end{proof}

The following lemma yields a set $\cone_1$ defined as in~\eqref{eq:def-cone}, for some $y_1$ and $F_1$, but with an improved bound for $\mfht_{\cone_1}(e,z)$. This lemma will serve as our main tool for incremental dimension increase in the proof of Proposition~\ref{prop:dim-1-C-rfrak}.

\begin{lemma}\label{lem:inductive-C}
There exists $n_0$ so that the following holds for all $n\geq n_0$.
Let the notation be as in Lemmas~\ref{lem:EG-Cheby-C} and~\ref{lem:E-good-h0}. 
In particular, $0<\vare\leq 0.1$ and  
\[
\ell=\lfloor\ref{k:epsilon-t}\vare n\rfloor\geq 3|\log\injr|+\ref{E:non-div-main}+6;
\] 
assume further that $\#F\geq \nuni^{n/2}$ and that Lemma~\ref{lem:noI-inh-C}(2) holds.

Let $h_0\in L_\cone$, and let $y=y_j$ for some $j\in\mathcal J(h_0)$. There exists some 
\[
h_0z_1\in h_0\cone(h_0)\cap\umt^G.y
\]
and a subset 
\[
\mbox{$F_1\subset B_\rfrak(0,\beta)\;\;$ with $\;\; \#F_1=\lceil\beta^{10}\cdot(\#F)\rceil$}
\]
containing $0$, so that both of the following are satisfied. 
\begin{enumerate}
\item For all $w\in F_1$, we have 
\[
\exp(w)h_0z_1\in \boxH_{100\beta^2}.h_0\cone(h_0).
\] 
\item If we define $\cone_1=\coneH.\{\exp(w)h_0z_1: w\in F_1\}$, then at least one of the following
two possibilities hold
\begin{subequations}
\begin{align}
\label{eq:mfht-ind-stop}\mfht_{\cone_1}(e,z)&\leq 2\cdot (\#F_1)^{1+\vare}&&\text{for all $z\in\cone_1$, or}\\
\label{eq:mfht-ind-improve}\mfht_{\cone_1}(e,z)&\leq \nuni^{(M-\frac{2\ref{k:epsilon-t}\vare}{3}) n}&&\text{for all $z\in\cone_1$}.
\end{align}
\end{subequations}
\end{enumerate}
\end{lemma}

\begin{proof}
Let $h_0\in L_\cone$ and $y=y_j$ be as in the statement of the lemma. 

The set $h_0\cone(h_0)\cap\umt^G.y$ is contained in a finite union of local $H$-orbits. Let $\mathsf M\in\bbn$ 
be minimal so that 
\be \label{eq:Dy-0}
h_0\cone(h_0)\cap\umt^G.y \subset \bigcup_{i=1}^{\mathsf M}  \umt^H.\exp(w_i)y
\ee
where $w_i\in B_\rfrak(0,2\beta^2)$.

For each $1\leq i\leq \mathsf M$, fix some 
$z_i\in\cone(h_0)$ so that $h_0z_i\in \umt^G.y$ and write
\be\label{eq: def z-i}
h_0z_i=\sfh_i\exp(w_i)y\qquad\text{for some $\sfh_i\in \umt^H$.}
\ee 
We claim that both of the following properties are satisfied
\begin{subequations}
\begin{align}
    \label{eq:Dy-1}&\umt^H.h_0z_i\cap\umt^H.h_0z_j=\emptyset\qquad\text{$1\leq i\neq j\leq M$.}\\
    \label{eq:Dy-2}&h_0\cone(h_0)\cap\umt^G.y\subset \bigcup_{i=1}^{\mathsf M}\,\umt^H\cdot(\umt^H)^{-1}.h_0z_i.
\end{align}
\end{subequations}

Assume contrary to~\eqref{eq:Dy-1} that $\sfh h_0z_i=\sfh' h_0z_j$ for $i\neq j$. Then 
\begin{align*}
\sfh^{-1}\sfh'\sfh_j\exp(w_j)y&=\sfh^{-1}\sfh' h_0z_j\\
&=h_0z_i=\sfh_i\exp(w_i)y.
\end{align*}
That is $\exp(-w_i)\hat\sfh\exp(w_j)y=y$
where $\hat\sfh=\sfh_i^{-1}\sfh^{-1}\sfh'\sfh_j$. Note moreover that $\hat\sfh\in \boxH_{100\beta^2}$, see~\eqref{eq:B-beta-almost-group}, and $w_i\neq w_j\in B_\rfrak(0,2\beta^2)$. Therefore $I\neq \exp(-w_i)\hat\sfh\exp(w_j)\in \boxG_{200\beta^2}$.  
Recall however that $\beta\leq \eta^2$ and $y\in X_{2\eta}$, thus, $g\mapsto g.h_0 z_i$ is injective on $\boxG_{1000\beta^2}$ for all small enough $\beta$. This contradiction implies that~\eqref{eq:Dy-1} holds.

We now show~\eqref{eq:Dy-2}. Let $h_0z\in h_0\cone(h_0)\cap\umt^G.y$, then $h_0z=\sfh\exp(w_i)y$ for $1\leq i \leq \mathsf M$ and $\sfh\in\umt^H$. Moreover, we have  $h_0z_i=\sfh_i\exp(w_i)y$, thus 
$h_0z=\sfh\sfh_i^{-1}h_0z_i$ as claimed in~\eqref{eq:Dy-2}.

\medskip

Recall now that $\cone=\coneH.\{\exp(w)x:w\in F\}$ where $\coneH\subset H$ with $m_H(\coneH)\asymp \beta^2\eta$. 
In view of the definition of $\mu_\cone$, see~\eqref{eq:def-mu-cone}, we conclude that 
\[
h_0\mu_\cone(\umt^H.h_0z_i)\ll \beta^6\nuni^{-\ell m_0}\beta^{-2}\eta^{-1}(\#F)^{-1}\ll \beta^{3.5}\nuni^{-\ell m_0}(\#F)^{-1};
\]
recall that $\beta\leq \eta^2$. 

Using~\eqref{eq:Dy-1} and the definition of $\mathcal J(h_0)$ in \eqref{eq:Jh-0}, 
we deduce from the above that $\mathsf M\gg \beta^{9.5}\cdot(\#F)$. 
Assuming $\beta$ is small so to account for the implied multiplicative constant 
(which depends only on $G$ and $\Gamma$), we get 
\be\label{eq:num-D-y'}
\mathsf M\geq \beta^{10}\cdot(\#F).
\ee

Let $1\leq i,j\leq \mathsf M$, then using~\eqref{eq: def z-i} we have 
\begin{align}
\label{eq:zi-hi-wi} h_0z_i&=\sfh_i\exp(w_i)y=\sfh_i\exp(w_i)\exp(-w_j)\sfh_j^{-1}h_0z_j\\
\notag &=\sfh_i\sfh_j^{-1}\exp(\Ad(\sfh_j)w_i)\exp(-\Ad(\sfh_j)w_j)h_0z_j\\
\notag&=\sfh_i\sfh_j^{-1}\sfh_{ij}\exp(w_{ij})h_0z_j
\end{align}
where $\sfh_{ij}\in H$ and $w_{ij}\in\rfrak$, $\sfh_{ii}=I$, $w_{ii}=0$ for all $i,j$; moreover, we have
\begin{subequations}
\begin{align}
\label{eq:hi-hat-wi-1}&\|\sfh_{ij}-I\|\leq \ref{E:BCH}\beta^2\|w_{ij}\|\qquad\text{and}\\
\label{eq:hi-hat-wi-2}&0.5\|\Ad(\sfh_j)(w_i-w_j)\|\leq \|w_{ij}\|\leq2\|\Ad(\sfh_j)(w_i-w_j)\|,
\end{align}
\end{subequations} 
for all $i,j$, see Lemma~\ref{lem:BCH}. 

Let $\{w_{i1}\}$ be defined as in~\eqref{eq:zi-hi-wi}, and let 
\be\label{eq:num-D-y-upper}
F_1\subset\{w_{i1}: 1\leq i\leq \mathsf M\}\quad\text{with}\quad \#F_1=\lceil\beta^{10}\cdot(\#F)\rceil;
\ee
this is possible thanks to~\eqref{eq:num-D-y'}. 
We will show that the claims in the lemma hold with $z_1$ and $F_1$.

\medskip

First note that $h_0z_1\in h_0\cone(h_0)\cap\umt^G.y$ by its definition, and that 
$F_1$ satisfies the claimed properties by its definition and~\eqref{eq:num-D-y-upper}.
Let us now show that part~(1) in the statement of the lemma holds. 
Indeed by~\eqref{eq:zi-hi-wi}, we have 
\[
h_0z_i=\sfh_i\sfh_1^{-1}\sfh_{i1}\exp(w_{i1})h_0z_1\in \Big(\boxH_{10\beta^2}\Big).\exp(w_{i1})h_0z_1\cap h_0\cone(h_0).
\]
Therefore, $\exp(w_{i1})h_0z_1\in (\boxH_{10\beta^2})^{-1}h_0\cone(h_0)\subset \boxH_{100\beta^2}h_0\cone(h_0)$, see~\eqref{eq:B-beta-almost-group} for the last inclusion. 
This establishes the claim in part~(1) of the lemma.

\medskip

For the proof of part~(2) in the statement of the lemma, we need the following. 

\begin{sublemma}
Let 
\[
\cone_1=\coneH.\{\exp(w)h_0z_1: w\in F_1\}.
\]
Let $z\in \cone_1$, and write $z=\sfh u_r\exp(w_{i1})h_0z_1$ where $\sfh\in \boxHs_\beta$, $|r|\leq 0.1\eta$, and $w_{i1}\in F_1$. 
Then
\[
\mfht_{\cone_1}(e,z)\leq 2\mfht_\cone(h_0,z_i)+\beta^{-2}\nuni^{\ell m_0}\cdot(\#F_1) 
\] 
where $z_i\in\cone(h_0)$ is defined as in~\eqref{eq: def z-i}, in particular it satisfies  
\[
h_0z_i=\sfh_i\sfh_1^{-1}\sfh_{i1}\exp(w_{i1})h_0z_1,
\]
see~\eqref{eq:zi-hi-wi}, and $\ell=\lfloor\ref{k:epsilon-t}\vare n\rfloor$.   
\end{sublemma}

Let us first assume the sublemma, and finish the proof of the lemma. 

Recall that $\beta=\nuni^{-\kappa n/2}$ where 
\be\label{eq:kappa-beta}
0<\kappa\leq 0.02\ref{k:epsilon-t}\vare.
\ee
In view of~\eqref{eq:num-D-y'}, we have 
\be\label{eq:num-F1}
\#F_1=\mathsf M\geq \beta^{10}\cdot(\#F)\geq \nuni^{(1-10\kappa)n/2}
\ee 
where we used the bound $\#F\geq \nuni^{n/2}$.  

Recall also that $\ref{k:epsilon-t}m_0\leq1/4$; this estimate and~\eqref{eq:kappa-beta} imply that 
\[
\ref{k:epsilon-t}\vare m_0+\kappa\leq (1-10\kappa)\vare/2.
\]
Using this and~\eqref{eq:num-F1}, we conclude that 
\be\label{eq:num-F1-2}
\nuni^{(\ref{k:epsilon-t}\vare m_0+\kappa)n} \cdot(\#F_1)\leq \nuni^{(1-10\kappa)\vare n/2}\cdot(\#F_1)\leq (\#F_1)^{1+\vare}.
\ee

Let $z\in\cone_1$, and let $z_i\in \cone(h_0)$ be as in the sublemma. Then, by~\eqref{eq:I-cone-isolated-2} we have 
\[
f_\cone(h_0,z_i)\leq \nuni^{Mn-\frac{3\ell}{4}}
\] 
where $\ell=\lfloor\ref{k:epsilon-t}\vare n\rfloor$.
Thus, using the sublemma and~\eqref{eq:num-F1-2} we deduce that 
\begin{align*}
\mfht_{\cone_1}(e,z)&\leq (2e)\cdot \nuni^{(M-\frac{3\ref{k:epsilon-t}\vare}{4})n}+\nuni^{(\ref{k:epsilon-t}\vare m_0+\kappa)n} \cdot(\#F_1)\\
&\leq 6\nuni^{(M-\frac{3\ref{k:epsilon-t}\vare}{4}) n}+(\#F_1)^{1+\vare}.
\end{align*}

We now consider two possibilities. Indeed, if  
$(\#F_1)^{1+\vare}\geq 6\nuni^{(M-\frac{3\ref{k:epsilon-t}\vare}{4}) n}$,
then the above bound implies that 
\[
\mfht_{\cone_1}(e,z)\leq 2(\#F_1)^{1+\vare},
\] 
hence,~\eqref{eq:mfht-ind-stop} holds.

Alternatively, if $(\#F_1)^{1+\vare}< 6\nuni^{(M-\frac{3\ref{k:epsilon-t}\vare}{4}) n}$, then 
\[
\mfht_{\cone_1}(e,z)\leq 7\nuni^{(M-\frac{3\ref{k:epsilon-t}\vare}{4}) n}\leq \nuni^{(M-\frac{2\ref{k:epsilon-t}\vare}{3}) n},
\]
assuming $n\geq n_0$ is large enough. In consequence,~\eqref{eq:mfht-ind-improve} holds. 

These estimate finish the proof of part~(2) and of the lemma, assuming the sublemma.
\end{proof}

\begin{proof}[Proof of the Sublemma]
The proof is similar to the proof of Lemma~\ref{lem:MargFun-enetrgy-rfrak}. 

Let $z\in\cone_1$. Then  
\begin{align}
\notag\mfht_{\cone_1}(e,z)&=\textstyle\sum_{w\in\margI_{\cone_1}(e,z)}\|w\|^{-\alpha}\\
\notag&=\textstyle\sum_{\|w\|\leq \nuni^{-\ell m_0}\beta^2}\|w\|^{-\alpha}+\textstyle\sum_{\|w\|> \nuni^{-\ell m_0}\beta^2}\|w\|^{-\alpha}\\
\label{eq:f-cone1-sum}&\leq\textstyle\sum_{\|w\|\leq \nuni^{-\ell m_0}\beta^2}\|w\|^{-\alpha}+\nuni^{\ell m_0}\beta^{-2}\cdot(\# F_1).
\end{align}

In consequence, we need to investigate the first summation in~\eqref{eq:f-cone1-sum}. 
Let $w\in I_{\cone_1}(e,z)$, then $z,\exp(w)z\in\cone_1$.
In view of the definition of $\cone_1$ and~\eqref{eq:zi-hi-wi}, we may write 
\[
z=\sfh u_r\exp(w_{i1})h_0z_1= \sfh u_r \sfh_{i1}^{-1}\sfh_1\sfh_i^{-1}h_0z_i=\bar\sfh h_0z_i
\]
similarly, $\exp(w)z= \bar\sfh'h_0z_j$ where $1\leq i,j\leq M$ and $\bar\sfh,\bar\sfh'\in\boxH_{0.15\eta}$, see~\eqref{eq:B-beta-almost-group}.

Recall also from~\eqref{eq:zi-hi-wi}, that 
\[
h_0z_j=\sfh_{j}\sfh_i^{-1}\sfh_{ji}\exp(w_{ji})h_0z_i
\] 
where $\sfh_{ji}$ and $w_{ji}$ satisfy~\eqref{eq:hi-hat-wi-1} and~\eqref{eq:hi-hat-wi-2}.
Hence we may apply Lemma~\ref{lem:dist-sheet}, recall that $\beta^2\leq 0.1\injr$, and conclude
\be\label{eq:wij-sublemma}
\|w_{ji}\|\leq 2\|w\|. 
\ee 
Moreover, since $h_0z_k$'s belong to different local $H$-orbits, see~\eqref{eq: def z-i}, 
$w\mapsto w_{ji}$ is well-defined and is one-to-one.

Assume now that $\|w\|\leq \nuni^{-\ell m_0}\beta^2$, then $\|w_{ji}\|\leq 2\nuni^{-\ell m_0}\beta^2$.
This estimate and~\eqref{eq:hi-hat-wi-1} imply that
\[
\|\sfh_{ji}-I\|\leq 2\ref{E:BCH}\beta^2\|w_{ji}\| \leq \nuni^{-\ell m_0}\beta^2
\]
assuming $\beta$ is small enough.
 
Recall also that $\sfh_i,\sfh_j\in\umt^H$ and that~\eqref{eq:thickening-cone-h0} holds for $z_j$.
Therefore, as $h_0\in\supp (\convL)$, 
in particular it is of the form $h_0=a_{\ell \onst} u_r$ for $|r|<2$, we have by~\eqref{eq:well-rd-tau-1} and~\eqref{eq:thickening-cone-h0} that $\sfh_{ji}^{-1}\sfh_i\sfh_j^{-1}h_0z_i\in h_0\cone$.
That yields
\[
\exp(w_{ji})h_0z_i=\sfh_{ji}^{-1}\sfh_i\sfh_j^{-1}h_0z_i\in h_0\cone
\] 
which implies $w_{ji}\in \margI_\cone(h_0,z_i)$ --- recall that $\|w_{ji}\|\leq 2\nuni^{-\ell m_0}\beta^2<\inj(h_0z_i)$. 
This,~\eqref{eq:wij-sublemma}, and the fact that $w\mapsto w_{ji}$ is one-to-one imply that
\[
\textstyle\sum_{\|w\|\leq \nuni^{-\ell m_0}\beta^2}\|w\|^{-\alpha}\leq 2\mfht_\cone(h_0,z_i).
\]

This estimate and~\eqref{eq:f-cone1-sum} finish the proof of the sublemma.
\end{proof}

We also need a lemma which is based on Proposition~\ref{prop:closing-lemma} 
and will provide the base case for our inductive argument in the proof Proposition~\ref{prop:dim-1-C-rfrak}.

\begin{lemma}\label{lem:mfht-base-case}
Let the notation be as in Proposition~\ref{prop:dim-1-C-rfrak}.
In particular, let $0<\eta<0.01\eta_X$, $D\geq D_0$, and $x_0\in X$. There exists $t_1$, depending on $\eta$, $D$, and the injectivity radius of $x_0$, so that the following holds for all $t\geq t_1$.  

Let $0<\vare<0.1$, and let $\beta=\nuni^{-\kappa (t+1)/2}$ 
where $0<\kappa\leq 0.02\ref{k:epsilon-t}\vare$. 
Then at least one of the following holds.  
\begin{enumerate}
\item There exists a subset $F\subset B_\rfrak(0,\beta)$ 
with 
\[
\nuni^{t-5\kappa(t+1)}\leq \#F\leq \nuni^{4t+0.5\kappa(t+1)}
\]
and some 
$y\in X_{2\eta}\cap \Bigl(\boxHs_{\beta}\cdot a_{9t}\Bigr). \{u_rx_0: r\in[0,1.05]\}$
so that if we put
\[
\cone=\coneH.\{\exp(w)y:w\in F\},
\]
then $\cone\subset\Bigl(\boxHs_{10\beta}\cdot a_{9t}\Bigr). \{u_rx_0: r\in[0,1.1]\}$ and 
\[
\mfht_\cone(e,z)\leq \nuni^{D(t+1)}\qquad\qquad\text{for all $z\in\cone$}.
\]
\item There is $x'\in X$ such that $Hx'$ is periodic with
\[
\vol(Hx')\leq \nuni^{D_0\rws}\quad\text{and}\quad\dist_X(x_0,x')\leq \nuni^{(-D+D_0)\rws}.
\] 
\end{enumerate} 
\end{lemma}

\begin{proof}
Put $\mathcal C_0= \{a_{8t}u_rx_0: r\in[0,1]\}$. 
Apply Proposition~\ref{prop:closing-lemma} with $x_0$ and $t$. 
If part~(2) in that proposition holds, then part~(2) above holds and the proof is complete. 
Therefore, let us assume that Proposition~\ref{prop:closing-lemma}(1) holds. 

Let $x\in X_{\rm cpt}\cap\mathcal C_0$ be a point given by Proposition~\ref{prop:closing-lemma}(1); put 
\[
\mathcal C=\Bigl(\boxHs_{\beta}\cdot a_t\Bigr).\{u_rx: r\in[0,1]\}\subset X;
\]
and let $\mathcal C_-=\boxHs_{\beta-100\beta^2}\cdot a_t\cdot \Big\{u_rx: r\in[100\nuni^{-t},1-100\nuni^{-t}]\Big\}$.

Let $\mu_{\mathcal C}$ denote the pushforward to $\mathcal C$ of the normalized restriction of the Haar measure on $H$ to 
$\mathsf C:=\boxHs_{\beta}\cdot a_t\cdot \{u_r:r\in[0,1]\}\subset H$ --- the set $\mathsf C$ was denoted by $\mathsf E_{1,t,\beta}$ in \eqref{eq:def-Ct}, we will use the notation $\mathsf C$ in this proof to avoid confusion with $\mathsf E=\boxHs_{\beta}\cdot \{u_r:|r|\leq 0.1\eta\}$ from \S\ref{sec:def-E-Ecal}.  

We now use arguments similar to, and simpler than, the ones used in Lemmas \ref{lem:E-good-h0} and~\ref{lem:inductive-C} 
to construct the set $\cone$ as in part~(1).

First note that by Proposition~\ref{lem:one-return}, if $t>|\log\eta|+C$ (where $C$ depends on $X$) we have 
\be\label{eq:recurrence-Ccal}
\mu_{\mathcal C}(\mathcal C_-\cap X_{4\eta})\geq 1-O(\sqrt\eta)
\ee
where the implied constant depends on $G$ and $\Gamma$.

Let $\{\boxG_{\beta^2}.\hat y_j:j\in J\}$ be a covering of $X_{4\injr}$ 
so that $J\asymp \beta^{-12}$ where the implied constant depends on $G$ and $\Gamma$, see Lemma~\ref{lem:E-good-h0}. 
Let $J'$ be the set of those $j \in J$ so that 
\be\label{eq:J'-mathcalC}
\mu_{\mathcal C}(\mathcal C_-\cap X_{4\eta}\cap \boxG_{\beta^2}.\hat y_j)\geq \beta^{13}.
\ee
This definition, the fact that $\mu_{\mathcal C}$ is a probability measure (and moreover by~\eqref{eq:recurrence-Ccal} a probability measure giving large measure to $\mathcal C_-\cap X_{4\eta}$) and the estimate~$J\asymp \beta^{-12} $ imply that 
\[
\mu_{\mathcal C}\biggl(\mathcal C_-\bigcap \biggl(\bigcup_{j\in J'}\boxG_{\beta^2}.\hat y_j\biggr)\biggr)\geq 1-O(\sqrt\eta)
\]
where the implied constant depends on $X$. Moreover, \eqref{eq:J'-mathcalC} implies that for any $j \in J'$, \ $\boxG_{\beta^2}.\hat y_j \subset X_{3\eta}$.

Let $j\in J'$; put $\hat y=\hat y_j$ and $\hat{\mathcal C}=\mathcal C_-\cap \boxG_{\beta^2}.\hat y$. 
Then, there are $w_i\in B_\rfrak(0,\beta^2)$ and $\sfh_i\in \boxH_{\beta^2}$, $i=1,\ldots, \mathsf M$, 
so that $\sfh_i\exp(w_i)\hat y\in\mathcal C_-$ and
\[
\hat{\mathcal C}=\bigcup_{i=1}^{\mathsf M} \mathsf C_i\sfh_i\exp(w_i)\hat y
\]
where $\mathsf C_i\subset\boxH_{10\beta^2}$. 

Recall that $\beta\leq \eta^2$ and that $m_H(\mathsf C)\asymp \nuni^t\beta^2$. In consequence, we have  
\[
\mu_{\mathcal C}(\boxH_{10\beta^2})\ll \beta^6\cdot (\nuni^{t}\beta^{2})^{-1}=\beta^{4}e^{-t}.
\]
This and~\eqref{eq:J'-mathcalC} imply that $\mathsf M\gg\beta^{9}\nuni^t$. 
Assuming that $\beta$ is small enough, to account for the implicit constant, we have 
\be\label{eq: M is large initial}
\mathsf M\geq \beta^{10}\nuni^t.
\ee

We now use $\hat{\mathcal C}$ to define $\cone$ which satisfies the desired properties in part~(1). 
To that end, note that for every $i$ and $j$ we have
\begin{align}
\label{eq:cone-mathcalC} \sfh_i\exp(w_i)\hat y&=\sfh_i\exp( w_i)\exp(-w_j)\sfh_j^{-1}\sfh_j\exp(w_j)\hat y\\
\notag&=\sfh_i\sfh_j^{-1}\sfh_{ij}\exp(w_{ij})\sfh_j\exp(w_j)\hat y
\end{align}
where $\sfh_{ij}\in H$ and $w_{ij}\in\rfrak$, $\sfh_{ii}=1$, $w_{ii}=0$ for all $i,j$; moreover, we have
\begin{subequations}
\begin{align}
\label{eq:hi-hat-wi-1-again}&\|\sfh_{ij}-I\|\leq \ref{E:BCH}\beta^2\|w_{ij}\|\qquad\text{and}\\
\label{eq:hi-hat-wi-2-again}&0.5\|\Ad(\sfh_j)(w_i-w_j)\|\leq \|w_{ij}\|\leq2\|\Ad(\sfh_j)(w_i-w_j)\|,
\end{align}
\end{subequations} 
for all $i,j$, see Lemma~\ref{lem:BCH}. In particular, for all $i,j$ we have
\be\label{eq:est-hi-mathcalC}
\|\sfh_{ij}-I\|\ll\beta^4
\ee
for an absolute implied constant. 

Thus, assuming $\beta$ is small enough, we have $\sfh_i\sfh_j^{-1}\sfh_{ij}\in \boxH_{10\beta^2}$, for all $i,j$.
This and the fact that $\sfh_i\exp(w_i)\hat y\in\mathcal C_-$ imply that
\be\label{eq:ratios-in-I}
\begin{aligned}
\exp(w_{ij})\sfh_j\exp(w_j)\hat y&=(\sfh_i\sfh_j^{-1}\sfh_{ij})^{-1}\sfh_i\exp(w_i)\hat y\\
&\in\boxH_{10\beta^2}.\mathcal C_- \subset\mathcal C,
\end{aligned}
\ee
for all $i$ and $j$. 

Let $y:=\sfh_1\exp(w_1)\hat y\in \mathcal C_-\cap X_{2\eta}$ and $F=\{w_{i1}: i=1,\ldots, \mathsf M\}$. First note that 
by~\eqref{eq:ratios-in-I} and Lemma~\ref{lem: explain 4t}, we have 
\[
\#F\ll e^{4t} \leq \beta^{-1}\nuni^{4t}
\]
where in the last inequality we assume $\beta$ is small to account for
the implied constant.  
This and~\eqref{eq: M is large initial} imply that 
\be\label{eq:No-F-mathcalC}
\nuni^{t-5\kappa(t+1)}=\beta^{10}\nuni^t\leq \#F=\mathsf M\leq \beta^{-1}\nuni^{4t}= \nuni^{4t+0.5\kappa(t+1)}
\ee
which is the bound we claimed in part~(1).

Define $\cone=\coneH.\{\exp(w_{i1})y: w_{i1}\in F\}$.
By~\eqref{eq:ratios-in-I}, we have  
$\{\exp(w_{i1})y: w_{i1}\in F\}\subset \boxH_{10\beta^2}.\mathcal C_-$. Recall also that $\coneH=\boxH_\beta\cdot\{u_r: |r|\leq 0.1\eta\}$ and
\be\label{eq:ur boxH at}
u_r\cdot \boxH_{\beta}\cdot a_{t}\subset \boxH_{2\beta}\cdot a_t\cdot u_{\nuni^{-t}r},
\ee
for all $|r|\leq 0.1\eta$. 
Thus 
\begin{align*}
\cone&=\boxH_\beta\cdot\{u_r: |r|\leq 0.1\eta\}.\{\exp(w_{i1})y: w_{i1}\in F\}\\
&\subset\boxH_\beta\cdot\boxH_{2\beta}\cdot a_{t}. \{u_rx: r\in[0,1]\}\\
&\subset\boxHs_{5\beta}\cdot a_{t}. \{u_rx: r\in[0,1]\}\\
&\subset\Bigl(\boxHs_{5\beta}\cdot a_{t}\cdot \{u_r: r\in[0,1]\}\Bigr)\cdot a_{8t}. \{u_rx_0: r\in[0,1]\}\\
&\subset \boxHs_{5\beta}\cdot a_{t}\cdot\mathsf B^s_{5\beta}\cdot\{u_r: |r|\leq 2\}\cdot a_{8t}. \{u_rx_0: r\in[0,1]\}.
\end{align*}
where $\mathsf B_{\varrho}^s=\{u^-_s: |s|\leq \varrho\}\cdot\{a_d: |d|\leq \varrho\}$ and we use $x\in \mathcal C_0$ in the third line. Using $u_ra_{8t}=a_{8t}u_{e^{-8t}r}$, 
which holds for all $r$ and $t$, we conclude 
\[
\cone\subset \boxHs_{5\beta}\cdot a_t\cdot \mathsf B^s_{5\beta}\cdot a_{8t}. \{u_rx_0: r\in [0,1.1]\},
\]
so long as $t\geq 1$. 

Finally note that 
$a_t\mathsf B^s_{2\beta}a_{-t}=\{u^-_s: |s|\leq 2\nuni^{-t}\beta\}\cdot\{a_\ell: |l|\leq 2\beta\}$
for all $t$. Thus assuming $t$ is large enough, we have 
\[
\cone\subset \boxHs_{10\beta}\cdot a_{9t}\cdot \{u_rx_0: r\in [0,1.1]\}.
\]

We claim  
\be\label{eq:initial-bd-cone-again}
\mfht_{\cone}(e,z)\leq 2\nuni^{Dt}\leq \nuni^{D(t+1)}\qquad\text{for all $z\in\cone$}.
\ee
In view of the above discussion, this estimate finishes the proof of part~(1) and of the lemma modulo~\eqref{eq:initial-bd-cone-again}.  

\medskip

The proof of~\eqref{eq:initial-bd-cone-again} is similar to the proof of Lemma~\ref{lem:MargFun-enetrgy-rfrak}. 
 For every $1\leq i\leq \mathsf M$, put $z_i=\sfh_i\exp(w_i)\hat y$.
Let $w\in I_{\cone}(e,z)$, then $z,\exp(w)z\in\cone$.
In view of the definition of $\cone$ and~\eqref{eq:cone-mathcalC}, we may write 
\[
z=\sfh u_r\exp(w_{i1})y=\sfh u_r (\sfh_i\sfh_1^{-1}\sfh_{i1})^{-1}z_i=\bar\sfh z_i
\]
similarly, $\exp(w)z= \bar\sfh'z_j$ where $1\leq i,j\leq M$ and $\bar\sfh,\bar\sfh'\in\boxH_{0.15\eta}$, see~\eqref{eq:est-hi-mathcalC} and~\eqref{eq:B-beta-almost-group}. Recall also from~\eqref{eq:cone-mathcalC} again that 
\[
z_j=\sfh_j\sfh_i^{-1}\sfh_{ji}\exp(w_{ji})z_i
\] 
where $\sfh_{ji}$ and $w_{ji}$ satisfy~\eqref{eq:hi-hat-wi-1-again} and~\eqref{eq:hi-hat-wi-2-again}.
Hence we may apply Lemma~\ref{lem:dist-sheet}, recall that $\beta\leq \injr^2$, and conclude
\be\label{eq:wij-sublemma-again}
\|w_{ji}\|\leq 2\|w\|. 
\ee 
Moreover, since $\sfh_k\exp(w_k)\hat y$'s belong to different local $H$-orbits, $w\mapsto w_{ji}$ is well-defined and one-to-one.
Recall also from~\eqref{eq:ratios-in-I} that 
\[
(\sfh_j\sfh_i^{-1}\sfh_{ji})^{-1}z_j=\exp(w_{ji})z_i\in \mathcal C,
\]
for all $i,j$. Moreover by~\eqref{eq:hi-hat-wi-2-again}, we have $\|w_{ji}\|\ll \beta^2\leq \inj(z_i)$. 
Altogether, we conclude that $w_{ji}\in \margI_{\mathcal C}(e,z_i)$. 

This,~\eqref{eq:wij-sublemma-again}, and the fact that $w\mapsto\hat w_{ji}$ is one-to-one imply that
\begin{align*}
\mfht_{\cone}(e,z)&=\textstyle\sum_{w\in \margI_{\cone}(e,z)}\|w\|^{-\alpha}\\
&\leq 2\textstyle\sum_{w\in \margI_{\mathcal C}(e,z_i)}\|w\|^{-\alpha}\\
&=2\mfht_{\mathcal C}(e,z_i)\leq 2\nuni^{Dt},
\end{align*}
where the last inequality is a consequence of Proposition~\ref{prop:closing-lemma}(1). 
\end{proof}

\subsection*{Proof of Proposition~\ref{prop:dim-1-C-rfrak}}
We now complete the proof of Proposition~\ref{prop:dim-1-C-rfrak}. 
Roughly speaking, the proof is based on repeatedly applying Lemma~\ref{lem:inductive-C}
to improve the bound on the corresponding Margulis function. 

Let $0<\eta<0.01\eta_X$, $D\geq D_0+1$ (for $D_0$ as in Proposition~\ref{prop:closing-lemma}), $x_0 \in X$, and $t>0$ (large) be as in the statement of Proposition~\ref{prop:dim-1-C-rfrak}.  

Fix some $\kappa$ satisfying 
\be\label{eq:kappa-beta'}
0<\kappa\leq \frac{\ref{k:epsilon-t}\vare}{100D},
\ee
and put $\beta=\nuni^{-\kappa (t+1)/2}$. 

We assume $t$ is large enough so that $\beta\leq \eta^2$;
assume further that $t\geq t_1$ where $t_1$ is as in Lemma~\ref{lem:mfht-base-case}.

\subsection*{Base of the induction}
Apply Lemma~\ref{lem:mfht-base-case} with $\eta$, $\beta$, $D$, $x_0$, and $t$. 
If Lemma~\ref{lem:mfht-base-case}(2) holds, then Proposition~\ref{prop:dim-1-C-rfrak}(2) holds and the proof is complete.
Therefore, we assume that Lemma~\ref{lem:mfht-base-case}(1) holds. Let 
\be\label{eq:cone0-near-orbit}
\cone=\coneH.\{\exp(w)y: w\in F\}\subset\\ \boxHs_{10\beta}\cdot a_{9t}\cdot \{u_rx_0: r\in[0,1.1]\}
\ee
be as in Lemma~\ref{lem:mfht-base-case}(1).  
Put $n=t+1$, $M=D$, $y_0=y$, $F_0=F$, and $\cone_0=\cone$.   
We further assume $t+1\geq 4n_0$ where $n_0$ is as in Lemma~\ref{lem:noI-inh-C}.

Apply Lemma~\ref{lem:noI-inh-C} with this $\cone_0$. 
If Lemma~\ref{lem:noI-inh-C}(1) holds, then $\nuni^{Mn}\leq \nuni^{\vare n/2}\cdot(\#F_0)$.
Since $\#F_0\geq \nuni^{t-5\kappa(t+1)}\geq \nuni^{n/2}$,
we have 
\[
\mfht_{\cone_0}(e,z)\leq \nuni^{Mn}\leq \nuni^{\vare n/2}\cdot(\#F_0)\leq (\#F_0)^{1+\vare}.
\]
Hence by Lemma~\ref{lem:MargFun-enetrgy-rfrak}, for all $w\in F_{0}$,
\[
\textstyle\sum_{w\neq w'}\|w-w'\|^{-\alpha}\leq 4\cdot (\#F_{0})^{1+\vare}.
\]
This estimate together with~\eqref{eq:cone0-near-orbit} implies that part~(1) in the proposition holds with $\tau=9t$, $x_1=y$ and $F=F_{0}$ if we choose $R$ large enough so that
$e^{-t/R}\geq 10\beta$.

\subsection*{The inductive step}
In view of the above discussion, let us assume that Lemma~\ref{lem:noI-inh-C}(2) holds for $\cone_0$. 
Let $L_{\cone_0}$ be as in Lemma~\ref{lem:EG-Cheby-C}.
Let $h_0\in L_{\cone_0}$, and let $y_j$ for some $j\in\mathcal J(h_0)$ be as in Lemma~\ref{lem:E-good-h0}.
Moreover, note that 
\[
\nuni^{n/2}\leq \nuni^{t-5\kappa(t+1)}\leq \#F_0\leq 
\nuni^{4t+0.5\kappa(t+1)}=\beta^{-1}\nuni^{4t},
\] 
and $n>n_0$. 
Therefore, we may apply Lemma~\ref{lem:inductive-C}. By that lemma, there exist $z_1$ with
\[
h_0z_1\in h_0\cone_0(h_0)\cap\umt^G.y_j
\]
and a subset $F_1\subset B_\rfrak(0,\beta)$, containing $0$, with
\[
\# F_1=\lceil\beta^{10}\cdot(\#F_0)\rceil 
\]
so that both of the following are satisfied. 
\begin{enumerate}[label={(I-\arabic*)}]
\item\label{(I-1)} For all $w\in F_1$, we have 
\[
\exp(w)h_0z_1\in \boxH_{100\beta^2}.h_0\cone_0(h_0).
\] 
\item If we put $\cone_1=\coneH.\{\exp(w)h_0z_1: w\in F_1\}$, then at least one of the following properties hold:
\begin{subequations}
\begin{align}
\label{eq:mfht-ind-stop-use}\mfht_{\cone_1}(e,z)&\leq 2\cdot (\#F_1)^{1+\vare}&&\text{for all $z\in\cone_1$, or}\\
\label{eq:mfht-ind-improve-use}\mfht_{\cone_1}(e,z)&\leq \nuni^{(M-\frac{2\ref{k:epsilon-t}\vare}{3}) n}&&\text{for all $z\in\cone_1$}.
\end{align}
\end{subequations}
\end{enumerate}

\medskip

If~\eqref{eq:mfht-ind-stop-use} holds, we set $\cone_{\rm fin}=\cone_1$. 
Otherwise, we repeat the above construction to define sets $F_2,\ldots$ and the corresponding $\cone_2,\ldots$.

Let $i_{\rm max}:=\lfloor\frac{6M-3}{4\ref{k:epsilon-t}\vare}\rfloor+1$, then by the choice of $\kappa$ in~\eqref{eq:kappa-beta'}, we have 
\be\label{eq:why-choose-kappa}
M-\tfrac{2\ref{k:epsilon-t}\vare }{3}i_{\rm max}\leq1/2\quad\text{and}\quad 5\kappa (i_{\rm max}+1)\leq 1/4
\ee

Suppose now that $i\leq i_{\rm max}$,
and we have constructed $\cone_0,\ldots, \cone_i$ so that~\eqref{eq:mfht-ind-stop-use} does {\em not} hold for $\cone_k$, 
for all $0\leq k\leq i$. 
Then~\eqref{eq:mfht-ind-improve-use} holds and we have 
\be\label{eq:eq:mfht-ind-improve-step-i}
\mfht_{\cone_k}(e,z)\leq \nuni^{(M-\frac{2\ref{k:epsilon-t}\vare}{3}k) n}\qquad\text{for all $0\leq k\leq i$ and all $z\in\cone_k$.}
\ee 

By the second estimate in~\eqref{eq:why-choose-kappa}, for all $0\leq k\leq i$, we have 
\begin{align*}
\#F_k\geq\beta^{10k}\cdot(\#F_0)&\geq \nuni^{t-5\kappa(k+1)(t+1)}\\
&\geq \nuni^{(3t-1)/4}\geq \nuni^{2n/3}. 
\end{align*}
Since~\eqref{eq:mfht-ind-stop-use} does not hold for $\cone_k$, but~\eqref{eq:mfht-ind-improve-use} holds,
we have 
\[
\nuni^{\vare n/2}\cdot(\#F_k)\leq (\#F_k)^{1+\vare}\leq \nuni^{(M-\frac{2\ref{k:epsilon-t}\vare}{3}k) n}
\]
for all $0\leq k\leq i$. 

Thus we are in case Lemma~\ref{lem:noI-inh-C}(2) for all these $k$, moreover, 
we have the bound $\#F_k\geq \nuni^{2n/3}$. In consequence, Lemma~\ref{lem:inductive-C} is applicable in every step, and we can define $F_{i+1}$ and $\cone_{i+1}$.

\subsection*{The conclusion of the proof}
We now show that in at most $i_{\rm max}$ many steps, we obtain a set $\cone$ 
which satisfies \ref{(I-1)} above and \eqref{eq:mfht-ind-stop-use}. 
Indeed, in view of the first estimate in~\eqref{eq:why-choose-kappa},   
\[
\nuni^{(M-\frac{2\ref{k:epsilon-t}\vare}{3}i_{\rm max}) n}< \nuni^{n/2}.
\]
As $\#F_k \geq \nuni^{2n/3}$ for all $F_k$'s which are constructed, 
this observation together with~\eqref{eq:eq:mfht-ind-improve-step-i} implies that 
in at most $i_{\rm max}$ number of steps,~\eqref{eq:mfht-ind-stop-use} holds. 

In consequence, we get some $i_{\rm fin}\leq i_{\rm max}$, so that if we put $F_{\rm fin}:=F_{i_{\rm fin}}\subset B_\rfrak(0,\beta)$, then $\#F_{\rm fin}\geq \nuni^{2n/3}$, and the set 
\[
\cone_{\rm fin}= \coneH.\{\exp(w)y_{\rm fin}: w\in F_{\rm fin}\}
\] 
satisfies
\begin{equation}
  \label{eq:mfht-est-final}\mfht_{\cone_{\rm fin}}(e,z)\leq 2\cdot (\#F_{\rm fin})^{1+\vare}  
\end{equation}
for all $z\in\cone_{\rm fin}$ (cf. \eqref{eq:mfht-ind-stop-use}).

\medskip 

We claim that $ F_{\rm fin}$ and $y_{\rm fin}$ also satisfy
\begin{equation}
\label{eq:fiber-near-final}\{\exp(w)y_{\rm fin}: w\in F_{\rm fin}\}\subset \Bigl(\boxH_{100(i_{\rm fin}+10)\beta} \cdot a_\tau\cdot \{u_r: |r|\leq 4\}\Bigr).x_0\cap X_\eta,
\end{equation}
with $\tau$ satisfying
\begin{equation}
    \label{eq:bound on tau}
    9t\leq \tau=9t+i_{\rm fin}\ref{k:epsilon-t}\vare m_0(t+1)\leq 9t+2m_0Mt=9t+2m_0Dt.
\end{equation}
Let us first assume~\eqref{eq:fiber-near-final} and finish the proof 
of the proposition.   

First note that using the above definitions, we have  
\[
\nuni^{t/2}\leq \#F_{\rm fin}\leq \#F_0\leq \beta^{-1}\nuni^{4t}\leq e^{5t}.
\]
The assertion~\eqref{eq:mfht-est-final} and Lemma~\ref{lem:MargFun-enetrgy-rfrak} imply that for all $w\in F_{\rm fin}$,
\[
\textstyle\sum_{w\neq w'}\|w-w'\|^{-\alpha}\leq 4\cdot (\#F_{\rm fin})^{1+\vare}.
\]
This estimate together with~\eqref{eq:fiber-near-final} implies that part~(1) in the proposition holds with $x_1=y_{\rm fin}$ and $F=F_{\rm fin}$ if we choose $R$ large enough so that
$e^{-t/R}\geq100(i_{\rm fin}+10)\beta$.
This concludes the proof of Proposition~\ref{prop:dim-1-C-rfrak} modulo the proof of~\eqref{eq:fiber-near-final}.

\medskip
To see that~\eqref{eq:fiber-near-final} holds, note that
 at every step, the element $h_0$ is of the form $a_{m_0\ell}u_{r_k}$ where $r_k\in [0,1]$ and $\ell=\lfloor \ref{k:epsilon-t}\vare (t+1)\rfloor$.
Now for all $0\leq k<i_{\rm fin}$, we have 
\be\label{eq:commutation-pf-marg-fun}
\cone_{k+1}\subset \mathsf B_{2\beta}^s \cdot a_{m_0\ell}u_{r_k}\cdot\{u_{\bar r}: |\bar r|\leq 2\nuni^{-m_0\ell}\}\cdot \cone_k.
\ee
where $\mathsf B_{\varrho}^s=\{u^-_s: |s|\leq \varrho\}\cdot\{a_d: |d|\leq \varrho\}$.
To see this note that by \ref{(I-1)}, we have 
\[
\{\exp(w)x_1: w\in F_{k+1}\}\subset \boxH_{100\beta^2}\cdot a_{m_0\ell}u_{r_k}.\cone_k.
\]
Now for every $|r|\leq 1$, $\skew{-3}\hat\sfh\in \boxH_{\beta}$ and $\sfh\in \boxHs_{100\beta^2}$, we have $\skew{-3}\hat\sfh u_r\sfh=\sfh'u_{r'}$ 
where $\sfh'\in \mathsf B_{2\beta}^s$
and $|r'|\leq 2$; moreover, $u_{r'}a_{m_0\ell}=a_{m_0\ell} u_{\nuni^{-m_0\ell}r'}$. Assuming $\ell\geq 5$, which may be guaranteed by taking $t$ large, and using the definition 
\[
\cone_{i+1}=\coneH.\{\exp(w)x_1: w\in F_{i+1}\},
\]  
the inclusion in~\eqref{eq:commutation-pf-marg-fun} follows.

Arguing similarly,~\eqref{eq:cone0-near-orbit} implies that 
\[
\cone_0\subset \mathsf B_{10\beta}^s\cdot a_{9t}\cdot \{u_rx_0: r\in[0,1.15]\}.
\] 
Using the fact that $a_{m_0\ell}\mathsf B_{\varrho}^sa_{-m_0\ell}\subset \mathsf B_{\varrho}^s$ and arguing inductively, 
\[
\cone_{i+1}\subset \boxH_{100(i_{\rm fin}+10)\beta} \cdot (a_{m_0\ell}u_{\hat r_{i+1}} \mathsf U_{i+1})\cdots (a_{m_0\ell} u_{\hat r_1}\mathsf U_1)\cdot \{a_{9t}u_r: |r|\leq 2\}.x_0
\]
where $\hat r_k\in[0,1]$ and $\mathsf U_k=\{u_{\bar r}: |\bar r|\leq 100(k+10)\beta\}$. 
Moreover, for every $i\leq i_{\rm max}$,
\[
(a_{m_0\ell}u_{\hat r_{i+1}} \mathsf U_{i+1})\cdots (a_{m_0\ell} u_{\hat r_1}\mathsf U_1)\subset a_{m_0(i+1)\ell}\cdot u_{\hat r}\cdot \{u_{\bar r}: |\bar r|\leq 10^4\beta\}
\]
where $\hat r=\sum \nuni^{-m_0(k-1)\ell}\hat r_k\in [0,1.5]$. 

This implies~\eqref{eq:fiber-near-final} except for the bound \eqref{eq:bound on tau} on~$\tau$. 
To see the claimed bound on $\tau$, note that
\[
i_{\rm max}\ell\leq (\tfrac{6M-3}{4\ref{k:epsilon-t}\vare}+1)\ref{k:epsilon-t}\vare (t+1)\leq 2Mt
\]
which implies the bound on $\tau$.  
\qed

\section{Proof of the main theorem}\label{sec:proof-main-prop}\label{sec:proof-main}
In this section we will complete the proofs of Proposition~\ref{prop:main-prop} and Theorem~\ref{thm:main}.

\subsection{Proof of Proposition~\ref{prop:main-prop}}
Let $D_0$ be as in Proposition~\ref{prop:closing-lemma}, 
and choose $D\geq 2D_0$ so that $\delta/2\leq D_0/(D-D_0)\leq\delta$.

Let $\eta_0=0.01\eta_X$, and let $0<\eta<\eta_0$. 
Let $x_1\in X_\eta$, and let $t_0$ be as in 
Proposition~\ref{prop:dim-1-C-rfrak} applied with $D$ and $\eta$.

Define $t$ by $T=\nuni^{(D-D_0)\rws}$, and let $T_1$ be so that $T\geq T_1$ implies $t\geq t_0$.

We may assume that Proposition~\ref{prop:dim-1-e}(1) holds. Indeed, if 
Proposition~\ref{prop:dim-1-e}(2) holds, then since 
$\nuni^{D_0t}=T^{{D_0}/(D-D_0)}$ and $\delta/2\leq D_0/(D-D_0)\leq\delta$, Proposition~\ref{prop:main-prop}(2) holds and the proof is complete. 

\medskip 

Let $0<\theta<1/2$ be arbitrary. Apply Proposition~\ref{prop:dim-1-e}(1) 
with $\vare=0.01\theta$ and $\alpha=1-\vare$. Without loss of generality, we will further assume that 
$T_1$ is large enough so that $\nuni^{-\vare t/2}\leq (2\ref{E:BCH}\ref{E:non-div-main})^{-1}\eta^3$, this is motivated by~\eqref{eq:num-elemnts-F-proj}.

By Proposition~\ref{prop:dim-1-e}(1), there exists $R>0$, depending on $D$ and $\theta$, so that the following holds.
There exist $x_1\in X_{\injr}$, some $9t\leq \tau\leq 9t+2m_0 D t$ (where $m_0$ depends on $\theta$ as in \eqref{eq:EMM-use'}), and a subset $F\subset B_\rfrak(0,1)$, 
containing $0$, with $\nuni^{t/2}\leq \#F\leq \nuni^{5t}$, 
so that both of the following properties are satisfied. 
\begin{subequations}
\begin{align}
\label{eq:prop-F-final-proof-1}&\{\exp(w)x_1: w\in F\}\subset \Bigl(\boxHs_{\nuni^{-t/R}} \cdot a_{\tau}.\{u_rx_0: |r|\leq 4\}\Bigr)\cap X_\injr\;\;\text{and}\\
\label{eq:prop-F-final-proof-2}&\textstyle\sum_{w'\neq w}\|w-w'\|^{-\alpha}\ll (\#F)^{1+\vare}\qquad\text{for all $w\in F$},
\end{align}
\end{subequations}
where the implied constant depends on $X$.

Now apply Proposition~\ref{prop:proj-general} with $\eta$, $\vare$, $\alpha=1-\vare$, $x_1$, and $F$; 
note that~\eqref{eq:num-elemnts-F-proj} is satisfied since $\#F\geq \nuni^{t/2}$.
Let  
\be\label{eq:where is x2}
x_2 \in X_\eta\cap a_{|\log\rhsco|}.\{u_r\exp(w)x_1: |r|\leq 2, w\in F\},
\ee
$I \subset [0,1]$, $\rhsco>0$, and the probability measure $\rho$ on $I$ be as in that proposition. In particular, we have 
\be\label{eq:b1-est-final-pf}
\nuni^{-5t}\leq (\#F)^{-\frac{2+6\vare}{2+21\vare}}\leq \rhsco\leq (\#F)^{-\vare},
\ee
and the following hold
\begin{subequations}
\begin{align}
\label{eq:prop:proj-general-use-1}&\rho(J)\leq C'_\vare|J|^{\alpha-30\vare}\qquad\text{for all $|J|\geq (\#F)^{\frac{-15\vare}{2+21\vare}}$}\\
\label{eq:prop:proj-general-use-2}&v_sx_2\in \boxG_{C\rhsco} \cdot a_{|\log\rhsco|}.\{u_r\exp(w)x_1: |r|\leq 2, w\in F\}\quad\text{for all $s\in I$},
\end{align}
\end{subequations} 
where $C$ is an absolute constant.

Set $\kappa:=\frac{\vare}{4D_0}=\frac{\theta}{400D_0}$. 
Since $\#F\geq \nuni^{t/2}$, we have 
\be\label{eq:b0-est-final}
(\#F)^{\frac{-15\vare}{2+21\vare}}\leq (\#F)^{-\vare}\leq \nuni^{-\vare t/2}\leq T^{-\delta\vare/4D_0}=T^{-\delta\kappa};
\ee
recall that $\delta/2\leq D_0/(D-D_0)\leq \delta$ and $T=\nuni^{(D-D_0)t}$. 

Combining~\eqref{eq:b0-est-final} and equation~\eqref{eq:prop:proj-general-use-1}, we conclude that 
\be\label{eq:proof-main-prop-1-a}
\rho(J)\leq C'_\vare|J|^{\alpha-30\vare}\leq C'_\vare|J|^{1-\theta},\quad\text{for all intervals $J$ with $|J|\geq T^{-\delta\kappa}$.}
\ee
This establishes Proposition~\ref{prop:main-prop}(1)(a) if we put $C_{\theta}=C'_\vare$.

Let us now turn to the proof of Proposition~\ref{prop:main-prop}(1)(b). 
We first claim that 
\be\label{eq:commutation-pf-main}
\{u_r\exp(w)x_1: |r|\leq 2, w\in F\}\subset \mathsf B_{10\varrho}^s \cdot a_{\tau}\cdot \{u_rx_0: |r|\leq 9/2\}.
\ee
where $\varrho=\nuni^{-t/R}$ and $\mathsf B_{\varrho}^s=\{u^-_d: |d|\leq \varrho\}\cdot\{a_\ell: |\ell|\leq \varrho\}$.
To see this, first note that using~\eqref{eq:prop-F-final-proof-1}, we have 
\[
\{\exp(w)x_1: w\in F\}\subset \boxH_{\varrho}\cdot a_{\tau}\cdot \{u_rx_0: |r|\leq 4\}.
\]
Now for every $|r|\leq 2$ and $h\in \boxHs_{\varrho}$, we have $u_rh=h'u_{r'}$ where $h'\in \mathsf B_{10\varrho}^s$
and $|r'|\leq 3$; moreover, $u_{r'}a_\tau=a_\tau u_{\nuni^{-\tau}r'}$. The claim follows as $\tau\geq 2$.   

Combining~\eqref{eq:commutation-pf-main},~\eqref{eq:prop:proj-general-use-2}, and~\eqref{eq:where is x2} for all $s\in I\cup\{0\}$ we have
\be\label{eq:main-prof-1b-pf}
\begin{aligned}
v_sx_2&\in \boxG_{C\rhsco} \cdot a_{|\log\rhsco|}\cdot \{u_r\exp(w)x_1: |r|\leq 2, w\in F\}\\
&\in \boxG_{C\rhsco} \cdot a_{|\log\rhsco|}\cdot \mathsf B_{10\varrho}^s \cdot a_{\tau}\cdot \{u_rx_0: |r|\leq 9/2\}.
\end{aligned}
\ee

By the definition of $\mathsf B_{10\varrho}^s$ above, we conclude that
\[
a_{|\log\rhsco|}\mathsf B_{10\varrho}^sa_{-|\log\rhsco|}\subset\{u^-_d: |d|\leq b_1\}\cdot\{a_\ell: |\ell|\leq 10\varrho\}.
\]
This and~\eqref{eq:main-prof-1b-pf} imply that 
\be\label{eq:main-prof-1b-pf'}
v_sx_2\in \boxG_{C'b_1}\cdot\Bigl(\{a_\ell: |\ell|\leq 10\varrho\}\cdot a_{\tau+|\log\rhsco|}\cdot \{u_r: |r|\leq 9/2\}\Bigr).x_0.
\ee

Recall that $b_1\leq (\#F)^{-\vare}\leq \nuni^{-\vare t/2}\leq T^{-\vare \delta/4D_0}$ and  $\varrho=\nuni^{-t/R}$. Moreover, note that the bound $\nuni^{-6t}\leq b_1$ in~\eqref{eq:b1-est-final-pf} and $\tau\leq 9t+2m_0 D t$ imply 
\[
\nuni^{(\tau+|\log\rhsco|)/2}\leq \nuni^{\tau}\leq \nuni^{9t+2m_0 Dt}\leq T^{A'-1},
\]
for $A'$ depending only on $\theta$. 
Hence, in view of~\eqref{eq:main-prof-1b-pf'}, we have 
\[
d_X\Big(v_sx_2,B_P\Big(e,\rH^{A'}\Big).x_0\Bigr)\ll_X T^{-\delta\vare/4D_0},
\] 
for all $s\in I\cup\{0\}$.

The above and~\eqref{eq:b0-est-final} finish the proof of the proposition if we let $y_0=x_2$ and $\ref{k:main-prop}=\frac{\vare}{4D_0}=\frac{\theta}{400D_0}$.
\qed

\subsection{Proof of Theorem~\ref{thm:main}}\label{sec:proof-main-1}
Let $\theta=\vare_0/2$ where $\vare_0$ is given by Proposition~\ref{prop:1-epsilon-N}.

Apply Proposition~\ref{prop:main} with $x_0$, $\theta$, $\eta=10^{-4}\eta_X$, and the given $\delta$.
Let $T >T_1$ where $T_1$ is as in Proposition~\ref{prop:main}.

If Proposition~\ref{prop:main}(2) holds, then Theorem~\ref{thm:main}(2) holds and we are done. 
Therefore, let us assume that Proposition~\ref{prop:main}(1) holds. 
Let $y_0$, $I$, and $\rho$ be as in Proposition~\ref{prop:main}(1). 

Let $0<\varrho<0.1\eta_X$, and let $z\in X_{\varrho}$. There is a function $f_{\varrho,z}$ 
supported on $\boxG_{0.1\varrho}.z$ with 
$\int f_{\varrho,z}\diff\!{m_X}=1$ and $\Scal(f_{\varrho,z})\leq \varrho^{-N}$, where $N$ is absolute. 

Let $b=T^{-\delta\ref{k:main-prop}}$, and let $t=|\log b|/4$. 
In view of Proposition~\ref{prop:main}(1), $\rho$ satisfies~\eqref{eq:C-rho-reg-N} with $C_{\theta}$. 

Apply Proposition~\ref{prop:1-epsilon-N}, with $f=f_{\varrho,z}$ for $\varrho=\nuni^{-\ref{k:mixing} t/2N}$. Then
\[
\bigg|\iint f(a_tu_rv_s.y_0)\diff\!\rho(s)\diff\!r-1\bigg|\ll_{C_{\theta}} \Scal(f)e^{-\ref{k:mixing} t}\ll_{C_{\theta}} e^{-\ref{k:mixing} t/2};
\]
where we used $\eta=10^{-4}\eta_X$, hence the dependence on $\eta$ in Proposition~\ref{prop:1-epsilon-N} can be absorbed in the implicit constant.

Assuming $T$ is large enough, depending on $\theta$, the right side of the above is $<1/2$.
Thus $a_tu_rv_s.y_0\in\supp(f)$ for some $r\in[0,1]$ and $s\in I$.    

Let $\constk\label{k:cusp-mixing}=\ref{k:mixing}/8N$. The above thus implies that 
\be\label{eq:y0-dense-final}
d_X\Big(z,a_t.\Bigl\{u_rv_sy_0: r\in[0,1], s\in I\Bigr\}\Big)\ll b^{\ref{k:cusp-mixing}} 
\ee
for all $z\in X_{b^{\ref{k:cusp-mixing}}}$.

Moreover, by Proposition~\ref{prop:main}(1), we have 
\[
d_X\Big(u_rv_s.y_0, \Big(u_r\cdot B_P(e,\rH^{A'})\Big).x_0\Big)\leq \ref{c:main-2}'b,
\]
for all $s\in I\cup\{0\}$ and $r\in [0,1]$.  
Note also that if $z,z'\in X$ satisfy,  
$d(z,z')\leq \ref{c:main-2}'b$, then $d_X(a_tz,a_tz')\ll b^{1/2}$. In consequence,    
\be\label{eq:dist-close-final}
d_X\Big(a_t.\Bigl\{u_rv_sy_0: r\in[0,1], s\in I\Bigr\}, B_P(e,\rH^{A'+1}).x_0\Big)\ll b^{1/2},
\ee
where we used 
\[
a_t\cdot \{u_r: r\in[0,1]\}\cdot B_P(e,\rH^{A'})\subset B_P(e,\rH^{A'+1}),
\]
which in turn follows from  $t=|\log b|/4$ and $b=T^{-\delta\ref{k:main-prop}}$.

Combining~\eqref{eq:y0-dense-final} and~\eqref{eq:dist-close-final}, we conclude that 
\[
d_X(z,B_P(e,\rH^{A'+1}).x_0)\ll b^{\ref{k:cusp-mixing}}=T^{-\delta\ref{k:main-prop}\ref{k:cusp-mixing}}
\] 
for all $z\in X_{b^{\ref{k:cusp-mixing}}}$, where the implied constant depends on $X$. 
This implies Theorem~\ref{thm:main}(1) with $\ref{k:main-1}=\ref{k:main-prop}\ref{k:cusp-mixing}$.

As was remarked in \S\ref{sec:Exp-Mixing}, $\mixexp$ in~\eqref{eq:actual-mixing} is absolute if $\Gamma$ is a congruence subgroup, see~\cite{Burger-Sarnak, Cl-tau, Gor-Mau-Oh}. Hence, if $\Gamma$ is assumed to be a congruence subgroup, then $A$ and $\ref{k:main-1}$ only depend on $\Gamma$ via~\eqref{eq: dimension of Gamma}. 
\qed

\section{Proof of Theorem~\ref{thm:main-closed}}\label{sec:proof-closed-orbit}

Let $\eta_X$ be as in Proposition~\ref{prop:non-div} and  $\ref{E:non-div-main}$ as in Proposition~\ref{prop:one-return}. Define 
\be\label{eq:def cX closed}
\cX=\eta_X^{-1}\,\vol(G/\Gamma)\,\nuni^{\ref{E:non-div-main}}
\ee
where $\vol(G/\Gamma)$ is computed using the Riemannian metric $d$, see also~\eqref{eq: cX Mixing section}.  

For $0<\alpha<1$ choose an $m_\alpha>0$ as in~\eqref{eq:EMM-use}, i.e., $m_\alpha$ satisfies that 
\begin{equation}\label{eq:EMM-use''}
\ave{\|a_{{m_\alpha}}\uvk w\|^{-\alpha}}\uvkd\leq \nuni^{-1}\|w\|^{-\alpha}\qquad\text{for all $w\in\gfrak$}.
\end{equation}

In this section, the notation $a\ll_X b$ means $a\leq L\cX^L\, b$ where $L$ is an absolute constant. Similarly, $a\ll_{X,\alpha}b$ means 
\be\label{eq:a ll X alpha b}
a\leq L\cX^L\nuni^{Lm_\alpha}\, b
\ee 
where $L$ is an absolute constant. Define $a\gg_Xb$ and $a\gg_{X,\alpha}$ accordingly. 

\medskip
Throughout this section, $Y=Hx$ is a periodic orbit. 
Let $\mu_{Hx}$ denote the probability $H$-invariant measure on $Hx$. We put $\vol(Y)=\vY$. In view of Lemma~\ref{lem:non-div-closed}, we have  $\vY\gg_X1$. 
The following proposition is our replacement for Proposition~\ref{prop:dim-1-e} in the setting at hand.

\begin{propos}\label{prop:dim-1-e-closed}
Let $0<\alpha<1$. 
There exists $y_0\in Y$ and  
a subset $F\subset B_\rfrak(0,1)$, containing $0$, with 
$\#F\gg_{X} \vY$ so that both of the following properties are satisfied:
\begin{enumerate}[label=(\ref{prop:dim-1-e-closed}-\alph*)]
    \item \label{item:dim-1-e-closed-a} $\Bigl \{\exp(w)y_0: w\in F\Bigr\}\subset Y\cap X_{\rm cpt}$, see \S\ref{sec:SiegelSet} for the definition of $X_{\rm cpt}$.
    \item $\sum_{w'\neq w}\|w-w'\|^{-\alpha}\ll_{X,\alpha} \#F$ for all $w\in F$ where the summation is over $w' \in F$.
\end{enumerate}
\end{propos}

The general strategy in proving Proposition~\ref{prop:dim-1-e-closed} is similar to the strategy we used to prove Proposition~\ref{prop:dim-1-e}. 
However, the argument simplifies significantly thanks to the fact that $Y$ is equipped with an $H$-invariant probability measure. In particular, we do not require Proposition~\ref{prop:closing-lemma}, hence $\Gamma$ is {\em not} assumed to be an arithmetic lattice in this section, see Proposition~\ref{prop:Marg-Fun-Y-closed}.

For every $0<\delta\leq 1$ and every $y\in Y$, put
\[
I(y,\delta)=\Bigl\{w\in \mathfrak r: 0<\|w\|<\delta\inj(y)\text{ and }\exp(w)y\in Y\Bigr\}, 
\]
see also~\eqref{eq:def-I-h-y-C}. We will write $I(y)=I(y,\delta_0)$
where 
\be\label{eq: def delta0 closed orbit}
\delta_0=\nuni^{-3-\ref{E:non-div-main}}\min\{\inj(x): x\in X_{\rm cpt}\},
\ee
see~\eqref{eq:def cX closed}; recall also that $\inj(x)\leq 1$ 
for all $x\in X$.  

We need the following lemma.

\begin{lemma}\label{lem:numb-I-closed}
There exists $\constE\label{E:numb-I-cl}\ll_X1$ so that 
\[
\#I(y)\leq \ref{E:numb-I-cl}\vY
\]
for every $y\in Y$. 
\end{lemma}

\begin{proof}
This is proved for $G=\SL_2(\bbc)$ in~\cite[Lemma 8.13]{MO-MargFun}, see also~\cite[\S8]{EMM-Orbit}.

The same argument applies in the case of $G=\SL_2(\bbr)\times\SL_2(\bbr)$ 
if we replace~\cite[Lemma 8.4]{MO-MargFun} by Proposition~\ref{prop:non-div}. We sketch the proof for the sake of completeness. 

By virtue of Lemma \ref{lem:noI-tri-bd}, for all $y\in X_{\rm cpt}$, we have 
\[
\#I(y,1)\ll_X \vY.
\]

Suppose now that $y\in Y\setminus X_{\rm cpt}$, and  
let $t=|\log\inj(y)|+\ref{E:non-div-main}$. By Proposition \ref{prop:non-div}, there exists $|r|\leq 1$ so that
$a_tu_ry \in X_{\rm cpt}$. Moreover, for all $\|w\|< \delta_0\inj(y)$, see~\eqref{eq: def delta0 closed orbit}, we have 
\[
\|a_tu_rw\| \leq 3 e^t \|w\| = 3\nuni^{\ref{E:non-div-main}}\inj(y)^{-1} \|w\|< 0.5\inj(a_tu_ry).
\]
This and the fact that $Y$ is invariant under $H$ imply that
if $w\in I(y)=I(y,\delta_0)$, then $a_tu_rw \in  I(a_tu_ry, 1)$.

The above estimate also implies that the map 
$w\mapsto a_tu_rw$ 
is an injective map from $I(y)$ into $I(a_tu_ry, 1)$.
Consequently, 
\[
\#I(y)\le \#I(a_tu_ry, 1) \ll_X \vY.
\] 
The proof is complete.
\end{proof}

Let $0<\alpha<1$, and define a Margulis function $f_Y:Y\to [2,\infty)$ by
\[
f_Y(y)=\begin{cases} \sum_{w\in I(y)}\|w\|^{-\alpha} & \text{if $I(y)\neq\emptyset $}\\
\inj(y)^{-\alpha}&\text{otherwise}
\end{cases}.
\]

Let $m_\alpha$ be as in~\eqref{eq:EMM-use''}. Define the probability measure $\rwm$ on $H$ by the property that for every $\varphi\in C_c(X)$
\[
\rwm*\varphi(y)=\ave\varphi(a_{m_\alpha}\uvk y)\uvkd.
\]

The following proposition may be thought of as our replacement for Proposition~\ref{prop:closing-lemma}. 

\begin{propos}\label{prop:Marg-Fun-Y-closed}
There exists $\constE\label{C:margf-closed}\ll_{X,\alpha} 1$ so that 
\[
\int f_Y(y)\diff\!\mu_Y(y)\leq \ref{C:margf-closed}\cdot\vY.
\]
\end{propos}

The following lemma is analogue of Lemma~\ref{lem:Margulis-inequality}, and will be used in the proof of Proposition~\ref{prop:Marg-Fun-Y-closed}.   

\begin{lemma}\label{lem:Margulis-inequality-closed} 
There exists $\constE\label{E:margb-lemma-cl}\ll_{X,\alpha}1$ so that for all $\ell\in\bbn$ and all $y\in Y$, we have 
\be\label{eq:convLf_Y}
\convL* f_Y(y)\leq \nuni^{-\ell} f_Y(y)+\ref{E:margb-lemma-cl}\vY \textstyle\sum_{j}^\ell\nuni^{j-\ell}\rwm^{(j)}*\inj(y)^{-\alpha}.
\ee
\end{lemma}

\begin{proof}
Note that $\supp (\rwm)\subset\{h\in H: \|h\|\leq \nuni^{2m_\alpha+1}\}$. Let $C\geq 1$ be so that 
\[
\|\Ad(h)w\|\leq C\|w\|
\] 
for all $h$ with $\|h\|\leq \nuni^{2m_\alpha+1}$ and all $w\in\mathfrak g$. Increasing $C$ if necessary, we also assume that
$\inj(z)/C\leq \inj(hz)\leq C\inj(z)$ for all such $h$ and all $z\in X$.
Arguing as in the proof of Lemma~\ref{lem:Margulis-inequality}, there exists some $C$ so that 
\[
\rwm*\mfht_Y(y)\leq \nuni^{-1}\cdot\mfht_Y(y)+C\cdot\rwm*\psi(y)
\]
for all $y\in Y$, where $\psi(y)=\max\{1,\#I(y)\}\cdot\inj(y)^{-\alpha}$. 
This and Lemma~\ref{lem:numb-I-closed} imply that
\be\label{eq:average-f-h-cl}
\rwm*\mfht_Y(y)\leq \nuni^{-1}\cdot\mfht_Y(y)+\ref{E:margb-lemma-cl}\vY \cdot \Bigl(\rwm*\inj(y)^{-\alpha}\Bigr)
\ee 
with $\ref{E:margb-lemma-cl}=C\ref{E:numb-I-cl}$. Iterating~\eqref{eq:average-f-h-cl}, we get \eqref{eq:convLf_Y}. 
\end{proof}

\begin{proof}[Proof of Proposition~\ref{prop:Marg-Fun-Y-closed}]
The fact that estimates similar to Lemma~\ref{lem:Margulis-inequality-closed} imply integrability is by now a standard fact, see e.g.~\cite[\S5]{EMM-Up} or~\cite[Lemma 11.1]{EMM-Orbit}; we recall the argument. 
In view of Proposition~\ref{prop:inj-alpha-int}, we have
\[
\int_H\inj(h x)^{-\alpha}\diff\!\convN(h)\leq \nuni^{-n}\inj^{-\alpha}(x)+B
\]
for all $n\in \bbn$ where $B\ll_X1$. This and Lemma~\ref{lem:Margulis-inequality-closed} imply that 
\be\label{eq:rwmk-B+1 closed}
\limsup \rwm^{(n)}*f_Y(y)\leq 1+2\ref{E:margb-lemma-cl}\vY B.
\ee

Note that $\supp(\convN)\subset \{a_{m_\alpha n}u_r: |r|\leq 4\}$. 
This, together with the fact that $(H,\mu_Y)$ is mixing, implies that $\mu_Y$ is $\rwm$-ergodic. 
Thus by Chacon-Ornstein theorem, for every $\varphi\in L^1(Y,\mu_Y)$ and $\mu_Y$-a.e.\ $y\in Y$, we have $\frac{1}{N+1}\sum_{n=0}^N \convN*\varphi(y)\to \int \varphi\diff\!\mu_Y$. 

For every $k\in\bbn$, put $\varphi_k=\min\{f_Y,k\}$.
There exists a full measure set $Y_0$ so that for every $y\in Y_0$ and every $k$, there exists some $N_{k,y}$
so that if $N\geq N_{k,y}$, then $\frac{1}{N+1}\sum_{n=0}^N \convN*\varphi_k(y)\geq 0.5 \int \varphi_k\diff\!\mu_Y$.

%Using Egorov's theorem, there exists $Y_0\subset Y$ with $\mu_Y(Y_0)\geq 0.9$ and for every $k$ there is some $N_k$ so that the following holds. For all $k$, all $y\in Y_0$, and all $N\geq N_k$, 
%we have $\frac{1}{N+1}\sum_{n=0}^N \convN*\varphi_k(y)\geq 0.5 \int \varphi_k\diff\!\mu_Y$.

Let $y\in Y_0$, then the above estimate and~\eqref{eq:rwmk-B+1 closed}, applied with $y$, imply that $\int\varphi_k\diff\!\mu_Y\leq 2(1+2\ref{E:margb-lemma-cl}\vY B)$ for all $k$. 
Using Lebesgue's monotone convergence theorem, we conclude that
\[
\int f_Y\diff\!\mu_Y\leq 2(1+2\ref{E:margb-lemma-cl}\vY B).
\]
The claim follows as $\vY\gg_X1$.
\end{proof}

\begin{proof}[Proof of Proposition~\ref{prop:dim-1-e-closed}]
Put $\eta=0.1\eta_X$ where $\eta_X$ is as in Proposition~\ref{prop:non-div}. Recall from Lemma~\ref{lem:non-div-closed} that 
\be\label{eq:non-div-closed}
\mu_Y(X_{2\eta})\geq 0.9.
\ee

As was done in Lemma~\ref{lem:EG-Cheby-C}, we will first convert the information in Proposition~\ref{prop:Marg-Fun-Y-closed} into a pointwise estimate at most points. Let
\be\label{eq: def Y'' closed}
Y''=\{y\in Y: f_Y(y)\leq 100\ref{C:margf-closed}\vY\}.
\ee
Then by Proposition~\ref{prop:Marg-Fun-Y-closed}, we have $\mu_Y(Y\setminus Y'')\leq 0.01$. 

Let $Y'=Y''\cap X_{2\eta}$, and let $\beta=\eta^2=0.01\eta^2_X$. 
The above and~\eqref{eq:non-div-closed} imply that $\mu_Y(Y')\geq 0.9$.
Let $\{\boxG_{\beta^2}.z_j: z_j\in X_{2\eta}, j\in\mathcal J\}$ be a covering of $X_{2\eta}$ so that $\#\mathcal J\ll_X1$. Then there exists some $c\gg_X1$ and some $j_0$ so that 
\be\label{eq:density-pt-closed}
\mu_Y(\boxG_{\beta^2}.z_{j_0}\cap Y')\geq c.
\ee

Recall that $Y$ is $H$-invariant and  
$gz_j\in X_{\rm cpt}$ for all $j$ and $\|g-I\|\leq2$, see
\S\ref{sec:SiegelSet} where $X_{\rm cpt}$ is defined.
Let $y_0\in \boxG_{\beta^2}.z_{j_0}\cap Y'$. 
As was done in Lemma~\ref{lem:inductive-C}, let 
$F_1\subset B_\rfrak(0,2\beta^2)$ be so that  
\[
\boxG_{\beta^2}.z_{j_0}\cap Y'\subset \bigcup_{w\in F_1}\boxH_{\beta}.\exp(w)y_0.
\]
Then $\#F_1\geq c\eta^{-3} \vY$. Put 
\[
\cone_1=\coneH\cdot\{\exp(w)y_0: w\in F_1\}\subset Y\cap X_{\rm cpt};
\] 
recall that $\coneH=\boxHs_\beta\cdot\Big\{u_r: |r|\leq 0.1\eta\Big\}$.

Recall the definition $\mfht_{\cone_1}$ from~\eqref{eq:def-mfht-C}. There exists $C'\ll_{X,\alpha}1$ so that
\be\label{eq:Marg func ptw closed}
\mfht_{\cone_1}(e,z)\leq f_Y(z)\leq C'\vY\quad\text{for all $z\in\cone_1$}
\ee
To see this, note that by the definition of $f_Y$, for every $h\in H$ with $\|h-I\|\leq 1$ and all $y\in X_{\eta}\cap Y$, we have $f_Y(hy)\leq f_Y(y)+\bar C\vY$ where $\bar C\ll_X1$. Now for every $z\in\cone_1$, there exists $y\in Y'\subset Y''$ and some $h\in H$ with $\|h-I\|\leq 10\eta^2$ so that $z=hy$. This implies the claim in view of the definition of $Y''$ in~\eqref{eq: def Y'' closed}. Alternatively,~\eqref{eq:Marg func ptw closed} can be seen by letting $\ell=0$ in the proof of the sublemma in Lemma~\ref{lem:inductive-C}, see in particular~\eqref{eq:f-cone1-sum}.

Now~\eqref{eq:Marg func ptw closed} and Lemma~\ref{lem:MargFun-enetrgy-rfrak} imply that 
\[
\sum_{w'\neq w}\|w-w'\|^{-\alpha}\leq C\vY
\]
where the summation is over $w'\in F_1$ and $C\ll_{X,\alpha}1$.

The proposition holds with $y_0$ and $F=F_1$.  
\end{proof}

\subsection{Proof of Theorem~\ref{thm:main-closed}}\label{sec:proof-main-closed}
The proof goes along the same lines as the proof of Theorem~\ref{thm:main} if we replace Proposition~\ref{prop:dim-1-e} with Proposition~\ref{prop:dim-1-e-closed} as we now explicate. 

Let $\vare=0.0005\vare_0$ and $\alpha=1-\vare$ where $\vare_0$ is given by Proposition~\ref{prop:1-epsilon-N}. 
By Proposition~\ref{prop:dim-1-e-closed}, the conditions in Proposition~\ref{prop:proj-general}
holds with $y_0\in Y\cap X_{\rm cpt}$, $F$, $\alpha$, and $\eta=0.1\eta_X$. 

Recall that $\#F\gg_{X} \vY$. We assume $\vY$ is large enough so that 
\[
(\#F)^{-\vare}\leq (2\ref{E:BCH}\ref{E:non-div-main})^{-1}\eta^3.
\]
Then by Proposition~\ref{prop:proj-general}, there exist $y_1\in X_\eta$, a finite subset $I\subset [0,1]$, and some $\rhsco>0$ with
\be\label{eq:size of b1 closed}
\vY^{-\frac{2+6\vare}{2+21\vare}}\ll_{X}(\#F)^{-\frac{2+6\vare}{2+21\vare}}\leq \rhsco\leq (\#F)^{-\vare}\ll_{X} \vY^{-\vare},
\ee 
so that both of the following two statements hold true:
\begin{enumerate} 
\item The set $I$ supports a probability measure $\rho$ which satisfies  
\[
\rho(J)\leq C'_\vare\cdot |J|^{\alpha-30\vare}
\]
for all intervals $J$ with $|J|\geq (\#F)^{\frac{-15\vare}{2+21\vare}}$, where $C'_\vare\ll\vare^{-\star}$ for absolute implied constants. 
\item There is an absolute constant $C\ll_X1$, so that for all $s\in I$, we have 
\begin{align*}
v_sy_1&\in \boxG_{C\rhsco} \cdot \Bigl(a_{|\log\rhsco|}\cdot \{u_r: |r|\leq 2\}\Big).\{\exp(w)y_0: w\in F\}\\
&\subset \boxG_{C\rhsco}. Y,
\end{align*}
\end{enumerate}
\medskip

\noindent 
For the last inclusion in (2) we used \ref{item:dim-1-e-closed-a} and the $H$-invariance of $Y$. 

\medskip 
In particular, part~(2) and $\rhsco\leq (\#F)^{-\vare}$ imply that   
\be\label{eq:local V orbit close to Y}
\dist_X\Bigl(v(s)y_1, Y\Bigr)\leq C'\vY^{-\vare}\qquad\text{for all $s\in I$},
\ee
where $C'\ll_{X,\alpha}1$.

The proof of Theorem~\ref{thm:main-closed} is now completed as the proof of Theorem~\ref{thm:main}
if we replace Proposition~\ref{prop:main} with part~(1) above and~\eqref{eq:local V orbit close to Y}, see \S\ref{sec:proof-main-1}.

We note that 
\be\label{eq: def C3 and k3}
\ref{c:periodic}\ll_{X,\alpha}1\quad\text{and}\quad\ref{k:periodic}=c\ref{k:mixing}\vare
\ee
where the notation $\ll_{X,\alpha}$ is defined in~\eqref{eq:a ll X alpha b},
$c$ is an absolute constant, and $\ref{k:mixing}$ is as in Proposition~\ref{prop:1-epsilon-N}; we also used the fact that $\ref{C:1-epsilon-N}\ll_X1$, see Proposition~\ref{prop:1-epsilon-N}.

Note that $\mixexp$ in~\eqref{eq:actual-mixing}, and hence $\ref{k:periodic}$, is absolute if $\Gamma$ is congruence. 
\qed

\subsection{Proof of Theorem~\ref{thm:finiteness-intro}}
Let $\Gamma\subset \SL_2(\bbc)$ be as in the statement.
As was mentioned prior to Theorem~\ref{thm:finiteness-intro}, 
a totally geodesic plane in $M$ lifts to a periodic orbit of $H=\SL_2(\bbr)$ in $X=G/\Gamma$. 

Recall from \S\ref{sec:SiegelSet} that $X\setminus X_{\eta_X}$ is a disjoint union of finitely many cusps. Let $M_0\subset M$ denote the image of $X_{\eta_X}$ in $M$. Then $M\setminus M_0$ is a disjoint union of finitely many (possibly none) cusps.

Let $\eta_1>0$ be so that for $i=1,2$ there exists $x_i\in X_{\eta_0}$ such that $\boxG_{\eta_1}.x_i$ projects into the interior of $N_i\cap M_0$. In view of~\cite[Thm.~1.5]{MO-MargFun}, applied with $s=1/2$, we have $\eta_1\gg_X{\rm area}(\Sigma)^{-4}$ where 
$\Sigma=\partial N_1=\partial N_2$.

Thus, Theorem~\ref{thm:main-closed} implies that 
if $Hx$ is a periodic orbit which satisfies  
\be\label{eq: orbits in half manifold}
\ref{c:periodic}\vol(Hx)^{-\ref{k:periodic}}\leq 0.5\min\{\eta_1,\eta_X\},
\ee
then $Hx\cap\boxG_{\eta_1}.x_i\neq \emptyset$, for $i=1,2$. Therefore, the corresponding plane crosses $\Sigma$.

Let us now assume that $S$ is a plane which crosses $\Sigma$. 
By~\cite[Thm.~4.1]{FLMS}, see also~\cite[Prop.~12.1]{BO-GF},
$S$ intersects $\Sigma$ orthogonally. 
It is shown in~\cite[Prop~5.1]{FLMS} that one can construct an explicit open set $O$ of the unit tangent bundle of $M$ which projects into the 1-neighborhood of $M_0$ and does not intersect such an $S$ --- indeed this set is constructed using a tubular neighborhood of $\Sigma \cap M_0$. 

Let $\eta_2$ and $x\in X$ be so that 
$\boxG_{\eta_2}.x$ projects into $O$. In view of~\cite[Thm.~1.5]{MO-MargFun}, applied with $s=1/2$, and the construction in~\cite[Prop~5.1]{FLMS}, we have $\eta_2\gg_X{\rm area}(\Sigma)^{-4}$. 

Note that $Hx\cap \boxG_{\eta_2}.x=\emptyset$. However, by Theorem~\ref{thm:main-closed} again, if  
\[
\ref{c:periodic}\vol(Y)^{-\ref{k:periodic}}\leq 0.5\eta_2,
\] 
then $Hx\cap \boxG_{\eta_2}.x\neq \emptyset$. 

This and~\eqref{eq: orbits in half manifold} thus imply that 
\[
\vol(Hx)\leq \Bigl(\tfrac{2\ref{c:periodic}}{\min\{\eta_X,\eta_1,\eta_2\}}\Bigr)^{1/\ref{k:periodic}}\ll_X {\rm area}(\Sigma)^{4/\ref{k:periodic}}\ref{c:periodic}^{1/\ref{k:periodic}}.
\]
Moreover, in view of~\cite[Cor.~10.7]{MO-MargFun}, the number of periodic $H$-orbits with $\vol(Hx)\leq T$ is $\ll_X T^6$.

When $G=\SL_2(\bbc)$ (which is the case here), $\ref{E:non-div-main}\ll|\log\eta_X|$ for an absolute implied constant; see the proof of Proposition~\ref{prop:one-return}.
Moreover, in view of  Lemma~\ref{lem:EMM-(2,1)} and the fact that $\alpha=1-0.0005\vare_0$, we have $\nuni^{m_\alpha}\ll \mixexp^{-\star}$ for absolute implied constants (see Proposition~\ref{prop:1-epsilon-N}).

The proof is thus complete in view of the above,~\eqref{eq: def C3 and k3}, and~\eqref{eq:a ll X alpha b}. 
\qed

\appendix
\section{Proof of Proposition~\ref{lem:one-return}, Case 2}\label{sec:app-non-div}

In this section we complete the proof of Proposition~\ref{lem:one-return}. Recall that we are left with the case where $G=\SL_2(\bbr)\times \SL_2(\bbr)$ and $\Gamma$ is irreducible.

By a theorem of Selberg~\cite{Selberge-Arith}, we have the following: 
up to automorphisms of $G$, irreducible non-uniform lattices in $\SL_2(\bbr)\times\SL_2(\bbr)$ are 
commensurable to $\SL_2(\mathcal O)$ where $\mathcal O$ is the ring of integers in a totally real quadratic extension $L/\bbq$.

Passing to a finite index subgroup, we may assume that 
$\Gamma\subset\SL_2(\mathcal O)$. Since the statement of Proposition~\ref{lem:one-return} is insensitive to passing to a finite index subgroup 
we may (and will) assume $\Gamma=\SL_2(\mathcal O)$. 
By fixing a $\bbz$-basis for $\mathcal O$ one can now identify  
\[
G={\bf G}(\bbr)\quad\text{and}\quad\Gamma={\bf G}(\bbz).
\]
where ${\bf G}={\rm Res}_{L/\bbq}(\SL_2)$, the restriction of scalars from $L$ to $\bbq$. This choice of $\bbz$ basis induces a canonical identification between ${\bf G}(\bbq)$ and $\SL_2(L)$
and in the sequel we shall implicitly identify these two groups.

Let ${\bf B}\subset \SL_2$ denote the group of upper triangular matrices in $\SL_2$ and put ${\bf P}={\rm Res}_{L/\bbq}({\bf B})$. 
Then ${\bf P}$ is a minimal and maximal $\bbq$-parabolic subgroup of $\bf G$.
By a theorem of Borel and Harish-Chandra, he action of $\Gamma$ on ${\bf P}(\bbq)\backslash\G(\bbq)$ has finitely many orbits; let $\Xi\subset \G(\bbq)$ %=\SL_2(L) 
 be a finite subset which contains exactly one representative for each orbit (we always assume $\Xi$ contains the identity element). Then 
\be\label{eq:cusps-rational}
%\SL_2(L)\simeq
\G(\bbq)={\bf P}(\bbq)\Xi\Gamma,
\ee
and if $\gamma\xi_1 {\bf P}(\bbq)\xi_1^{-1}\gamma^{-1}=\xi_2{\bf P}(\bbq)\xi_2^{-1}$ where 
$\gamma\in\Gamma$ and $\xi_i^{-1}\in \Xi$, then $\xi_1=\xi_2$.

In the case at hand, $\gfrak=\Lie(G)=\mathfrak{sl}_2(\bbr)\oplus\mathfrak{sl}_2(\bbr)$,
moreover, $\gfrak$ is equipped with the $\bbq$-structure: 
\[
\gfrak_\bbq=\sl_2(L)\subset \gfrak.
\]
We will also write $\gfrak_\bbz$ for $\sl_2(\mathcal O)$; then $\gfrak_\bbz$ is a lattice in $\gfrak$.

Note that $\mathcal O^\times \gfrak_\bbz=\gfrak_\bbz$.
Recall the following elementary fact: there exists some $c=c_L$ so that the following holds.
For every $w=(w_1,w_2)\in\gfrak$ with $\|w_1\|\|w_2\|\neq0$, there exists some 
$\mathsf s\in \mathcal O^\times$ so that  
\be\label{eq:units-norm}
c^{-1}\Big(\|w_1\|\|w_2\|\Big)^{1/2}\leq \|p_i(\mathsf s w)\|\leq c\Big(\|w_1\|\|w_2\|\Big)^{1/2},
\ee
for $i=1,2$, where $p_i$ denotes the projection onto the $i$-th components, see e.g.~\cite[Lemma 8.6]{Kleinbock-Tomanov}.  

Let $N=R_u({\bf P}(\bbr))$, i.e.\ $N$ is the unipotent radical of ${\bf P}(\bbr)$.
We fix a basis $\{v_1,v_2\}$ for $\Lie(N)$ consisting of primitive integral vectors as follows. 
Write $L=\bbq[\sqrt{\beta}]$; put $v_1=\Bigl(E_{12}, E_{12}\Bigr)$ and $v_2=\Bigl(\sqrt{\beta}E_{12}, -\sqrt\beta E_{12}\Bigr)$ where $E_{12}$ denotes the elementary matrix with $1$ at the $(1,2)$-entry, and  
define
\[
v:=v_1\wedge v_2\in \wedge^2\gfrak.
\] 

Since $v\in \wedge^2\gfrak_\bbz$, for any $g\in \G(\bbq)$, we have $\Gamma g.v$ is contained in the set of 
rational vectors in $\wedge^2\gfrak$ whose denominators (with respect to the $\bbz$-structure given by $\gfrak_\bbz$) are bounded in terms of $g$. 
In particular, $\Gamma g.v$ is a discrete and closed subset of $\wedge^2\gfrak$.

Note that for any $g=(g_1,g_2)\in G$, we have 
\be\label{eq:gv-E12}
\begin{aligned}
gv&=(gv_1)\wedge (gv_2)\\
&=-2\sqrt{\beta}\Big(g_1E_{12},0\Big)\wedge \Big(0,g_2E_{12}\Big).
\end{aligned}
\ee

Define $\omega:G/\Gamma\to [2,\infty)$ as follows: 
\be\label{eq:def-omega-irred}
\omega(g\Gamma)=\max\biggl\{2, \max\Bigl\{\|g\gamma\xi .v\|^{-1}:\xi\in\Xi^{-1}, \gamma\in \Gamma\Bigr\}\biggr\}.
\ee

We have the following analogue of Lemmas~\ref{lem:one-return-1} and~\ref{lem:one-return-1-1}. 
In the case at hand, this result is a consequence of the fact that the $\bbq$-rank of $\bf G$ is 1 --- recall that $\bf P$ is a minimal and maximal $\bbq$-parabolic subgroup of $\bf G$. 

\begin{lemma}\label{lem:one-return-2}
Let the notation be as above. 
\begin{enumerate}
\item There exists $C=C(\Gamma)\geq2$ so that the following holds. Let $g\Gamma\in X$. 
If $\omega(g\Gamma)\geq C$, then
there is $\xi_0\in\Xi^{-1}$ and $\gamma_0\in\Gamma$ so that 
$\|g\gamma_0\xi_0 .v\|^{-1} =\omega(g\Gamma)$ and 
\[
\|g\gamma\xi .v\|> 1/C,\quad\text{ for all $\;(\xi,\gamma)\;$ so that $\;\gamma\xi.v\neq \gamma_0\xi_0.v$}.
\]
\item 
There exists $\constE\label{E:C-alpha-2}$ so that the following holds. 
Let $0<\varrho,\eta<1$, $t>0$, and $g\in G$. Let $I\subset \bbr$ be an interval of length at least $\eta$. Then 
\[
\Big|\Big\{r\in I:\|a_t\uvk g.v\|\leq \nuni^{2t}\eta^4\varrho^4\|gv\|\Big\}\Big|\leq \ref{E:C-alpha-2}\varrho|I|.
\]
\end{enumerate}
\end{lemma}

\begin{proof}
As we mentioned above, there is some $M\in\bbn$ so that $\Gamma\Xi^{-1}.v_i\subset \frac{1}{M}\gfrak_\bbz$.

Let $0<\delta<1$ be a small number which will be explicated later. 
Suppose there are $\gamma\xi.v\neq\gamma'\xi'.v$ so that
\be\label{eq: two small vectors app}
\|g\gamma\xi.v\|<\delta\quad\text{and}\quad\|g\gamma'\xi'.v\|<\delta.
\ee
We first show that $\gamma\xi.v\not\in \bbr.\gamma'\xi' v$.
Assume contrary to this claim that $\gamma\xi.v=\lambda\gamma'\xi' v$ for some $\lambda\in\bbr$.  
Then since ${\bf P}(\bbr)$ is the projective stabilizer of $v$, we conclude that 
\[
\gamma\xi{\bf P}(\bbr)\xi^{-1}\gamma^{-1}=\gamma'\xi'{\bf P}(\bbr)\xi'^{-1}\gamma'^{-1}.
\] 
This in view of the choice of $\Xi$, see the discussion following~\eqref{eq:cusps-rational}, implies that $\xi=\xi'$. 
Thus, since ${\bf P}(\bbr)$ is its own normalizer in $\mathbf G (\bbr)$,\  $\gamma^{-1}\gamma'\in \xi{\bf P}(\bbr)\xi^{-1}$. We conclude that $\lambda=N_{L/\bbq}(\sfs^2)$ for  
a unit in $\sfs\in\mathcal O^\times$ (recall that $\G=R_{L/\bbq}(\SL_2)$).  Hence, $\lambda=1$ which contradicts our assumption. 

Recall that $v=v_1\wedge v_2$ where $v_1=\Bigl(E_{12}, E_{12}\Bigr)$ and $v_2=\Bigl(\sqrt{\beta}E_{12}, -\sqrt\beta E_{12}\Bigr)$.
Since $\gamma\xi.v\not\in \bbr.\gamma'\xi' v$ the subspace generated by the four vectors $w_i=g\gamma\xi.v_i$  $w'_i=g\gamma'\xi'.v_i$, for $i=1,2$ 
has dimension $\geq 3$. We claim this subspace also generates a nilpotent subalgebra of $\gfrak$. 
This contradicts the fact that the dimension of any maximal nilpotent subalgebra in $\gfrak$ is $2$ and finishes the proof of part~(1).

To see the claim, note that~\eqref{eq: two small vectors app} and the identity in~\eqref{eq:gv-E12} imply 
\[
\|p_1(w_1)\|\cdot\|p_2(w_2)\|\leq \delta/2,
\]
similarly for $w'_1$ and $w'_2$. 
In view of the definition of $v_i$ (and $w_i$), therefore, $\|p_1(w_i)\|\cdot\|p_2(w_i)\|\ll_\beta \delta$ for $i=1,2$.
Similarly, we have $w'_1$ and $w'_2$. 

We now apply~\eqref{eq:units-norm} to the four vectors $w_1, w_2, w'_1, w'_2$. 
In consequence, there are $\sfs_i,\sfs'_i\in\mathcal O^{\times}$ so that $\|\sfs_iw_i\|\ll_\beta\delta^{1/2}$ and 
$\|\sfs_i'w_i'\|\ll_\beta\delta^{1/2}$ for $i=1,2$. 

Moreover, $\{\sfs_1w_1,\sfs_2w_2,\sfs_1'w_1', \sfs_2'w_2'\}$ are nilpotent elements in
$\frac{1}{M}\Ad(g)\gfrak_\bbz$. Since $\Bigl\|[w,w']\Bigr\|\leq \|w\|\|w'\|$, 
we get from the discreteness of $\Ad(g)\gfrak_\bbz$ that if $\delta$ is small enough, 
then $\{\sfs_1w_1,\sfs_2w_2,\sfs_1'w_1',\sfs_2'w_2'\}$ generates a nilpotent Lie algebra as we claimed.

\medskip

The argument for part~(2) is similar to the proof of 
Lemma~\ref{lem:one-return-1-1} as we now explain.
For every $g\in G$ and every $\delta>0$, put
\[
I(g,\delta)=\Bigl\{r\in I: \|p_i^+(u_rg.v_i)\|\leq 0.01\delta \eta^2\|p_i(g.v_i)\|\;\text{for $i=1$ or $i=2$}\Bigr\}
\]
where $p_1^+$ denotes the projection from $\gfrak$ onto $\bbr(E_{12},0)$ and $p_2^+$ denotes the projection 
from $\gfrak$ onto $\bbr(0,E_{12})$; recall also that $p_i$ denotes projection onto the $i$-th component. 
As it was observed in Lemma~\ref{lem:one-return-1-1}, we have 
\[
|I(g,\delta)|\leq 2C'\delta^{1/2}|I|.
\]

Let $\delta=100\varrho^{2}$, and let $r\in I\setminus I(g,\delta)$. Then
\be\label{eq:exp-p1-p2}
\|p_i^+(\uvk g.v_i)\|\geq \eta^2\|p_i(g.v_i)\|\varrho^2\quad\text{for $i=1,2$}.
\ee

Using~\eqref{eq:gv-E12}, we have $\|g.v\|=2\|p_1(g.v_1)\|\cdot\|p_2(g.v_2)\|$. 
Since $a_t.w=\nuni^{t}w$ for any $w\in{\rm span}\Bigl\{(E_{12},0),(0,E_{12})\Bigr\}$, using~\eqref{eq:gv-E12} and~\eqref{eq:exp-p1-p2}, 
we conclude that
\begin{align*}
\nuni^{2t}\eta^4 \|g.v\|\varrho^4&=2\nuni^{2t}\eta^4\|p_1(g.v_1)\|\cdot\|p_2(g.v_2)\|\varrho^4\\
&\leq 2\nuni^{2t}\|p_1^+(\uvk g.v_1)\|\cdot\|p_2^+(\uvk g.v_2)\|\\
&= \Bigl\|a_t\Bigl((p_1^+(\uvk g.v_1),0)\wedge (0,p_2^+(\uvk g.v_2))\Bigr)\Bigr\|\leq \| a_t\uvk g.v\|.
\end{align*}

The claim thus holds with $\ref{E:C-alpha-2}=20C'$.
\end{proof}

\begin{lemma}\label{lem:inj-rad-irred}
Let the notation be as above. There exists $\constE\label{E:inj-irred}$ so that 
\[
\ref{E:inj-irred}^{-1}\omega(x)^{-1}\leq\inj(x)^2 \leq \ref{E:inj-irred} \omega(x)^{-1}
\]
for all $x\in X$.
\end{lemma}

\begin{proof}
Let $g\in G$ and assume that $\inj(g\Gamma)<\delta$. Then 
\[
g\Gamma g^{-1}\cap \boxG_{C\delta}\neq\{e\}
\]  
where $C$ is an absolute constant.
 
If $\delta$ is small enough, then $g\Gamma g^{-1}\cap \boxG_{C\delta}$ consists only of unipotent elements. 
Therefore, there exists some nilpotent element $w\in\gfrak_\bbz$ so that 
\[
\|gw\|\ll \delta
\]
where the implied constant is absolute.

Since all minimal $\bbq$-parabolic subgroups of $\G$ are conjugate to each other by elements in
$\G(\bbq)$, it follows from~\eqref{eq:cusps-rational} that there exists some 
$\gamma\in\Gamma$ and some $\xi\in\Xi$ so that $w\in \gamma^{-1} \xi^{-1}.\Lie(N)$. Therefore, we may write 
\[
w= \gamma^{-1}\xi^{-1}.\Bigl((b+c\sqrt\beta) E_{12}, (b-c\sqrt\beta) E_{12}\Bigr)
\]
where $b,c\in\frac{1}{M}\bbz$ for some $M$ depending on $\Xi$.

Using the Iwasawa decomposition, we write 
$g\gamma^{-1}\xi^{-1}= kan$ where 
$k\in\SO(2)\times \SO(2)$, $n\in N$, and $a=(a_{t_1},a_{t_2})$ is diagonal. Therefore,
\[
\nuni^{t_1+t_2}(b^2+c^2\beta)\ll \delta^2
\]
where the implied constant is absolute. 

Now since $b,c\in\frac{1}{M}\bbz$ are non-zero, we have $b^2+c^2\beta\gg_M1$. 
Altogether, we conclude that 
\begin{align*}
\|g\gamma^{-1}\xi^{-1} .v\|&=2\|p_1(a_{t_1}.v_1)\|\|p_2(a_{t_2}.v_2)\|\\
&\leq 2\sqrt{\beta}\nuni^{t_1+t_2}\leq 2\sqrt{\beta}
\hat C\delta^2
\end{align*} 
where $\hat C$ depends on $\Gamma$.
Since $\omega(g\Gamma)^{-1}\leq \|g\gamma^{-1}\xi^{-1} .v\|$, the lower bound in the lemma follows.

We now turn to the proof of the upper bound. 
Using the reduction theory for arithmetic groups, see e.g.~\cite[Ch.~4]{PlRa}, there exist $t_0,r_0>0$ so that  
\[
\Bigl(\SO(2)\times\SO(2)\Bigr)\cdot\Bigl\{(a_{t},a_{t'}\Bigr): t+t'\leq t_0\Bigr\}\cdot\{n(r,s); |r|, |s|\leq r_0\}\cdot \Xi
\]
is a (generalized) fundamental domain for $\Gamma$ in $G$.

In particular, using Lemma~\ref{lem:one-return-2}(1), 
there exists $t_1\leq t_0$ so that if $g=k(a_t,a_{t'})n(r,s)\xi_0\gamma_0$ for $t+t'\leq t_1$,
then 
\begin{align*}
\omega(g\Gamma)&=\max\Bigl\{\|g\gamma \xi^{-1} .v\|^{-1}: (\xi,\gamma)\in\Xi\times\Gamma\Bigr\}=\|g\gamma_0^{-1}\xi_0^{-1} .v\|^{-1}\\
&=\|k(a_t,a_{t'})n(r,s).v\|^{-1}=\nuni^{-t-t'}\|v\|^{-1}.
\end{align*}
 
Moreover, using ~\eqref{eq:gv-E12} and~\eqref{eq:units-norm}
we conclude that $g\gamma_0^{-1}\xi_0^{-1} (N\cap \Gamma)\xi_0\gamma_0g^{-1}$ contains elements 
of size $\nuni^{(-t-t')/2}$. The upper bound estimate follows.
\end{proof}

\begin{proof}[Proof of Proposition~\ref{lem:one-return}: Case 2]
By Lemma~\ref{lem:inj-rad-irred}, $t\geq |\log(\eta^2\inj(g\Gamma))|+\ref{E:non-div-main}$ implies 
$2t\geq \log(\omega(g\Gamma)/\eta^4)$ if we assume $\ref{E:non-div-main}$ is large enough. 

Let $\varrho_0=0.1\ref{E:C-alpha-2}^{-1}$. In view of Lemma~\ref{lem:one-return-2}(2) we have 
\[
\sup\Big\{\|a_tu_rg\gamma\xi^{-1}.v\|: r\in I\Big\}\geq \varrho_0^4\qquad\text{for all $\gamma\in \Gamma$ and $\xi\in\Xi$}
\]
so long as $2t\geq |\log(\omega(g\Gamma)/\eta^4)|$.

Altogether, the conditions in~\cite[Thm.~4.1]{KM-Nondiv} 
are satisfied so long as $t\geq |\log(\eta^2\inj(g\Gamma))|+\ref{E:non-div-main}$.
Hence, similar to the previous case, the conclusion of the proposition in this case also holds true in view of ~\cite[Thm.~4.1]{KM-Nondiv}
--- in light of Lemma~\ref{lem:one-return-2}(1), the proof simplifies significantly.
\end{proof}

We also record the following which is a special case of the results and techniques developed in~\cite{EMM-Upp} and~\cite{EM-RW}
tailored to our setup here.

\begin{propos}\label{prop:inj-alpha-int}
Let $0<\alpha<1$ and let $\osa$ be as in~\eqref{eq:EMM-use}. 
There exists some $B=B(X,\alpha)\geq 1$ satisfying the following.
For every $x\in X$ and every $n\in\bbn$ we have 
\[
\int_H\inj(h x)^{-\alpha}\diff\!\convN(h)\leq \nuni^{-n}\inj^{-\alpha}(x)+B
\] 
where $\rwm(\varphi)=\ave\varphi(a_{\osa}u_r)\diff\!r$ for every $\varphi\in C_c(H)$ and $\convN$ denotes the $n$-fold convolution of $\rwm$.
\end{propos}

\begin{proof}
If $X$ is compact, then $\inj: X\to \bbr$ is a bounded function and the result is clear.

Therefore, we may assume $X$ is not compact. 
If $G=\SL_2(\bbc)$, the claim in the proposition is proved in~\cite{MO-MargFun}.

We now consider $G=\SL_2(\bbr)\times\SL_2(\bbr)$ and consider separately the cases where $\Gamma$ is a reducible lattice 
and $\Gamma$ is irreducible.

\medskip

{\em Case 1.} Let use first assume that $\Gamma$ is reducible. 
As was done before, passing to a finite index subgroup, we may assume 
$\Gamma=\Gamma_1\times\Gamma_2$. 

Let $\omega$ be defined as in~\eqref{eq:def-omega}. That is: 
\[
\omega(x)=\max\{\omega_1(x_1),\omega_2(x_2)\}
\]
where $x=(x_1,x_2)$.

By~\cite[Prop.~6.7]{MO-MargFun} we have $\omega(x)\asymp{\rm inj}(x)^{-1}$.
Therefore, it suffices to prove the proposition with ${\rm inj}(x)$ replaced by $\omega(x)$.
The result for $\omega_1$ and $\omega_2$ is well-known, see e.g.~\cite{MO-MargFun, EM-RW, EMM-Upp}. 

The result for $\omega$ thus follows as $\omega^\alpha\leq \omega_1^\alpha+\omega_2^\alpha\leq 2\omega^\alpha$.

\medskip

{\em Case 2.} Assume now that $\Gamma$ is irreducible. 
We will use the notation which we fixed in the beginning of this appendix. 
In particular, as was done in~\eqref{eq:def-omega-irred}, define
\[
\omega(g\Gamma)=\max\biggl\{2, \max\Bigl\{\|g\gamma\xi .v\|^{-1}:\xi\in\Xi^{-1}, \gamma\in \Gamma\Bigr\}\biggr\}.
\]

In view of Lemma~\ref{lem:inj-rad-irred}, we have $\omega(x)\asymp\inj(x)^{-2}$ for all $x\in X$.
Therefore, it suffices to prove the claim for $\omega^{1/2}$ instead if $\inj$. 

Let us recall from~\eqref{eq:gv-E12} that 
\begin{align}\label{eq:gv-E12-again}
gv&=-2(p_1(g.v),0)\wedge (0,p_2(g.v))\\
\notag &=-2\sqrt{\beta}(g_1E_{12},0)\wedge (0,g_2E_{12})
\end{align}
for any $g=(g_1,g_2)$.

Let $x=g\Gamma$. Fix $\gamma\in\Gamma$ and $\xi\in\Xi^{-1}$; for all $r\in[0,1]$ and $\ell\in\bbn$ put $h_r=a_{\ell}\uvk\gamma\xi$.
In view of the Cauchy-Schwarz inequality and~\eqref{eq:gv-E12-again}, applied with $h_rg$, we have  
\begin{multline}\label{eq:wedge-prod}
\left(\int_{0}^1\|h_rgv\|^{-\alpha/2}\uvkd\right)^{2}\leq\\ 
2\sqrt{\beta}\int_{0}^1\|h_{r1}g_1E_{12}\|^{-\alpha}\uvkd \int_{0}^1\|h_{r2}g_2 E_{12}\|^{-\alpha}\uvkd. 
\end{multline}

Then for $i=1,2$, by choice of $\osa$, we have 
\[
\ave\|a_{\osa}\uvk g_i\gamma_i\xi_iE_{12}\|^{-\alpha}\uvkd< \nuni^{-1}\|g_i\gamma_i\xi_iE_{12}\|^{-\alpha},
\]
see~\eqref{eq:EMM-use}.

Using~\eqref{eq:gv-E12-again} in reverse order and~\eqref{eq:wedge-prod}, 
we conclude from the above two estimates that
\be\label{eq:omgea-alpha-irr}
\ave\|a_{\osa}\uvk g\gamma\xi v\|^{-\alpha/2}\uvkd \leq \nuni^{-1} \|g\gamma\xi v\|^{-\alpha/2}.
\ee

Let $C(\Gamma)$ be as in Lemma~\ref{lem:one-return-2}. Then there exists some $B'_{\osa}>0$ 
so that if $\omega(g\Gamma)=\|g\gamma\xi v\|^{-1}\geq C(\Gamma)\cdot B'_{\osa}$, 
then 
\[
\omega(a_{\osa}\uvk g\Gamma)=\|a_{\osa}\uvk g\gamma\xi v\|^{-1}\geq C(\Gamma)
\]
for all $r\in[0,1]$.

This and~\eqref{eq:omgea-alpha-irr} imply that for all $x\in X$, we have 
\[
\int\omega^{\alpha/2}(hx)\diff\!\rwm(h)=\ave\omega^{\alpha/2}(a_{\osa}\uvk x)\uvkd\leq \nuni^{-1}\omega^{\alpha/2}(x)+B''
\]
where $B''=\max\{\omega(a_{\osa}\uvk g\Gamma): r\in[0,1], \omega(g\Gamma)\leq C(\Gamma)\cdot B'_{\osa}\}$. 

Iterating this estimate and summing the geometric sum, we conclude that
\be\label{eq:rec-iteration-all-ell}
\int\omega^{\alpha/2}(hx)\diff\!\rwm^{(n)}(h)\leq \nuni^{-n}\omega^{\alpha/2}(x)+ B
\ee
for all $n\in\bbn$ where $ B=2B''$. The proof is complete.
\end{proof}

\section{Proof of Theorem~\ref{thm:proj-thm}}\label{sec:proof-proj}

Recall that $\rfrak\subset \Lie(G)$ is identified with $\sl_2(\bbr)$ equipped with the adjoint action of $\SL_2(\bbr)$.

\begin{thm}\label{thm:proj-thm-app}
Let $0<\alpha\leq1$, and let  $0<\rhsc_0<\rhsc_1<1$. 
Let $E\subset B_\rfrak(0,\rhsc_1)$ be a finite set, and let $\rho$ denote the uniform measure on $E$. Assume that
\be\label{eq:appendix-regularity-F}
\rho(B_\rfrak(w, \rhsc))\leq \Upsilon\cdot (\rhsc/\rhsc_1)^\alpha\qquad\text{for all $w$ and all $\rhsc\geq \rhsc_0$}
\ee
where $\Upsilon\geq 1$.

Let $0<\vare<0.01\alpha$, and let $J\subset [0,1]$ be an interval with $|J|\geq 10^{-6}$. For every $\rhsc\geq \rhsc_0$, 
there exists a subset $J_{\rhsc}\subset J$ with $|J\setminus J_{\rhsc}|\leq C_\vare(\rhsc/\rhsc_1)^{\vare}$ 
so that the following holds. 
Let $r\in J_{\rhsc}$, then there exists a subset $E_{\rhsc,r}\subset E$ with 
\[
\rho(E\setminus E_{\rhsc,r})\leq C_\vare(\rhsc/\rhsc_1)^{\vare}
\]
such that for all $w\in E_{b,r}$, we have 
\[
\rho\Bigl(\{w'\in E: |\xi_{r}(w')-\xi_r(w)|\leq \rhsc\}\Bigr)\leq C_\vare(\rhsc/\rhsc_1)^{\alpha-7\vare}
\] 
where $C_\vare\ll\vare^{-\star}\Upsilon^\star$ (implied constants are absolute) and 
\[
\xi_r(w)=(\Ad(u_r)w)_{12}=-w_{21}r^2-2w_{11}r+w_{12}.
\]
\end{thm}

We need some more notation for the proof. 
First note that the assumption and the conclusion in the theorem are invariant under scaling.
Thus replacing $E$ by $\rhsc_1^{-1}\cdot E$ and $\rhsc_0$ by $\rhsc_0/\rhsc_1$, we may assume $\rhsc_1=1$. 
We prove the theorem for $J=[0,1]$, the proof in general is similar. 

Let
\[
\Xi(w)=\Bigl\{(r,\xi_r(w)):r\in [0,1]\Bigr\}
\]
for every $w\in E$, and let $\Xi=\bigcup_{w} \Xi(w)$. 

For every $\rhsc>0$ and every $w\in E$, let 
\[
\Xi^{\rhsc}(w)=\Bigl\{(q_1,q_2)\in [0,1]\times \bbr: |q_2-\xi_{q_1}(w)|\leq\rhsc\Bigr\}.
\] 
Finally, for all $q\in\bbr^2$ and $\rhsc>0$, define
\be\label{eq:def-trho-delta}
m^\rhsc_{\rho}(q):=\rho\Bigl(\{w'\in\rfrak: q\in \Xi^\rhsc(w')\}\Bigr).
\ee

The assertion in the theorem may be rewritten in terms of the multiplicity function $m^{\rhsc}_{\rho}$ as follows.
We seek the set $J_b\subset [0,1]$, and for every $r\in J_b$, the set $E_{b,r}\subset E$ so that 
\be\label{eq:thm-mult-app}
m^{\rhsc}_{\rho}\Bigl((r,\xi_{r}(w))\Bigr)\leq C_\vare\rhsc^{\alpha-7\vare}\quad\text{for all $w\in E_{b,r}$.}
\ee

The following lemma plays a crucial role in the proof of Theorem~\ref{thm:proj-thm-app}. 
This is a more detailed version of~\cite[Lemma 8]{Schlag} in the setting at hand, see also~\cite[Lemma 1.4]{Wolff} and~\cite[Lemma 2.1]{Zahl}. Indeed, Lemma~\ref{lem:Schlag} is a restatement of~\cite[Lemma 5.1]{kenmki2017marstrandtype} for a family of parabolas; 
similar to loc.\ cit., the regularity of the measure $\rho$,~\eqref{eq:appendix-regularity-F}, is used as a replacement for the assumption in~\cite[Lemma 8]{Schlag} that the family has separated radii.   

\begin{lemma}\label{lem:Schlag}
Let the notation be as in Theorem~\ref{thm:proj-thm-app} with $b_1=1$. 
In particular, $E\subset B_\rfrak(0,1)$ and~\eqref{eq:appendix-regularity-F} is satisfied.  
For every $0<\vare\leq 0.01\alpha$, there exists 
$0<D\ll \vare^{-\star}\Upsilon^{\star}$ (implied constants are absolute) so that the following holds.
Let $\rhsc\geq \rhsc_0$. Then there exists a subset $\hat E=\hat E_b\subset E$
with $\#(E\setminus\hat E)\leq b^{\vare}\cdot(\#E)$ so that for every $w\in\hat E$, we have
\[
\Bigl|\Xi^{\rhsc}(w)\cap\Big\{q\in\bbr^2: m_{\rho}^{\rhsc}(q)\geq D\rhsc^{\alpha-7\vare}\Big\}\Bigr|\leq \rhsc^{2\vare/\alpha}|\Xi^{\rhsc}(w)|.
\]
\end{lemma}

The proof of this lemma is mutatis mutandis of the argument in~\cite[Lemma 5.1]{kenmki2017marstrandtype} where one replaces the use of~\cite[Lemma 1.4]{Wolff} with~\cite[Lemma 5.18]{Zahl}. We explicate the notation and the main steps for the convenience of the reader. 

Define $\Phi:\lf^2\times\lf^2\to \lf$ by 
\[
\Phi(x,y)=y_2+2x_1y_1+x_2y_1^2.
\]
Given $x_0\in\bbr^2$ and $r_0\in\bbr$, the set 
$\{y\in \bbr^2:\Phi(x_0,y)=r_0\}$ is a special example of a $\Phi$-circle in~\cite{Kolasa-Wolff,Zahl}.

Note that 
$\Xi(w)=\Bigl\{y\in \bbr^2:y_1\in[0,1], \Phi\Bigl((w_{11},w_{21}),y\Bigr)=w_{12}\Bigr\}$.
The family $\Xi$ satisfies the {\em cinematic} curvature conditions~\cite[Eq.~(1.5) and~(1.6)]{Zahl}. 
Indeed in the case at hand, these conditions follow from the following estimate  
\be\label{eq:cinematic curvature}
\tfrac{1}{3}\max\{|x_1|,|x_2|\}\leq|\tfrac{\partial \Phi}{\partial y_1}|+|\tfrac{\partial^2 \Phi}{\partial y_1^2}|\leq 3\max\{|x_1|, |x_2|\}; 
\ee 
we remark that when $\Phi(0,y)=y_2$, as is the case here,~\eqref{eq:cinematic curvature} (with $3$ replaced by a constant $C$) may be taken as the definition of the cinematic curvature conditions, see~\cite[Eq.~(21)]{Kolasa-Wolff}.

Let $w,w'\in B_\rfrak(0,1)$; define 
\[
\Delta(w-w')=\Bigl|\det(w-w')\Bigr|.
\]
The function $\Delta$ may be used to quantitatively measure the tangency of $\Xi(w)$ and $\Xi(w')$.
Our choice of $\Delta$ is different from $\Delta_{B_\rfrak(0,2)}$ which is defined in \cite[Def.\ 2.2]{Zahl}, 
however, in the case at hand $\Delta\asymp\Delta_{B_\rfrak(0,2)}$ --- indeed, the (reduced) discriminant of $\xi_r(w)-\xi_r(w')$ equals $-\det(w-w')$.   

By~\cite[Lemma 3.1]{Kolasa-Wolff}, for all $0<\delta<0.1$ and all $w,w'\in B_\rfrak(0,1)$, we have 
\begin{subequations}
\begin{align}
\label{eq:Schlag lemma Xi vare'}
&{\rm diam}\Bigl(\Xi^{\delta}(w)\cap \Xi^{\delta}(w')\Bigr)\ll \frac{\sqrt{\Delta(w-w')+\delta}}{\sqrt{\|w-w'\|+\delta}}\\
\label{eq:Schlag lemma Xi vare}
&|\Xi^{\delta}(w)\cap \Xi^{\delta}(w')|\ll \frac{\delta^2}{\sqrt{(\|w-w'\|+\delta)(\Delta(w-w')+\delta)}},
\end{align}
\end{subequations}
here and in the remaining parts of the argument, the implied constants are absolute unless otherwise is stated explicitly.

Let $\white,\black\subset B_\rfrak(0,1)$. We say $(\white,\black)$ is $t$-bipartite if
\be\label{def:bipart}
\max\{{\rm diam}(\white), {\rm diam}(\black)\}\leq t\leq d(\white,\black).
\ee

Let $0<\delta\leq t\leq 1$. A $(\delta,t)$-rectangle $R\subset \lf^2$ is a $\delta$-neighborhood 
of a piece of a parabola $\Xi(w)$, $w\in B_\rfrak(0,1)$, with length $\sqrt{\delta/t}$. We say that two $(\delta,t)$-rectangles are $C$-comparable 
if there is a $(C\delta,t)$-rectangle which contains both of them. Otherwise, they are $C$-incomparable. 
Let $w\in B_\rfrak(0,1)$, the parabola $\Xi(w)$ is $C$-tangent to a $(\delta, t)$-rectangle $R$, if $\Xi^{C\delta}(w)$ contains $R$. Finally, fixing some large absolute constant $\hat C \geq 1$, we say that two rectangles 
are comparable, if they are $\hat C$-comparable. Similarly, $\Xi(w)$ is said to be tangent to a
rectangle $R$ if $\Xi(w)$ is $\hat C$-tangent to $R$.

Let $0<\delta\leq t\leq 1$, and let $(\white,\black)$ be $t$-bipartite. Let $R$ be a $(\delta,t)$-rectangle. Put $\white_R=\{w\in\white: \Xi(w)$ is tangent to $R\}$; define $\black_R$ analogously. We say $R$ is of type $(\geq\mu,\geq\nu)$ with respect to $\rho,\white$, and $\black$ if 
\[
\rho(\white_R)\geq \mu\quad\text{and} \quad \rho(\black_R)\geq\nu.
\]
We say $R$ is of type $(\mu,\nu)$ if $\mu\leq \rho(\white_R)<2\mu$ and $\nu\leq \rho(\black_R)< 2\nu$. 

The following is an analogue of \cite[Lemma 1.4]{Wolff} tailored to our setting here; see also \cite[Lemma 5.18]{Zahl} and \cite[Lemma 4.4]{kenmki2017marstrandtype}. 

\begin{lemma}\label{lem:Wolff}
Let $0<\delta\leq t\leq 1$, and let $(\white,\black)$ be $t$-bipartite. Let $\vare>0$. Then the number of pairwise incomparable $(\delta,t)$-rectangles of type $(\geq\mu,\geq\nu)$ with respect to $\rho,\white$, and $\black$ is at most 
\[
D_\vare (\mu\nu\delta)^{-\vare}\biggl(\biggl(\tfrac{\rho(\white)\rho(\black)}{\mu\nu}\biggr)^{3/4}+\tfrac{\rho(\white)}{\mu}+\tfrac{\rho(\black)}{\nu}\biggr)
\]
where $D_\vare\ll\vare^{-\star}$ and the implied constants are absolute.
\end{lemma}

\begin{proof}
Replacing the use of~\cite[Lemma 1.4]{Wolff} with~\cite[Lemma 5.18]{Zahl}, the same proof as in~\cite[Lemma 4.4]{kenmki2017marstrandtype} applies here. The argument is standard: given $(\white,\black)$ and a collection $\mathcal R$ of incomparable $(\delta,t)$-rectangles, one uses a dyadic decomposition argument to find $i,j\in\bbn$ with 
\[
\text{$2^i/i^2\leq\delta^{-3}\mu^{-1}\quad$ and $\quad2^j/j^2\leq\delta^{-3}\nu^{-1}$,}
\]
a subset $\mathcal R'\subset\mathcal R$ with $\#\mathcal R'\gg \vare^{-\star}(\#\mathcal R)\delta^{6\vare}\mu^\vare\nu^{\vare}$, and a $t$-bipartite $(\white',\black')$ where $\white',\black'\subset B_\rfrak(0,1)$ are $\delta$-separated with $\#\white'\ll 2^i\rho(\white)$ and $\#\black'\ll 2^j\rho(\black)$, so that every $R\in\mathcal R'$ is of type 
\[
(\geq D'_\vare 2^i\mu^{1+\vare}\delta^{3\vare},\geq D'_\vare2^j\nu^{1+\vare}\delta^{3\vare})
\]
with respect to the counting measure, $\white'$, and $\black'$ for some $D'_\vare\ll\vare^{-\star}$. One then applies~\cite[Lemma 5.18]{Zahl} to $(\white',\black')$ and $\mathcal R'$ and obtains a bound for $\#\mathcal R'$ which implies the desired bound for $\#\mathcal R$. We note that the definition of a $t$-bipartite family in~\cite{Zahl} requires the radii are $\delta$-separated,~\cite[Def.\ 2.3]{Zahl}; this assumption however is not used in the proof of~\cite[Lemma 5.18]{Zahl}. Indeed as in~\cite[Lemma 1.4]{Wolff}, one only needs $\delta$-separation is the parameter space, i.e.\ $\|w-w'\|\geq \delta$ in the case at hand.     

The final estimate $D_\vare\ll\vare^{-\star}$ follows from $D'_\vare\ll\vare^{-\star}$ and the fact that the implied constant in~\cite[Lemma 5.18]{Zahl} is $\ll \vare^{-\star}$. This follows from the proof of~\cite[Lemma 5.18]{Zahl}, see in particular~\cite[pp.~1252--1253]{Wolff}.
\end{proof}

\begin{proof}[Proof of Lemma~\ref{lem:Schlag}]
Throughout the argument, $D$ will be assumed to be a large constant which is allowed to depend (polynomially) on $1/\vare$ and $\Upsilon$. 

Let $b\geq b_0$ be the largest dyadic number where the lemma fails; taking $D$ large enough, 
we assume that $b$ is small compared to absolute constants whenever necessary. 
Let $A=(Db^{-3\vare})^{1/\alpha}$ and $\lambda=\rhsc^{2\vare/\alpha}$.
By the choice of $b$, there exists $\mu\geq D\rhsc^{\alpha-7\vare}=A^\alpha\lambda^{-2\alpha}b^\alpha$ and 
a subset $E'\subset E$ with $\# E'> \rhsc^{\vare}\cdot(\#E)=D^{1/3}A^{-\alpha/3}\cdot(\#E)$ so that for all $w\in E'$, we have 
\[
\Bigl|\Xi^{\rhsc}(w)\cap\Big\{q\in\bbr^2: m_{\rho}^{\rhsc}(q)\geq \mu\Big\}\Bigr|\geq \lambda|\Xi^{\rhsc}(w)|.
\]

For every $w\in\rfrak$ and dyadic numbers $t,\delta\in (b,1]$, define 
\[
E_{\delta,t}(w)=\biggl\{w'\in E: \Xi^{b}(w)\cap\Xi^{b}(w')\neq \emptyset, \begin{array}{l}t\leq \|w-w'\|< 2t,\\ 
\delta\leq \Delta(w-w')<2\delta\end{array}\biggr\}.
\]
Define $E_{b,t}(w)$ similarly, except in this case no lower bound is assumed for $\Delta$, that is, we only assume $\Delta(w-w')<2b$. 

For every $F\subset E$, define $m^{{\bigcdot}}_{\rho}(q|F)=\rho\Bigl(\{w'\in F: q\in \Xi^{\,\bigcdot}(w')\}\Bigr)$. 
Replacing the use of~\cite[Lemma 3.6]{kenmki2017marstrandtype} with~\eqref{eq:Schlag lemma Xi vare'} and~\eqref{eq:Schlag lemma Xi vare}, one may argue as in the proof of~\cite[Eq.~(5.4)]{kenmki2017marstrandtype} and conclude the following. 
There exist absolute constants $C, C_1\geq 1$, $\bar E\subset E'$ with $\#\bar E\geq C^{-1}|\log b|^{-C}\cdot(\#E')$, and 
some dyadic number $n\in\{1,\ldots, \delta/b\}$, so that if we put
\be\label{eq:Schlag lemma def lambda A}
\lambda_\delta=|\log b|^{-C}\cdot\frac{\lambda \delta}{Cn b},\qquad A_\delta=C|\log b|^{C}\cdot\frac{A\delta}{n b},
\ee  
and $\mu_\delta=|\log b|^{-C}\cdot\frac{n \mu}{C}$, then for all $w\in \bar E$ we have  
\be\label{eq: Schlag lemma vare}
|\Xi^{\delta}(w)\cap\{q\in\bbr^2: m^{C_1\delta}_{\rho}(q|E_{\delta,t}(w))\geq \mu_\delta\}|\geq 2\lambda_\delta|\Xi^{\delta}(w)|,
\ee 
see~\cite[Eq.~(5.12)]{kenmki2017marstrandtype}. 
Note also that $\mu_\delta\gg |\log b|^{-\star}A_\delta^\alpha\lambda_\delta^{-2\alpha}\delta^{\alpha}$.

Fix a large dyadic number $N\geq 2$, in particular, $N\delta\geq 2b$. Now~\eqref{eq: Schlag lemma vare} and the inductive hypothesis (recall the choice of $b$), imply that
there exists a subset $\bar E'\subset \bar E$ with $\#\bar E'\gg \#\bar E$ so that for all $w\in\bar E'$, we have 
\begin{multline}\label{eq: Schlag lemma vare'}
\Bigl|\Xi^{\delta}(w)\cap\Bigl\{q\in\bbr^2: \mu_\delta\leq m^{C_1\delta}_{\rho}(q|E_{\delta,t}(w))\leq m^{N\delta}_\rho(q)\leq M_\delta\Bigr\}\Bigr|\\\geq \lambda_\delta\Bigl|\Xi^{\delta}(w)\Bigr|
\end{multline}
where $M_\delta=A_\delta^\alpha (\lambda_\delta/ CN)^{-2\alpha}\delta^{\alpha}\ll |\log b|^{\star} \mu_\delta$, see~\cite[Eq.~(5.14)]{kenmki2017marstrandtype}.   

Let $\{B_\rfrak(w_i,0.1t)\}$ be a covering of $\bar E'$ chosen so that $\{B_\rfrak(w_i,2.1t)\}$ has bounded multiplicity. Replacing $\bar E'$ with a subset whose $\rho$ measure is $\geq 0.5\rho(\bar E')$, we assume that $\rho(B_\rfrak(w_{i},0.1t)\cap \bar E')\gg t^3\rho(\bar E')$ for all $w_i\in\bar E'$.

Let $i_0$ be so that $\rho(B_\rfrak(w_{i_0},0.1t)\cap \bar E')/\rho(B_\rfrak(w_{i_0},2.1t))$ is maximized.
Put $\mathcal W':=B_\rfrak(w_{i_0},0.1t)\cap \bar E'$ and $\mathcal B:=B_\rfrak(w_{i_0},2.1t)\setminus B_\rfrak(w_{i_0},0.9t)$. 

Replacing $\mathcal W'$ by a subset $\mathcal W\subset\mathcal W'$ with $\rho(\mathcal W)\geq 0.5\rho(\mathcal W')$, we may assume that for all $z\in \mathcal W$, there is a dyadic cube $Q(z)$ of side-length $\delta$ which contains $z$ and $\rho(Q(z)\cap \mathcal W)\gg (\delta/t)^3\rho(\mathcal W)\gg |\log b|^{-\star}A^{-\alpha/3}\delta^3$. 
Note also that since the covering $\{B_\rfrak(w_{i_0},2.1t)\}$ has bounded multiplicity, we have 
\[
\rho(\mathcal W)\geq 0.5\rho(\mathcal W')\gg |\log b|^{-\star}A^{-\alpha/3}\rho(\mathcal B).
\]

By the definition, $(\mathcal W,\mathcal B)$ is $t$-bipartite, see~\eqref{def:bipart}. 
Moreover, for all $w\in\mathcal W$, we have $E_{\delta, t}(w)\subset\mathcal B$. Hence,   
\be\label{eq:w in W E(w) in B}
m_\rho^{C_1\delta}(q|E_{\delta, t}(w)\cap\mathcal B)=m_\rho^{C_1\delta}(q|E_{\delta, t}(w)),
\ee
for all $w\in\mathcal W$ and $q\in\mathbb R^2$.
We conclude from~\eqref{eq:w in W E(w) in B},~\eqref{eq: Schlag lemma vare'}, and~\eqref{eq:appendix-regularity-F} that
\[
|\log b|^{-\star}A_\delta^\alpha\lambda_\delta^{-2\alpha}\delta^{\alpha}\ll  \mu_\delta\leq m_\rho^{C_1\delta}(q|E_{\delta, t}(w)\cap\mathcal B)\leq \rho(\mathcal B)\ll t^\alpha;
\]
therefore, $\delta$ is much smaller than $t$ if $D$ is large enough, see~\eqref{eq:Schlag lemma def lambda A} and recall that $A=(Db^{-3\vare})^{1/\alpha}$ and $0<\lambda_\delta\leq 1$.

Since $\mathcal W\subset \bar{E}'$,~\eqref{eq: Schlag lemma vare'} and~\eqref{eq:w in W E(w) in B} imply that for all $w\in\mathcal W$, we have  
\begin{align}
\notag\Bigl|\Xi^{\delta}(w)\cap\Bigl\{q\in\bbr^2: \mu_\delta\leq m^{C_1\delta}_{\rho}(q|E_{\delta,t}(w)\cap \mathcal B)\leq m^{N\delta}_\rho(q)&\leq M_\delta\Bigr\}\Bigr|\\
\label{eq: Schlag lemma vare''}&\geq \lambda_\delta\Bigl|\Xi^{\delta}(w)\Bigr|.
\end{align}

Assuming $N$ is large enough, depending on $C_1$,~\eqref{eq: Schlag lemma vare''} implies that every
$w\in\white$ supplies $\gg \lambda_\delta\sqrt{t/\delta}$ incomparable $(\delta,t)$-rectangles each of which is
$N/2$-tangent to $\Xi(w)$ and has type $\geq \mu_\delta$ with respect to $\black$ where the type refers to $N$-tangency. 
From this, we conclude that there are 
\[
\gg |\log b|^{-\star}\rho(\white)\lambda_\delta\sqrt{t/\delta}/\nu_\delta\]
incomparable $(\delta,t)$-rectangles of type $(\geq \nu_\delta, \geq \mu_\delta)$ with respect to $\rho$, $\white$, and $\black$ where
$b^4\leq \nu_\delta\leq M_\delta$ is a dyadic number and type refers to $N$-tangency. Comparing this bound with the bound given by Lemma~\ref{lem:Wolff} yields a contradiction and finishes the proof, see~\cite[pp.\ 20--21]{kenmki2017marstrandtype}.

The assertion $D\ll \vare^{-\star}\Upsilon^{-\star}$ follows from the above outline, together with the fact $D_\vare$ in Lemma~\ref{lem:Wolff} is $\ll\vare^{-\star}$. 
\end{proof}

We now turn to the proof of Theorem~\ref{thm:proj-thm-app}. The argument is a slight modification of the proof of~\cite[Thm.~7.2]{kenmki2017marstrandtype}. 

\begin{proof}[Proof of Theorem~\ref{thm:proj-thm-app}]
Assume that the conclusion of the theorem fails for some $C$. That is, there exists a subset 
$\bar J\subset [0,1]$ with $|\bar J|> C\rhsc^\vare$ so that for all $r\in \bar J$ we have 
\be\label{eq:bad-L-rho}
\rho(E'_r)\geq C\rhsc^{\vare}
\ee
where $E'_r=\Bigl\{w\in E: m^{\rhsc}_\rho\Bigl((r,\xi_r(w))\Bigr)> C\rhsc^{\alpha-7\vare}\Bigr\}$. 
We will get a contradiction if $C$ is large enough. 

Let $\hat E$ be as in Lemma~\ref{lem:Schlag} applied with $8b$, then $\rho(\hat E)\geq 1-(8\rhsc)^{\vare}$. 
This and~\eqref{eq:bad-L-rho} now imply that for every $r\in \bar J$, we have $\rho(\hat E\cap E'_r)\geq C\rhsc^{\vare}/2$ so long as 
$C\geq16$. 

We conclude that
\begin{align*}
0.5C^2\rhsc^{2\vare}&\leq \int_{\bar J}\rho(\hat E\cap E'_r)\diff\! r\\
\notag&\leq \int_{\hat E}|\{r: m^{\rhsc}_\rho(r,\xi_r(w))> C\rhsc^{\alpha-7\vare}\}|\diff\!\rho.
\end{align*}
Therefore, there exists some $w_0\in \hat E$ so that 
\be\label{eq:w-in-Fhat}
\Bigl|\Bigl\{r\in [0,1]: m^{\rhsc}_\rho\Bigl((r,\xi_r(w_0))\Bigr)> C\rhsc^{\alpha-7\vare}\Bigr\}\Bigr|\geq 0.5C^2\rhsc^{2\vare}.
\ee

For every $r\in[0,1]$, let $L_r:=\{(r, s): s\in\bbr\}$ be a vertical line, and
let $I\subset L_r$ be an interval of length $\rhsc$ containing $(r,\xi_r(w_0))$. Put
\[
I_{+,\rhsc}=\Big\{(q_1,q_2)\in [r-\rhsc,r+\rhsc]\times \bbr: \exists (r, s)\in I, |q_2-s|\leq \rhsc\Big\}.
\]
If $(q_1,q_2)\in I_{+,\rhsc}$, then $|q_1-r|\leq \rhsc$ and $|q_2-\xi_r(w_0)|\leq 2\rhsc$. Therefore, 
\[
|q_2-\xi_{q_1}(w_0)|\leq |q_2-\xi_r(w_0)|+|\xi_r(w_0)-\xi_{q_1}(w_0)|\leq 8\rhsc.
\]
We conclude that $(q_1,q_2)\in \Xi^{8\rhsc}(w_0)$.  
This and $m^{\rhsc}_\rho\Bigl((r,\xi_r(w_0))\Bigr)> C\rhsc^{\alpha-7\vare}$ imply that for every 
$q\in I_{+,\rhsc}$, we have 
\be\label{eq:I+delta-m}
m^{8\rhsc}_{\rho}(q)\geq \rho\Bigl(\{w'\in E: (r,\xi_{r}(w'))\in I\}\Bigr)\geq C\rhsc^{\alpha-7\vare}.
\ee

Combining~\eqref{eq:w-in-Fhat} and~\eqref{eq:I+delta-m}, we obtain that 
\begin{align*}
\Bigl|\Xi^{8\rhsc}(w_0)\cap\{q\in\bbr^2: m^{8\rhsc}_{\rho}(q)\geq C\rhsc^{\alpha-7\vare}\}\}\Bigr|&\gg C^2\rhsc^{1+2\vare}\\
\gg C^2\rhsc^{2\vare}|\Xi^{8\rhsc}(w_0)|&>\rhsc^{2\vare/\alpha}|\Xi^{8\rhsc}(w_0)|.
\end{align*}
where the implied constant is absolute,
and we assume $C$ is large enough so that the final estimate holds --- recall that $0<\alpha\leq 1$.

This contradicts the fact that $w_0\in \hat E$ and finishes the proof. 
\end{proof}

\begin{proof}[Proof of Theorem~\ref{thm:proj-thm}]
Fix some $\kappa$. We may assume $b$'s are dyadic numbers, in particular $b_i=2^{-\ell_i}$, for $i=0,1$. 
Let $\ell_2$ be so that 
\[
\sum_{\ell=\ell_2}^{\infty} C_\kappa 2^{-\kappa\ell}< 0.1\min\{|J|, 1\}.
\]
Let $J'=\bigcap_{\ell=\ell_2}^{\ell_0} J_{2^{-\ell}}$. Then the choice of $\ell_2$ and Theorem~\ref{thm:proj-thm-app} imply that $|J'|\geq 0.9|J|$.

For every $r\in J'$, let $E_r=\bigcap_{\ell=\ell_2}^{\ell_0}E_{2^{-\ell},r}$.
Then by Theorem~\ref{thm:proj-thm-app}, $\rho(E_r)\geq 0.9$.
Moreover, for all $w\in E_r$ and all $\ell_2\leq \ell\leq \ell_0$
we have
\[
\rho(\{w'\in E: |\xi_{r}(w')-\xi_r(w)|\leq 2^{-\ell}\})\leq C_\kappa 2^{(\alpha-7\kappa)(\ell_1-\ell)}.
\]

The above implies that Theorem~\ref{thm:proj-thm} holds true with $J'$ and $E_r$ if we increase $C_\kappa$ to account for all $\rhsc\geq2^{-\ell_2}$. 
\end{proof}

\section{Proof of Lemma~\ref{lem:max-ineq-proj-thm}}\label{sec:max-ineq-proj}

We will prove Lemma~\ref{lem:max-ineq-proj-thm} in this section. 
As was mentioned before, the proof is taken from \cite[Lemma 5.2]{BFLM}, see also~\cite{Bour-Proj}; 
we reproduce the argument to explicate the stated bounds on $b_1$.  

\begin{proof}[Proof of Lemma~\ref{lem:max-ineq-proj-thm}]
We identify $\rfrak$ with $\bbr^3$. 
By a dyadic cube we mean a cube 
\[
[\tfrac{n_1}{2^k},\tfrac{n_1+1}{2^k})\times[\tfrac{n_2}{2^k},\tfrac{n_2+1}{2^k})\times[\tfrac{n_3}{2^k},\tfrac{n_3+1}{2^k})
\]
for an integer $k\geq 0$ and $0\leq n_i<2^k$. 

Let $\rho$ denote the uniform measure on $F$. 
Let $b\geq (\#F)^{-(1+\vare)/\alpha}$ and $w\in\bbr^3$, then 
\begin{equation}\label{eq:rho-regular}
    \begin{aligned}
    b^{-\alpha}\rho\Big(B(w,b)\Big)&\leq \frac{1}{\#F}\biggl(b^{-\alpha}+\sum_{w'\in B(w,b), w'\neq w}\|w-w'\|^{-\alpha}\biggr)\\
&\leq  \frac{1}{\#F}\Big(b^{-\alpha}+D(\#F)^{(1+\vare)}\Big)
\\
&\leq (D+1)\cdot(\#F)^\vare.
    \end{aligned}
\end{equation}

We will absorb the constant $D$ using the notation $\gg$ and $\ll$ in what follows. Let $b_0=(\#F)^{-1}$. 
Using the Besicovitch covering lemma and the fact that $\rho$ is probability measure, we conclude from~\eqref{eq:rho-regular} that $F$ contains a subset $\hat F$ of $b_0$-separated points with 
\[
\#\hat F\gg b_0^{\vare-\alpha}
\]
where the implied constant is absolute. 

Arguing as in the proof \cite[Lemma 5.2]{BFLM}, see also~\cite{Bour-Proj}, 
with $\hat F$ and $\alpha-\vare$, there exists some $T$, depending on $\vare$, and a subset $F_1\subset \hat F$,
with 
\be\label{eq:delta-separated}
\#F_1\geq \hat Cb_0^{2\vare-\alpha}
\ee 
so that the following holds. Let $k_1=\lceil -\log_2(b_0)/T\rceil$, then 
there exist integers $R_1,\ldots, R_{k_1}$ with $1\leq R_\ell\leq 2^{3T}$ so that every $2^{-\ell T}$-cube which intersects $F_1$ 
contains exactly $R_{\ell+1}$, $2^{-(\ell+1)T}$-cubes which intersect $F_1$.  

Since each remaining $2^{-k_1T}$-cube contains exactly one point, we have
\be\label{maximal-ineq-1}
\sum_{\ell=1}^{k_1} \log_2R_\ell=\log_2(\#F_1)\geq (\alpha-2\vare)T(k_1-2)
\ee
where we assume $T$ is large enough to account for the constant $\hat C$. 

For every $k>\lfloor k_1\vare\rfloor=:k_0$, let 
\[
M_k=\min_{k< \ell\leq k_1}\frac{1}{\ell-k}\sum_{k+1}^\ell\log_2 R_i.
\] 
Let $k_2$ be the smallest integer so that $M_{k_2}\geq (\alpha-20\vare) T$
if such exists, else let $k_2=k_1$. We claim 
\be\label{eq:maximal-ineq-1-1}
\vare k_1\leq k_2\leq \tfrac{3-\alpha+5\vare}{3-\alpha+20\vare}k_1
\ee
The lower bound follows from the definition of $k_2$, we show the upper bound.
First note that if $k_2=k_0+1$, there is nothing to prove; suppose thus that $k_2>k_0+1$. 
Then for every $k_0< i<k_2$, there is some $i<i'\leq k_1$ so that $\sum_{\ell=i}^{i'}\log_2 R_\ell\leq (\alpha-s+\vare) T(i-i')$; thus 
there is $k_2\leq k\leq k_1$, so that
\[
\sum_{\ell=k_0+1}^k\log_2R_\ell\leq (\alpha-20\vare) T(k-k_0).
\]
This,~\eqref{maximal-ineq-1}, and the fact that $\log_2R_\ell\leq 3T$ for all $\ell$ imply that 
\begin{multline}\label{eq:maximal-ineq-2}
\notag 3Tk_0+(\alpha-20\vare) T(k-k_0)+3T(k_1-k)\geq 
3Tk_0+\sum_{\ell=k_0+1}^k\log_2R_\ell \quad +\\3T(k_1-k)
\geq \sum_{\ell=1}^{k_1} \log_2R_\ell\geq (\alpha-2\vare)T(k_1-2);
\end{multline}
we conclude that $k(3-\alpha+20\vare)\leq k_1\Bigl(3-\alpha+5\vare\Bigr)$. 
This finishes the proof of~\eqref{eq:maximal-ineq-1-1} as $k_2\leq k$.  

Let now $D$ be any $2^{-k_2T}$-cube which intersects $F_1$. Let $k_2< \ell\leq k_1$, and let 
$D'\subset D$ be a $2^{-\ell T}$-cube. Then 
\[
\#(D'\cap F_1)\leq \Bigl(\#(D\cap F_1)\Bigr)\cdot \prod_{i=k_2+1}^{\ell} R_i^{-1}.
\]
Since $\sum_{k_2}^{\ell}\log_2 R_i\geq (\alpha-20\vare) T(\ell-k_2)$, we conclude that 
\[
\tfrac{\#(B(w,b)\cap D\cap F_1)}{\#(D\cap F_1)}\leq C' \Bigl(b/2^{-Tk_2}\Bigr)^{\alpha-20\vare}
\]
for all $b\geq (\#F)^{-1}$ where $C'\ll\vare^{-\star}$ with absolute implied constants.  

Let $F'=D\cap F_1$, and let $w_0\in D\cap F_1$. The lemma holds with $w_0$, $b_1=2^{1-Tk_2}$, and $F'=D\cap F_1\subset B(w_0,b_1)$. 
\end{proof}

\bibliographystyle{amsplain}
\bibliography{papers}

\end{document}